\documentclass[11pt]{amsart}

\usepackage[colorlinks=true, pdfstartview=FitV, linkcolor=blue, citecolor=blue, urlcolor=blue, breaklinks=true]{hyperref}
\usepackage{amsmath,amsfonts,amssymb,amsthm,amscd,comment,paralist,etoolbox,mathtools,enumitem,mathdots}
\usepackage[usenames,dvipsnames]{xcolor}
\usepackage{mathrsfs} % to get \mathscr
\usepackage{tabu}
\usepackage{tensor} % to get \tensor command for left superscripts and subscripts
\usepackage{youngtab}

\usepackage[all]{xy}
\usepackage{tikz}
\usetikzlibrary{arrows}

\makeatletter
\def\namedlabel#1#2{\begingroup
    #2%
    \def\@currentlabel{#2}%
    \phantomsection\label{#1}\endgroup
}

\newcommand{\pushright}[1]{\ifmeasuring@#1\else\omit\hfill$\displaystyle#1$\fi\ignorespaces}
\newcommand{\pushleft}[1]{\ifmeasuring@#1\else\omit$\displaystyle#1$\hfill\fi\ignorespaces}
\makeatother

\newcommand\C{\mathbb{C}}
\newcommand\Z{\mathbb{Z}}
\newcommand\Q{\mathbb{Q}}
\newcommand\R{\mathbb{R}}
\newcommand\N{\mathbb{N}}

\newcommand\fsl{\mathfrak{sl}}

\newcommand\cA{\mathcal{A}}

\newcommand\cC{\mathcal{C}}

\newcommand\cM{\mathcal{M}}
\newcommand\cP{\mathcal{P}}
\newcommand\cV{\mathcal{V}}

\newcommand\sA{\mathscr{A}}
\newcommand\sC{\mathscr{C}}
\newcommand\sH{\mathscr{H}}
\newcommand\sM{\mathscr{M}}
\newcommand\sU{\mathscr{U}}

\newcommand\ba{\mathbf{a}}
\newcommand\bA{\mathbf{A}}
\newcommand\bB{\mathbf{B}}
\newcommand\bF{\mathbf{F}}
\newcommand\bc{\mathbf{c}}
\newcommand\bfr{\mathbf{r}}
\newcommand\bi{\mathbf{i}}
\newcommand\bj{\mathbf{j}}

\newcommand\bS{\mathbf{S}}
\newcommand\bT{\mathbf{T}}

\newcommand\bzero{\mathbf{0}}

\newcommand\tr{\mathrm{tr}}

\newcommand\rR{\mathrm{R}}
\newcommand\rL{\mathrm{L}}

\newcommand\se{\mathsf{e}}
\newcommand\sE{\mathsf{E}}
\newcommand\sF{\mathsf{F}}
\newcommand\sQ{\mathsf{Q}}

\newcommand{\md}{\textup{-mod}}
\newcommand{\op}{\textup{op}}

\newcommand{\trunc}{\textup{\normalfont tr}}

\newcommand\redcircle[1]{\filldraw[fill=white, draw=red] #1 circle (3pt)}

\newcommand\regionlabel[1]{$\scriptstyle{#1}$}
\newcommand\strandlabel[1]{$\scriptstyle{#1}$}

\usetikzlibrary{decorations.markings}
\usetikzlibrary{decorations.pathreplacing}
\usetikzlibrary{shapes,positioning}

\tikzset{anchorbase/.style={>=stealth,baseline={([yshift=-0.5ex]current bounding box.center)}}}
\tikzstyle directed=[postaction={decorate,decoration={markings,
    mark=at position #1 with {\arrow{>}}}}]
\tikzstyle rdirected=[postaction={decorate,decoration={markings,
    mark=at position #1 with {\arrow{<}}}}]

\DeclareMathOperator{\1Mor}{1Mor}
\DeclareMathOperator{\2Mor}{2Mor}
\DeclareMathOperator{\Mor}{Mor}
\DeclareMathOperator{\Hom}{Hom}
\DeclareMathOperator{\End}{End}

\DeclareMathOperator{\Span}{Span}

\DeclareMathOperator{\id}{id}
\DeclareMathOperator{\Sym}{Sym}

\DeclareMathOperator{\Ob}{Ob}

\DeclareMathOperator{\Kar}{Kar}

\newtheorem{theo}{Theorem}[section]
\newtheorem{prop}[theo]{Proposition}
\newtheorem{lem}[theo]{Lemma}
\newtheorem{cor}[theo]{Corollary}

\theoremstyle{definition}
\newtheorem{defin}[theo]{Definition}
\newtheorem{rem}[theo]{Remark}
\newtheorem{eg}[theo]{Example}

\numberwithin{equation}{section}
\allowdisplaybreaks

\newtoggle{details}
\newtoggle{detailsnote}

\toggletrue{detailsnote}   % For including note about details toggle (default is false)

\iftoggle{details}{%
  \newcommand{\details}[1]{
      \ \\
      {\color{OliveGreen}
        \textbf{Details:} #1
      }
      \ \\
  }
}{%
  \newcommand{\details}[1]{}
}

\begin{document}
%
%%%%%%%%%%%%%%%%%%%%%%%%%%%%%%%%%%%

\title[Truncations of categorified quantum groups and Heisenberg categories]{An equivalence between truncations of categorified quantum groups and Heisenberg categories}

\author{Hoel Queffelec}
\address{H.~Queffelec: Institut Montpelli\'erain Alexander Grothendieck, Universit\'e de Montpellier, CNRS, Montpellier, France}
\urladdr{\url{http://www.math.univ-montp2.fr/~queffelec/}}
\email{hoel.queffelec@umontpellier.fr}

\author{Alistair Savage}
\address{A.~Savage: Department of Mathematics and Statistics, University of Ottawa, Canada}
\urladdr{\url{http://alistairsavage.ca}}
\email{alistair.savage@uottawa.ca}

\author{Oded Yacobi}
\address{O.~Yacobi: School of Mathematics and Statistics, University of Sydney, Australia}
\urladdr{\url{http://www.maths.usyd.edu.au/u/oyacobi/}}
\email{oded.yacobi@sydney.edu.au}

\begin{abstract}
  We introduce a simple diagrammatic 2-category $\sA$ that categorifies the image of the Fock space representation of the Heisenberg algebra and the basic representation of $\fsl_\infty$.  We show that $\sA$ is equivalent to a truncation of the Khovanov--Lauda categorified quantum group $\sU$ of type $A_\infty$, and also to a truncation of Khovanov's Heisenberg 2-category $\sH$.  This equivalence is a categorification of the principal realization of the basic representation of $\fsl_\infty$.

  As a result of the categorical equivalences described above, certain actions of $\sH$ induce actions of $\sU$, and vice versa.  In particular, we obtain an explicit action of $\sU$ on representations of symmetric groups.  We also explicitly compute the Grothendieck group of the truncation of $\sH$.

  The 2-category $\sA$ can be viewed as a graphical calculus describing the functors of $i$-induction and $i$-restriction for symmetric groups, together with the natural transformations between their compositions.  The resulting computational tool is used to give simple diagrammatic proofs of (apparently new) representation theoretic identities.
\end{abstract}

%\date{\today}

%\prelim

\maketitle
\thispagestyle{empty}

\tableofcontents

%%%%%%%%%%%%%%%%%%%%%%%%%%%%%%%%%%%%%%%%%%%%%%%%%%%%%%%%%%%%%%%%%%%%
%
\section{Introduction}
%
%%%%%%%%%%%%%%%%%%%%%%%%%%%%%%%%%%%%%%%%%%%%%%%%%%%%%%%%%%%%%%%%%%%%

Affine Lie algebras play a key role in many areas of representation theory and mathematical physics.  One of their prominent features is that their highest-weight irreducible representations have explicit realizations.  In particular, constructions of the so-called basic representation involve deep mathematics from areas as diverse as algebraic combinatorics (symmetric functions), number theory (modular forms), and geometry (Hilbert schemes).

Two of the most well-studied realizations of the basic representation are the \emph{homogeneous} and \emph{principal} realizations (see, for example \cite[Ch.~14]{Kac90}).  The homogeneous realization in affine types ADE has been categorified in \cite{CL11}.  In the current paper we focus our attention on the principal realization in type $A_\infty$.  The infinite-dimensional Lie algebra $\fsl_\infty$ behaves in many ways like an affine Lie algebra, and in particular, it has a basic representation with a principal realization coming from a close connection to the infinite-rank Heisenberg algebra $H$.

The Heisenberg algebra $H$ has a natural representation on the space $\Sym$ of symmetric functions (with rational coefficients), called the Fock space representation.  The universal enveloping algebra $U = U(\fsl_\infty)$ also acts naturally on $\Sym$, yielding the basic representation.  So we have algebra homomorphisms
\begin{equation} \label{eq:principal-maps}
  H \xrightarrow{r_H} \End_\Q \Sym \xleftarrow{r_U} U.
\end{equation}
Consider the vector space decomposition
\[
  \Sym = \bigoplus_{\lambda \in \cP} \Q s_\lambda,
\]
where the sum is over all partitions $\cP$ and $s_\lambda$ denotes the Schur function corresponding to $\lambda$. Let $1_\lambda \colon \Sym \to \Q s_\lambda$ denote the natural projection.  While the images of the representations $r_H$ and $r_U$ are not equal, we have an equality of their idempotent modifications:
\begin{equation} \label{eq:principal-images}
  \bigoplus_{\lambda,\mu \in \cP} 1_\mu r_H(H) 1_\lambda
  = \bigoplus_{\lambda,\mu \in \cP} 1_\mu r_U(U) 1_\lambda.
\end{equation}
This observation is an $\fsl_\infty$ analogue of the fact that the basic representation of $\widehat{\fsl}_n$ remains irreducible when restricted to the principal Heisenberg subalgebra---a fact which is the crucial ingredient in the {principal realization of the basic representation.  We view \eqref{eq:principal-images} as an additive $\Q$-linear category $\cA$ whose set of objects is the free monoid $\N[\cP]$ on $\cP$ and with
\[
  \Mor_\cA(\lambda,\mu)
  = 1_\mu r_H(H) 1_\lambda
  = 1_\mu r_U(U) 1_\lambda
  = \Hom_\Q (\Q s_\lambda, \Q s_\mu).
\]

In \cite{Kho14}, Khovanov introduced a monoidal category, defined in terms of planar diagrams, whose Grothendieck group contains (and is conjecturally isomorphic to) the Heisenberg algebra $H$.  Khovanov's category has a natural 2-category analogue $\sH$.  On the other hand, in \cite{KL3}, Khovanov and Lauda introduced a 2-category, which we denote $\sU$, that categorifies quantum $\fsl_n$ and can naturally be generalized to the $\fsl_\infty$ case (see \cite{CL15}).  A related construction was also described by Rouquier in \cite{Rou2}.  These categorifications have led to an explosion of research activity, including generalizations, and applications to representation theory, geometry, and topology.  It is thus natural to seek a connection between the 2-categories $\sH$ and $\sU$ that categorifies the principal embedding relationship between $H$ and $U$ discussed above.  This is the goal of the current paper.

We define a 2-category $\sA$ whose 2-morphism spaces are given by planar diagrams modulo isotopy and local relations.  The local relations of $\sA$ are exceedingly simple and we show that $\sA$ categorifies $\cA$.  We then describe precise relationships between $\sA$ and the 2-categories $\sH$ and $\sU$.  Our first main result is that $\sA$ is equivalent to a degree zero piece of a truncation of the categorified quantum group $\sU$.  More precisely, recalling that the objects of $\sU$ are elements of the weight lattice of $\fsl_\infty$, we consider the truncation $\sU^\trunc$ of $\sU$ where we kill weights not appearing in the basic representation.  Specifically, we quotient the 2-morphism spaces by the identity 2-morphisms of the identity 1-morphisms of such weights.  (This type of truncation has appeared before in the categorification literature, for example, in \cite{MSV2,QR}.)  The resulting 2-morphism spaces of $\sU^\trunc$ are nonnegatively graded, and we show that the degree zero part $\sU_0$ of $\sU^\trunc$ is equivalent to the 2-category $\sA$ (Theorem~\ref{theo:sA-tsU-equivalence}).

Our next main result is that $\sA$ is also equivalent to a summand of an idempotent completion of a truncation of the Heisenberg 2-category $\sH$.  More precisely, recalling that the objects of $\sH$ are integers, we consider the truncation ${\sH^\trunc}'$ of $\sH$ obtained by killing objects corresponding to negative integers.  We then take an idempotent completion $\sH^\trunc$ of ${\sH^\trunc}'$, show that we have a natural decomposition $\sH^\trunc \cong \sH_\epsilon \bigoplus \sH_\delta$, and that $\sA$ is equivalent to the summand $\sH_\epsilon$ (Theorem~\ref{theo:sU-sHepsilon-equivalence}).  This summand can be obtained from $\sH^\trunc$ by imposing one extra local relation (namely, declaring a clockwise circle in a region labeled $n$ to be equal to $n$).  We note that the idempotent completion we consider in the above construction is larger than the one often appearing in the categorification literature since we complete with respect to both idempotent 1-morphisms and 2-morphisms (see Definition~\ref{def:idem-completion} and Remark~\ref{rem:idem-completion}).  As a result, the idempotent completion has more objects, with the object $n$ splitting into a direct sum of objects labeled by the partitions of $n$.

We thus have 2-functors
\[
  \sH \xrightarrow{\text{truncate}} \sH^\trunc \xrightarrow{\text{summand}} \sH_\epsilon \cong \sA
  \cong \sU_0 \xleftarrow{\text{summand}} \sU^\trunc \xleftarrow{\text{truncate}} \sU.
\]
that can be thought of as a categorification of \eqref{eq:principal-maps}.  The equivalence $\sH_\epsilon \cong \sU_0$ is a categorification of the isomorphism~\eqref{eq:principal-images} and yields a categorical analog of the principal realization of the basic representation of $\fsl_\infty$.  In particular, any action of $\sH$ factoring through $\sH^\trunc$ (which is true of any action categorifying the Fock space representation) induces an explicit action of $\sU$.  Conversely, any action of $\sU$ factoring through $\sU^\trunc$ (which is true of any action categorifying the basic representation) induces an explicit action of $\sH$.  See Section~\ref{subsec:induced-actions}.

In \cite{Kho14}, Khovanov described an action of his Heisenberg category on modules for symmetric groups.  This naturally induces an action of the 2-category $\sH$ factoring through $\sH^\trunc$.  Applying the categorical principal realization to this action we obtain an explicit action of the Khovanov--Lauda categorified quantum group $\sU$ on modules for symmetric groups, relating our work to \cite{BK09,BK09b}.  See Section~\ref{sec:action-cat-quantum-group}.

By computations originally due to Chuang and Rouquier in \cite[\S7.1]{CR08}, one can easily deduce that there is a categorical action of $\fsl_\infty$ on modules for symmetric groups.  This action is constructed using $i$-induction and $i$-restriction functors, and thus is closely related to Khovanov's categorical Heisenberg action.  The equivalence $\sH_\epsilon \cong \sU_0$ gives the precise diagrammatic connection between these actions on the level of 2 categories.  In particular, the 2-category $\sA$ yields a graphical calculus for describing $i$-induction and $i$-restriction functors, together with the natural transformations between them  (see Proposition~\ref{prop:sU-sM-equivalence}).  This provides a computational tool for proving identities about the representation theory of the symmetric groups.  See Section~\ref{subsec:diagrammatic-computation} for some examples of identities that, to the best of our knowledge, are new.

One of the most important open questions about Khovanov's Heisenberg category is the conjecture that it categorifies the Heisenberg algebra (see \cite[Conj.~1]{Kho14}).  In the framework of 2-categories, this conjecture is the statement that the Grothendieck group of $\sH$ is isomorphic to $\bigoplus_{m \in \Z} H$.  (The presence of the infinite sum here arises from the fact that, in a certain sense, the 2-category $\sH$ contains countably many copies of the monoidal Heisenberg category defined in \cite{Kho14}.)  We prove the analog of Khovanov's conjecture for the truncated category $\sH^\trunc$, namely that the Grothendieck group of $\sH^\trunc$ is isomorphic to $\bigoplus_{m \in \N} \cA$.  See Corollary~\ref{cor:K(Htrunc)}.

We now give an overview of the contents of the paper.  In Section~\ref{sec:prelims} we recall some basic facts about the basic representation and define the category $\cA$.  We also set some category theoretic notation and conventions.  In Section~\ref{sec:Sn-modules} we recall some facts about modules for symmetric groups, discuss eigenspace decompositions with respect to Jucys--Murphy elements, and prove some combinatorial identities that will be used elsewhere in the paper.  Then, in Section~\ref{sec:sA-def}, we introduce the 2-category $\sA$ and show that it is equivalent to $\sU_0$.  We also prove some results about the structure of $\sA$ and prove that it categorifies $\cA$.  We turn our attention to the Heisenberg 2-category in Section~\ref{sec:sH-def}.  In particular, we introduce the truncated Heisenberg 2-category $\sH^\trunc$, describe the decomposition $\sH^\trunc \cong \sH_\epsilon \oplus \sH_\delta$, and prove that $\sH_\epsilon$ is equivalent to $\sA$.  In Section~\ref{sec:action} we discuss how our results yield categorical Heisenberg actions from categorified quantum group actions and vice versa.  In particular, we describe an explicit action of the Khovanov--Lauda 2-category on modules for symmetric groups.  Finally, in Section~\ref{sec:applications} we give an application of our results to diagrammatic computation and discuss some possible directions for further research.

\iftoggle{detailsnote}{
\medskip

\paragraph{\textbf{Note on the arXiv version}} For the interested reader, the tex file of the \href{https://arxiv.org/abs/1701.08654}{arXiv version} of this paper includes hidden details of some straightforward computations and arguments that are omitted in the pdf file.  These details can be displayed by switching the \texttt{details} toggle to true in the tex file and recompiling.
}{}

%%%%%%%%%%%%%%%%%%%%%%%%%%%%%%
\subsection*{Acknowledgements}
%%%%%%%%%%%%%%%%%%%%%%%%%%%%%%

The authors would like to thank C.~Bonnaf\'e, M.~Khovanov, A.~Lauda, A.~Licata, A.~Molev, E.~Wagner, and M.~Zabrocki for helpful conversations.  H.Q.\ was supported by a Discovery Project from the Australian Research Council and a grant from the Universit\'e de Montpellier. A.S.\ was supported by a Discovery Grant from the Natural Sciences and Engineering Research Council of Canada.  O.Y.\ was supported by a Discovery Early Career Research Award from the Australian Research Council.

%%%%%%%%%%%%%%%%%%%%%%%%%%%%%%%%%%%%%%%%
%
\section{Algebraic preliminaries} \label{sec:prelims}
%
%%%%%%%%%%%%%%%%%%%%%%%%%%%%%%%%%%%%%%%%

%-----------------------------------------------------
\subsection{Bosonic Fock space and the category $\cA$} \label{subsec:Fock-space}
%-----------------------------------------------------

Let $\cP$ denote the set of partitions and write $\lambda \vdash n$ to denote that $\lambda=(\lambda_1, \lambda_2, \dotsc)$, $\lambda_1 \geq \lambda_2 \geq \dotsb$, is a partition of $n \in \N$.  Let $\Sym$ be the algebra of symmetric functions with rational coefficients.  Then we have
\[
  \Sym = \bigoplus_{\lambda \in \cP} \Q s_\lambda,
\]
where $s_\lambda$ denotes the Schur function corresponding to the partition $\lambda$.  For $\lambda \in \cP$, we let $1_\lambda \colon \Sym \to \Q s_\lambda$ denote the corresponding projection.

Let $\N[\cP]$ be the free monoid on the set of partitions.  Define $\cA$ to be the additive $\Q$-linear category whose set of objects is $\N[\cP]$, where we denote the zero object by $\bzero$.  The morphisms between generating objects are
\[
  \Mor_{\cA}(\lambda,\mu)
  = 1_\mu (\End_\Q \Sym) 1_\lambda
  = \Hom_\Q(\Q s_\lambda, \Q s_\mu),\quad \lambda,\mu \in \cP.
\]
If $\cV$ denotes the category of finite-dimensional $\Q$-vector spaces, then we have an equivalence of categories
\begin{equation} \label{eq:rhoA-def}
  \bfr \colon \cA \to \cV,\quad \lambda \mapsto \Q s_\lambda.
\end{equation}

Let $\langle \cdot, \cdot \rangle$ be the inner product on $\Sym$ under which the Schur functions are orthonormal.  For $f \in \Sym$, let $f^*$ denote the operator on $\Sym$ adjoint to multiplication by $f$:
\[
  \langle f^*(g), h \rangle = \langle g, fh \rangle \quad \text{for all } f,g,h \in \Sym.
\]
The \emph{Heisenberg algebra} $H$ is the subalgebra of $\End_\Q \Sym$ generated by the operators $f$ and $f^*$, $f \in \Sym$.  The tautological action of $H$ on $\Sym$ is called the \emph{(bosonic) Fock space representation}.

For $\lambda, \mu \in \cP$, we have
\[
  1_\mu H 1_\lambda = \Hom_\Q(\Q s_\lambda, \Q s_\mu) = \Mor_\cA(\lambda,\mu),
\]
where the first equality follows from the fact that $s_\mu s_\lambda^*1_\lambda$ is the map $s_\nu \mapsto \delta_{\lambda,\nu}s_\mu$.  Thus, $\cA$ may be viewed as an idempotent modification of $H$.

%-----------------------------------------------------
\subsection{The basic representation} \label{subsec:basic-rep}
%-----------------------------------------------------

Let $\fsl_\infty$ denote the Lie algebra of all trace zero infinite matrices $a=(a_{ij})_{i,j \in \Z}$ with rational entries such that the number of nonzero $a_{ij}$ is finite, with the usual commutator bracket.  Set
\[
  e_i=E_{i,i+1},\quad f_i=E_{i+1,i},\quad h_i=[e_i,f_i] = E_{i,i} - E_{i+1,i+1},
\]
where $E_{i,j}$ is the matrix whose $(i,j)$-entry is equal to one and all other entries are zero.  Let $U=U(\fsl_\infty)$ denote the universal enveloping algebra of $\fsl_\infty$.

To a partition $\lambda=(\lambda_1, \dotsc, \lambda_n)$, we associate the Young diagram with rows numbered from top to bottom, columns numbered left to right, and which has $\lambda_1$ boxes in the first row, $\lambda_2$ boxes in the second row, etc.  A box in row $k$ and column $\ell$ has \emph{content} $\ell-k\in \Z$.  A Young diagram will be said to have an \emph{addable} $i$-box if one can add to it a box of content $i$ and get a Young diagram. Similarly, a Young diagram has a \emph{removable} $i$-box if there is a box of content $i$ that can be removed yielding another Young diagram.  If $\lambda \vdash n$ has an addable $i$-box we let $\lambda \boxplus i$ be the partition of $n+1$ obtained from $\lambda$ by adding the box of content $i$, and similarly define $\lambda \boxminus i$.

\begin{eg}
  Let $\lambda=(3,2)\vdash 5$.  Then we have
  \[
    \lambda = {\tiny\yng(3,2)}, \qquad
    \lambda\boxplus 3 = {\tiny\yng(4,2)}, \qquad
    \lambda\boxminus 0 ={\tiny\yng(3,1)}.
  \]
\end{eg}

For $\lambda\in \cP$, define
\begin{gather*}
  B^+(\lambda) = \{i \mid \lambda \text{ has an addable $i$-box}\}
  \quad \text{and} \quad
  B^-(\lambda) = \{i \mid \lambda \text{ has a removable $i$-box}\}.
\end{gather*}
Note that, for all $\lambda \in \cP$, we have $B^+(\lambda) \cap B^-(\lambda) = \varnothing$.
If $i\notin B^+(\lambda)$ (respectively $i\notin B^-(\lambda)$), then we consider $\lambda \boxplus i = \bzero$ (respectively $\lambda \boxminus i=\bzero$) when viewing partitions as objects in $\cA$.

Consider the action of $U$ on $\Sym$ given by
\begin{equation} \label{eq:sl-action-on-Sym}
  e_i\cdot s_\lambda = s_{\lambda \boxminus i}, \quad
  f_i\cdot s_\lambda = s_{\lambda \boxplus i},
\end{equation}
where, by convention, $s_\bzero = 0$.  This defines an irreducible representation of $U$ on $\Sym$ known as the \emph{basic representation}.  In fact, one can write explicit expressions for the action of the generators $e_i$ and $f_i$ in terms of the action of the Heisenberg algebra $H$ on $\Sym$.  This construction is known as the \emph{principal realization}.  We refer the reader to \cite[\S\S 14.9--14.10]{Kac90} for details.  The element $s_\lambda$ spans the weight space of weight
\begin{equation} \label{eq:omega_lambda-def}
  \omega_\lambda := \Lambda_0 - \sum_{i \in C(\lambda)} \alpha_i,
\end{equation}
where the sum is over the multiset $C(\lambda)$ of contents of the boxes of $\lambda$, $\Lambda_0$ is the zeroth fundamental weight, and $\alpha_i$ is the $i$-th simple root.  In particular, the map
\begin{equation} \label{eq:partions-label-basic-rep-weights}
  \lambda \mapsto \omega_\lambda
\end{equation}
is a bijection between $\cP$ and the set of weights of the basic representation.

%----------------------------------------------
\subsection{A Kac--Moody presentation of $\cA$}
%----------------------------------------------

Let $\hat U$ denote the image of $U$ in $\End_\Q \Sym$ under the basic representation described in Section~\ref{subsec:basic-rep}.  Then, for $\lambda, \mu \in \cP$, we have
\[
  1_\mu \hat U 1_\lambda = \Hom_\Q(\Q s_\lambda, \Q s_\mu) = \Mor_\cA(\lambda,\mu).
\]
This observation allows us to deduce a Kac--Moody-type presentation of $\cA$.  Define morphisms
\begin{gather*}
  e_i 1_\lambda \in \Mor_\cA (\lambda,\lambda\boxminus i), \quad s_\lambda \mapsto s_{\lambda \boxminus i}, \\
  f_i 1_\lambda\in \Mor_\cA (\lambda,\lambda\boxplus i),\quad s_\lambda \mapsto s_{\lambda \boxplus i},
\end{gather*}
for $i \in \Z$, $\lambda \in \cP$.  Since Young's lattice is connected, these morphisms clearly generate all morphisms in $\cA$.

\begin{prop} \label{prop:dotU-presentation}
  The morphisms in $\cA$ are generated by $e_i 1_\lambda$, $f_i 1_\lambda$, for $i \in \Z$, $\lambda \in \cP$, subject only to the relations
  \begin{gather}
    e_i e_j 1_\lambda=e_je_i1_\lambda, \quad \text{if } |i-j|>1 \label{rel:eiej}, \\
    f_i f_j 1_\lambda =f_jf_i1_\lambda, \quad \text{if } |i-j|>1  \label{rel:fifj}, \\
    e_i f_j 1_\lambda = f_j e_i 1_\lambda, \quad \text{if } i\neq j \label{rel:eifj}, \\
    e_i f_i 1_\lambda = 1_\lambda, \quad \text{if } i\in B^+(\lambda) \label{rel:eifi}, \\
    f_i e_i 1_\lambda = 1_\lambda, \quad \text{if } i\in B^-(\lambda) \label{rel:fiei}.
  \end{gather}
\end{prop}

\begin{proof}
  Let $\mathcal{C}$ be the category with objects $\N[\cP]$ and morphisms given by the presentation in the statement of the proposition.  Since the relations \eqref{rel:eiej}--\eqref{rel:fiei} are immediate in $\cA$, we have a full and essentially surjective functor $\mathcal{C} \to \cA$.  Therefore  it suffices to show that $\dim \Mor_\mathcal{C}(\lambda,\mu) \le 1$ for all $\lambda,\mu \in \cP$.

  In fact, we will prove that, for $\lambda,\mu \in \cP$, $\Mor_\mathcal{C}(\lambda,\mu)$ is spanned by a single morphism of the form
  \begin{equation} \label{eq:dotU0-hom-element}
    f_{i_1} f_{i_2} \dotsm f_{i_k} e_{j_1} e_{j_2} \dotsm e_{j_\ell} 1_\lambda,
  \end{equation}
  where $\{i_1,\dotsc,i_k\} \cap \{j_1,\dotsc,j_\ell\} = \varnothing$ and $\mu = \lambda \boxminus j_\ell \dotsb \boxminus j_2 \boxminus j_1 \boxplus i_k \boxplus \dotsb \boxplus i_2 \boxplus i_1$.  This follows from the following three statements:
  \begin{asparaenum}
    \item \label{lem-claim:U0-span} Morphisms of the form \eqref{eq:dotU0-hom-element} span $\Mor_\mathcal{C}(\lambda,\mu)$.

    \item \label{lem-claim:U0-f-order} Suppose $\lambda \in \cP$ and $j_1,\dotsc,j_\ell,i_1,\dotsc,i_\ell \in \Z$ such that
      \begin{equation} \label{eq:U0-box-add-order}
        \lambda \boxplus i_\ell \boxplus \dotsb \boxplus i_2 \boxplus i_1
        = \lambda \boxplus j_\ell \boxplus \dotsb \boxplus j_2 \boxplus j_1
      \end{equation}
      are nonzero.  Then
      \[
        f_{i_1} f_{i_2} \dotsm f_{i_\ell} 1_\lambda
        = f_{j_1} f_{j_2} \dotsm f_{j_\ell} 1_\lambda.
      \]

    \item \label{lem-claim:U0-e-order} Suppose $\lambda \in \cP$ and $j_1,\dotsc,j_\ell,i_1,\dotsc,i_\ell \in \Z$ such that
      \[
        \lambda \boxminus i_\ell \boxminus \dotsb \boxminus i_2 \boxminus i_1
        = \lambda \boxminus j_\ell \boxminus \dotsb \boxminus j_2 \boxminus j_1
      \]
      are nonzero.  Then
      \[
        e_{i_1} e_{i_2} \dotsm e_{i_\ell} 1_\lambda
        = e_{j_1} e_{j_2} \dotsm e_{j_\ell} 1_\lambda.
      \]
  \end{asparaenum}

  \emph{Proof of \eqref{lem-claim:U0-span}}: Given a morphism in $\mathcal{C}$ that is a composition of $e_i 1_\lambda$ and $f_i 1_\lambda$, it follows from \eqref{rel:eifj} and \eqref{rel:eifi} that this composition is isomorphic to a 1-morphism of the form \eqref{eq:dotU0-hom-element}, possibly not satisfying the condition $\{i_1,\dotsc,i_k\} \cap \{j_1,\dotsc,j_\ell\} = \varnothing$.  To see that we can also satisfy this condition, choose $a \in \{1,\dotsc,k\}$ and $b \in \{1,\dotsc,\ell\}$ such that $i_a = j_b$ and such that we cannot find $a' \in \{a,\dotsc,k\}$ and $b' \in \{1,\dotsc,b\}$ such that $i_{a'} = j_{b'}$ and either $a' > a$ or $b' < b$.  (Intuitively speaking, we pick an ``innermost'' $f_i$, $e_i$ pair.)  We claim that none of the indices $a+1,a+2,\dotsc,k$ or $1,2,\dotsc,b-1$ is equal to $i_a-1$ or $i_a+1$.  It will then follow from \eqref{rel:eiej} and \eqref{rel:fifj} that our morphism is equal to one in which $f_{i_a}$ is immediately to the left of $e_{j_b}$, allowing us to use \eqref{rel:fiei} to cancel this pair.  Then statement~\eqref{lem-claim:U0-span} follows by induction.

  To prove the claim, consider the morphism $1_\mu e_{j_b} e_{j_{b+1}} \dotsm e_{j_\ell} 1_\lambda$.  We then have $\mu = \lambda \boxminus j_\ell \boxminus \dotsb \boxminus j_b$.  In particular, $\mu$ has an addable $j_b$ box.  If we now remove a $j_b + 1$ box or a $j_b - 1$ box, the resulting Young diagram will no longer have an addable $j_b$ box.  Therefore, by our assumption that we have picked the innermost $f_i$, $e_i$ pair, none of the indices $1,2,\dotsc,b-1$ is equal to $j_b+1$ or $j_b-1$.  So $\mu \boxminus j_{b-1} \boxminus \dotsb \boxminus j_1$ has an addable $j_b$ box.  But then it does not have an addable $j_b+1$ box or an addable $j_b-1$ box.  Therefore, none of the indices $i_{a+1},\dotsc,i_k$ is equal to $i_a+1$ or $i_a-1$.  This proves the claim.

  \emph{Proof of \eqref{lem-claim:U0-f-order}}:  Fix $\lambda \in \cP$.  We prove the statement by induction on $\ell$.  It is clear for $\ell=0$ and $\ell=1$.  Suppose $\ell \ge 2$.  The partition $\lambda$ has a $j_\ell$-addable box and an $i_\ell$-addable box.  If $j_\ell = i_\ell$, then the result follows by the inductive hypothesis applied to $\lambda \boxplus j_\ell$.  So we assume $j_\ell \ne i_\ell$.  By assumption, there must exist some $a \in \{1,\dotsc,\ell-1\}$ such that $j_a = i_\ell$.  Choose the maximal $a$ with this property.  By \eqref{eq:U0-box-add-order}, $\lambda$ has an addable $i_\ell$-box.  Thus, by an argument as in the proof of statement \eqref{lem-claim:U0-span}, none of the integers $j_\ell, j_{\ell-1}, \dotsc, j_{a+1}$ can be equal to $j_a \pm 1$.  Then, by \eqref{rel:fifj}, we have
  \[
    f_{j_1} f_{j_2} \dotsm f_{j_\ell} 1_\lambda
    = f_{j_1} f_{j_2} \dotsm f_{j_{a-1}} f_{j_{a+1}} \dotsm f_{j_\ell} f_{j_a} 1_\lambda.
  \]
  Then statement \eqref{lem-claim:U0-f-order} follows by the inductive hypothesis applied to $\lambda \boxplus j_a = \lambda \boxplus i_\ell$.

  \emph{Proof of \eqref{lem-claim:U0-e-order}}: The proof of statement \eqref{lem-claim:U0-e-order} is analogous to that of statement \eqref{lem-claim:U0-f-order}.
\end{proof}

%-----------------------------------------------------
\subsection{Notation and conventions for 2-categories}
%-----------------------------------------------------

We will use calligraphic font for 1-categories ($\cA$, $\cC$, $\cM$, $\cV$, etc.) and script font for 2-categories ($\sA$, $\sC$, $\sU$, $\sH$, etc.).   We use bold lowercase for functors ($\ba$, $\bfr$, etc.) and bold uppercase for 2-functors ($\bF$, $\bS$, etc.).  The notation $\bzero$ will denote a zero object in a 1-category or 2-category.  Other objects will be denoted with italics characters ($x$, $y$, $e$, etc.).  We use sans serif font for 1-morphisms ($\se$, $\mathsf{x}$, $\sQ$, etc.) and Greek letters for 2-morphisms.

If $\sC$ is a 2-category and $x,y$ are objects of $\sC$, we let $\sC(x,y)$ denote the category of morphisms from $x$ to $y$.  We denote the class of objects of $\sC(x,y)$, which are 1-morphisms in $\sC$ by $\1Mor_\sC(x,y)$.  For $\mathsf{P}, \mathsf{Q}$ objects in $\sC(x,y)$, we denote the class of morphisms from $\mathsf{P}$ to $\mathsf{Q}$, which are 2-morphisms in $\sC$, by $\2Mor_\sC(\mathsf{P},\mathsf{Q})$.  For an object $x$ of $\mathcal{C}$, we let $1_x$ denote the identity 1-morphism on $x$ and let $\id_x$ denote the identity 2-morphism on $1_x$.  We denote vertical composition of 2-morphisms by $\circ$ and horizontal composition by juxtaposition. Whenever we speak of a linear category or 2-category, or a linear functor or 2-functor, we mean $\Q$-linear.

If $\sC$ is an additive linear 2-category, we define its Grothendieck group $K(\sC)$ to be the category with the same objects as $\sC$ and whose space of morphisms between objects $x$ and $y$ is $K(\sC(x,y))$, the usual split Grothendieck group, over $\Q$, of the category $\sC(x,y)$.

%%%%%%%%%%%%%%%%%%%%%%%%%%%%%%%%%%%%%%%%%%%%%%%%%
%
\section{Modules for symmetric groups} \label{sec:Sn-modules}
%
%%%%%%%%%%%%%%%%%%%%%%%%%%%%%%%%%%%%%%%%%%%%%%%%%

In this section we recall some well-known facts about modules for symmetric groups and prove some combinatorial identities that we will need later on in our constructions.
%-----------------------------
\subsection{Module categories} \label{subsec:Sn-module-categories}
%-----------------------------

For an associative algebra $A$, we let $A\md$ denote the category of finite-dimensional left $A$-modules.  For $n \in \N$, we let $A_n = \Q S_n$ denote the group algebra of the symmetric group.  By convention, we set $A_0=A_1 = \Q$.  We index the representations of $A_n$ by partitions of $n$ in the usual way, and for $\lambda \vdash n$, we let $V_\lambda$ be the corresponding irreducible representation of $A_n$.

Let $\cM_\lambda$ denote the full subcategory of $A_n\md$ whose objects are isomorphic to direct sums of $V_\lambda$ (including the empty sum, which is the zero representation).  We then have a decomposition
\begin{equation} \label{eq:Mn-decomp}
  \cM_n := A_n\md = \bigoplus_{\lambda \vdash n} \cM_\lambda.
\end{equation}

We consider $A_n$ to be a subalgebra of $A_{n+1}$ in the natural way, where $S_n$ is the subgroup of $S_{n+1}$ fixing $n+1$.  We use the notation $(n)$ to denote $A_n$ considered as an $(A_n,A_n)$-bimodule in the usual way. We use subscripts to denote restriction of the left and right actions.  Thus, $(n+1)_n$ is $A_{n+1}$ considered as an $(A_{n+1},A_n)$-bimodule, $_n(n+1)$ is $A_{n+1}$ considered as an $(A_n,A_{n+1})$-bimodule, etc.  Then
\[
  (n+1)_n \otimes_{A_n} - \colon A_n\md \to A_{n+1}\md
  \quad \text{and} \quad
  {}_n(n+1) \otimes_{A_{n+1}} - \colon A_{n+1}\md \to A_n\md
\]
are the usual induction and restriction functors.  Tensor products of such bimodules correspond in the same way to composition of induction and restriction functors, and bimodule homomorphisms correspond to natural transformations of the corresponding functors.

We define a 2-category $\sM$ as follows.  The objects of $\sM$ are finite direct sums of $\cM_\lambda$, $\lambda \in \cP$, and a zero object $\bzero$.  We adopt the conventions that $\cM_n = \bzero$ when $n<0$, $\cM_{\lambda \boxplus i} = \bzero$ when $i \not \in B^+(\lambda)$, and $\cM_{\lambda \boxminus i} = \bzero$ when $i \not \in B^-(\lambda)$.  The 1-morphisms are generated, under composition and direct sum, by additive $\Q$-linear direct summands of the functors
\[
  (n+1)_n \otimes_{A_n} - \colon \cM_n \to \cM_{n+1},\quad
  \tensor*[_{n-1}]{(n)}{} \otimes_{A_n} - \colon \cM_n \to \cM_{n-1}.
\]
The 2-morphisms of $\sM$ are natural transformations of functors.

\begin{rem}
  In the above definition, it is important that we allow \emph{direct summands} of the given functors.  In Section~\ref{subsec:JM-elements} we will discuss the direct summands
  \[
    (n+1)_n^i \otimes_{A_n} - \colon \cM_\lambda \to \cM_{\lambda \boxplus i},\quad
    \tensor*[_{n-1}^i]{(n)}{} \otimes_{A_n} - \colon \cM_\lambda \to \cM_{\lambda \boxminus i},\quad
    \text{where } n = |\lambda|.
  \]
  arising from decomposing induction and restriction according to eigenspaces for the action of Jucys--Murphy elements.
\end{rem}

For $\lambda \vdash n$, consider the functors
\begin{equation} \label{eq:MV-V-equivalence}
  \bi_\lambda := \Hom_{A_n}(V_\lambda, -) \colon \cM_\lambda \to \cV
  \quad \text{and} \quad
  \bj_\lambda := V_\lambda \otimes_\Q - \colon \cV \to \cM_\lambda.
\end{equation}
We have $\bi_\lambda \circ \bj_\lambda \cong 1_\cV$ and $\bj_\lambda \circ \bi_\lambda \cong 1_{\cM_\lambda}$, and hence an equivalence of categories $\cM_\lambda \cong \cV$.

%------------------------------
\subsection{Decategorification}
%------------------------------

Suppose $\lambda, \mu \in \cP$ and consider an additive linear functor $\ba \colon \cM_\lambda \to \cM_\mu$.  Then the functor $\bi_\mu \circ \ba \circ \bj_\lambda$ is naturally isomorphic to a direct sum of some finite number of copies of the identity functor.   In other words, under the equivalences \eqref{eq:MV-V-equivalence}, every object in $\sM(\cM_\lambda,\cM_\mu)$ is isomorphic to $1_\cV^{\oplus n}$ for some $n\geq0$, where $1_\cV \colon \cV \to \cV$ is the identity functor.

\details{
  Let $m = |\mu|$ and $n = \dim \Hom_{A_m}(V_\mu, \ba(V_\lambda))$.  Since $\cM_\lambda$ is an additive linear category, we have the copower functor
  \[
    \bc_\lambda \colon \cM_\lambda \times \cV \to \cM_\lambda,\quad (V,U) \mapsto V \otimes_\Q U.
  \]
  This copower functors commutes, up to natural isomorphism, with additive linear functors such as $\ba$.  More precisely, we have a diagram of functors
  \[
    \xymatrix{
      \cM_\lambda \times \cV \ar[r]^(0.6){\bc_\lambda} \ar[d]_{\ba \times 1_\cV} & \cM_\lambda \ar[d]^{\ba} \\
      \cM_\mu \times \cV \ar[r]_(0.6){\bc_\mu} & \cM_\mu
    }
  \]
  that is commutative up to natural isomorphism of functors.  (See, for example, \href{https://mathoverflow.net/questions/246948/do-copowers-commute-with-k-linear-functors}{here}.)  Now, we have $\bj_\lambda = \bc_\lambda(V_\lambda,-)$.  Therefore,
  \begin{multline*}
    \bi_\mu \circ \ba \circ \bj_\lambda
    = \bi_\mu \circ \ba \circ \bc (V_\lambda, -)
    \cong \bi_\mu \circ \bc \circ (\ba \times 1_\cV) (V_\lambda, -)
    = \bi_\mu \circ \bc (\ba(V_\lambda),-) \\
    = \Hom_{A_m}(V_\mu, \ba(V_\lambda) \otimes_\Q -)
    \cong \Hom_{A_m}(V_\mu, \ba(V_\lambda)) \otimes_\Q -
    \cong 1_\cV^{\oplus n}.
  \end{multline*}
}

It follows that $K(\sM)$ is the category given by
\[
  \Ob K(\sM) = \Ob \sM
  \quad \text{and} \quad
  \Mor_{K(\sM)}(\cM_\lambda,\cM_\mu) = \Q,\ \lambda,\mu \in \cP.
\]
Composition of morphisms is given by multiplication of the corresponding elements of $\Q$.

We have a natural functor $K(\sM) \to \cV$ given by
\begin{equation} \label{eq:KM-to-V-functor}
    \cM_\lambda \mapsto \Q s_\lambda, \quad
    z \mapsto (s_\lambda \mapsto z s_\mu),
\end{equation}
for $\lambda,\mu \in \cP$ and $z \in \Q = \Mor_{K(\sM)}(\cM_\lambda, \cM_\mu)$.  This functor is clearly an equivalence of categories.

%-------------------------------------------------------------------
\subsection{Biadjunction and the fundamental bimodule decomposition}
%-------------------------------------------------------------------

\begin{prop} \label{prop:adjunction-maps}
  The maps
  \begin{gather*}
    \varepsilon_\rR \colon (n+1)_n(n+1) \to (n+1),\quad \varepsilon_\rR(a \otimes b) = ab,\quad a \in (n+1)_n,\ b \in \prescript{}{n}(n+1), \\
    \eta_\rR \colon (n) \hookrightarrow {_n}(n+1)_n,\quad \eta_{\rR}(a) = a,\quad a \in (n), \\
    \varepsilon_\rL \colon {_n}(n+1)_n \to (n),\quad
    \varepsilon_\rL(g) =
    \begin{cases}
      g & \text{if } g \in S_n, \\
      0 & \text{if } g \in S_{n+1} \setminus S_n,
    \end{cases}\\
    \eta_\rL \colon (n+1) \to (n+1)_n (n+1),\quad
    \eta_\rL(a) = a \sum_{i \in \{1,\dotsc,n+1\}} s_i \dotsm s_n \otimes s_n \dotsm s_i,\quad a \in (n+1),
  \end{gather*}
  (where we interpret the expression $s_i \dotsm s_n \otimes s_n \dotsm s_i$ as $1 \otimes 1$ when $i=n+1$) are bimodule homomorphisms and satisfy the relations

  \noindent\begin{minipage}{0.5\linewidth}
    \begin{equation} \label{eq:up-right-zigzag}
      (\varepsilon_\rR \otimes \id) \circ (\id \otimes \eta_\rR) = \id,
    \end{equation}
  \end{minipage}%
  \begin{minipage}{0.5\linewidth}
    \begin{equation} \label{eq:down-right-zigzag}
      (\id \otimes \varepsilon_\rR) \circ (\eta_\rR \otimes \id) = \id,
    \end{equation}
  \end{minipage}\par\vspace{\belowdisplayskip}

  \noindent\begin{minipage}{0.5\linewidth}
    \begin{equation} \label{eq:down-left-zigzag}
      (\varepsilon_\rL \otimes \id) \circ (\id \otimes \eta_\rL) = \id,
    \end{equation}
  \end{minipage}%
  \begin{minipage}{0.5\linewidth}
    \begin{equation} \label{eq:up-left-zigzag}
      (\id \otimes \varepsilon_\rL) \circ (\eta_\rL \otimes \id) = \id.
    \end{equation}
  \end{minipage}\par\vspace{\belowdisplayskip}

  \noindent In particular, $(n+1)_n$ is both left and right adjoint to $\prescript{}{n}(n+1)$ in the 2-category of bimodules over rings.
\end{prop}

\begin{proof}
  The verification of these relations, which are a formulation of the well-known Frobenius reciprocity between induction and restriction for finite groups, is a straightforward computation.
\end{proof}

It is well known (see, for example, \cite[Lem.~7.6.1]{Kle05}) that we have a decomposition
\begin{equation}
  A_{n+1} = A_n \oplus (A_n s_n A_n),
\end{equation}
and an isomorphism of $(A_n,A_n)$-bimodules
\begin{equation}
  (n)_{n-1}(n) \xrightarrow{\cong} A_n s_n A_n \subseteq \tensor*[_n]{(n+1)}{_n},\quad a \otimes b \mapsto a s_n b.
\end{equation}
This yields an isomorphism of $(A_n,A_n)$-bimodules
\begin{equation}
  (n)_{n-1}(n) \oplus (n) \xrightarrow{\cong} \tensor*[_n]{(n+1)}{_n},\quad (a \otimes b,c) \mapsto a s_n b + c.
\end{equation}
More precisely, the maps
\begin{equation} \label{eq:decomp-diamond}
  \vcenter{\vbox{
    \xymatrix{
      (n)_{n-1}(n) \ar@<3pt>[r]^{\rho} & {}_n(n+1)_n \ar@<3pt>[l]^{\tau} \ar@<-3pt>[r]_(.6){\varepsilon_\rL} & (n) \ar@<-3pt>[l]_(.4){\eta_\rR}
    }
  }}
\end{equation}
where
\begin{gather}
  \rho(a \otimes b) = a s_n b,\quad \text{for } a,b \in A_n, \label{eq:rho-def} \\
  \tau(a) = 0,\quad \tau(a s_n b) = a \otimes b,\quad \text{for } a,b \in A_n \subseteq A_{n+1}, \label{eq:tau-def}
\end{gather}
satisfy
\begin{gather}
  \varepsilon_\rL \circ \eta_\rR = \id,\quad
  \tau \circ \rho = \id,\quad
  \varepsilon_\rL \circ \rho = 0,\quad
  \tau \circ \eta_\rR = 0, \label{eq:diamond-components} \\
  \rho \circ \tau + \eta_\rR \circ \varepsilon_\rL = \id. \label{eq:diamond-sum-id}
\end{gather}

Note that
\[
  A_n s_n A_n = \Span_\C (S_{n+1} \setminus S_n) \subseteq A_{n+1}.
\]

%------------------------------------------------------------
\subsection{The Jucys--Murphy elements and their eigenspaces} \label{subsec:JM-elements}
%------------------------------------------------------------

Recall that the \emph{Jucys--Murphy elements} of $A_n$ are given by
\begin{equation}
  J_1 = 0,\quad J_i = \sum_{k=1}^{i-1} (k,i),\quad i = 2,\dotsc,n,
\end{equation}
where $(k,i) \in S_n$ denotes the transposition of $k$ and $i$.  The element $J_i$ commutes with $A_{i-1}$.  Thus, left multiplication by $J_{n+1}$ is an endomorphism of the bimodule $_n(n+1)$.  In fact, this action is semisimple and the set of eigenvalues is $\{-n,-n+1,\dotsc,n-1,n\}$.

We let $\tensor*[^i_n]{(n+1)}{}$, $i \in \Z$, denote the $i$-eigenspace of $\tensor*[_n]{(n+1)}{}$ under left multiplication by $J_{n+1}$. Similarly, we let $(n+1)_n^i$, $i \in \Z$, denote the $i$-eigenspace of $(n+1)_n$ under right multiplication by $J_{n+1}$.  Since these two actions (right and left multiplication by $J_{n+1}$) commute, we can also consider the simultaneous eigenspaces $\tensor*[^i_n]{(n+1)}{^j_n}$, $i,j \in \Z$.  So we have
\begin{equation} \label{eq:ind-res-decomp}
  (n+1)_n = \bigoplus_{i \in \Z} (n+1)_n^i,\quad
  \tensor*[_n]{(n+1)}{} = \bigoplus_{i \in \Z} \tensor*[_n^i]{(n+1)}{},\quad
  \tensor*[_n]{(n+1)}{_n} = \bigoplus_{i,j \in \Z} \tensor*[_n^i]{(n+1)}{_n^j}.
\end{equation}

Similarly, for $i,j \in \Z$, we let $(n+1)_{n-1}^{i,j}$ denote the simultaneous eigenspace of $(n+1)_{n-1}$ under right multiplication by $J_{n+1}$ and $J_n$, with respective eigenvalues $i$ and $j$.  Similarly, we let $\tensor*[_{n-1}^{j,i}]{(n+1)}{}$ denote the simultaneous eigenspace of $\tensor*[_{n-1}]{(n+1)}{}$ under left multiplication by $J_{n+1}$ and $J_n$, with respective eigenvalues $i$ and $j$.

We have
\begin{equation} \label{eq:i-ind-res-irred}
  (n+1)_n^i \otimes_{A_n} V_\lambda \cong V_{\lambda \boxplus i}
  \quad \text{and} \quad
  \tensor*[_n^i]{(n+1)}{} \otimes_{A_{n+1}} V_\mu \cong V_{\mu \boxminus i}
\end{equation}
for $\lambda \vdash n$ and $\mu \vdash n+1$, where we define $V_\bzero$ to be the zero module.

The primitive central idempotents in $\Q S_n$ are
\begin{equation} \label{eq:central-idempotent-explicit}
  e_\lambda = \frac{\dim V_\lambda}{n!} \sum_{w \in S_n} \tr (w^{-1}_\lambda) w,\quad \lambda \vdash n,
\end{equation}
where $w^{-1}_\lambda$ denotes the action of $w^{-1}$ on the representation $V_\lambda$.  (See, for example, \cite[(2.13), p.~23]{FH91}.)  Multiplication by $e_\lambda$ is projection onto the $V_\lambda$-isotypic component.  It follows that
\begin{gather} \label{eq:i-induction-idempotent-decomp}
  (n+1)_n^i
  = \bigoplus_{\mu \vdash n} e_{\mu \boxplus i} (n+1)_n e_\mu
  \\ \label{eq:i-restriction-idempotent-decomp}
  \tensor*[_n^i]{(n+1)}{}
  = \bigoplus_{\mu \vdash n} e_\mu\, {}_n(n+1) e_{\mu \boxplus i}
\end{gather}

%----------------------------------
\subsection{Combinatorial formulas}
%----------------------------------

In this subsection we prove some combinatorial identities, used elsewhere in the paper, involving the dimensions $d_\lambda:=\dim(V_\lambda)$ of irreducible representations of symmetric groups.  By convention, $d_{\lambda\boxplus i }=0$ if $\lambda$ has no addable $i$-box, and $d_{\lambda \boxminus i} = 0$ if $\lambda$ has no removable $i$-box.

It follows from \eqref{eq:ind-res-decomp} and \eqref{eq:i-ind-res-irred} that
\begin{equation} \label{rel:hl}
  d_\lambda=\sum_{j\in B^-(\lambda)}d_{\lambda \boxminus j}
\end{equation}
\begin{equation}\label{rel:hl3}
  \sum_{i\in B^+(\lambda)}d_{\lambda \boxplus i}=(|\lambda|+1)d_{\lambda}
\end{equation}

Recall the hook-length formula
\begin{equation} \label{rel:hldef}
  d_{\lambda}=\frac{|\lambda|!}{\prod_{i,j} h_{\lambda}(i,j)},\quad \lambda \in \cP.
\end{equation}
Here $h_{\lambda}(i,j)$ counts the number of boxes in the Young diagram of $\lambda$ in the hook whose upper left corner is in position $(i,j)$ and the product is over the positions $(i,j)$ of all the boxes in the Young diagram of $\lambda$.

\begin{lem}
  For a partition $\lambda$, we have
  \begin{equation} \label{rel:hl1}
    \left(1-\frac{1}{(i-j)^2}\right) \frac{|\lambda|+1}{|\lambda|+2}\, \frac{d_{\lambda}d_{\lambda \boxplus i \boxplus j}}{d_{\lambda \boxplus i}d_{\lambda \boxplus j}}=1\quad \forall\ i,j \in B^+(\lambda),\ |i-j|>1,
  \end{equation}
  \begin{equation}\label{rel:hl5}
    \sum_{i\in B^+(\lambda)}\frac{d_{\lambda \boxplus i}}{(i-j)^2} = \frac{(|\lambda|+1)d_{\lambda}^2}{|\lambda|d_{\lambda \boxminus j}}\quad \forall\ j \in B^-(\lambda),
  \end{equation}
  \begin{equation}\label{rel:hl8}
    \sum_{j\in B^-(\lambda)}\frac{d_{\lambda \boxminus j}}{(i-j)^2} = \left(\frac{(|\lambda|+1)d_\lambda^2}{d_{\lambda \boxplus i}|\lambda|} - \frac{d_{\lambda}}{|\lambda|}\right)\quad \forall\ i\in B^+(\lambda).
  \end{equation}
\end{lem}

\begin{proof}
  Relation \eqref{rel:hl1} follows from a direct computation using \eqref{rel:hldef}.  We omit the details.

  \details{
    Suppose $i,j \in B^+(\lambda)$ with $|i-j|>1$.  Let $(i_1,i_2)$ be the addable $i$-box and $(j_1,j_2)$ be the addable $j$-box.  Without loss of generality, we may assume that $i_1 < j_1$.  Then we have
    \begin{gather*}
      d_\lambda = \frac{|\lambda|!}{\prod_{k,\ell} h_{\lambda}(k,\ell)}, \\
      d_{\lambda \boxplus i} = \frac{(|\lambda|+1)!}{\prod_{\substack{k\neq i_1 \\ \ell \neq i_2}} h_\lambda (k,\ell) \prod_\ell (h_\lambda (i_1,\ell)+1) \prod_k (h_\lambda(k,i_2)+1)}, \\
      d_{\lambda \boxplus j} = \frac{(|\lambda|+1)!}{\prod_{\substack{k \neq j_1 \\ \ell \neq j_2}} h_\lambda (k,\ell) \prod_\ell (h_{\lambda}(j_1,\ell)+1) \prod_k (h_{\lambda}(k,j_2)+1)},
    \end{gather*}
    while $d_{\lambda \boxplus i \boxplus j}$ is equal to
    \[
      \frac{(|\lambda|+2)!}{\prod_{\substack{k\neq i_1,j_1\\ \ell \neq i_2,j_2}} h_\lambda (k,\ell) \prod_{\ell \neq j_2} (h_\lambda (i_1,\ell)+1)\prod_\ell (h_\lambda (j_1,\ell)+1) \prod_{k\neq i_1}(h_{\lambda}(k,j_2)+1) \prod_k (h_\lambda (k,i_2)+1) (h_\lambda(i_1,j_2)+2)}.
    \]
    Therefore,
    \[
      \frac{d_{\lambda}d_{\lambda\boxplus i \boxplus j}}{d_{\lambda \boxplus i}d_{\lambda \boxplus j}}=\frac{|\lambda|+2}{|\lambda|+1}\frac{(h_{\lambda(i_1,j_2)}+1)^2}{h_{\lambda}(i_1,j_2 )(h_{\lambda}(i_1,j_2)+2)}.
    \]
    Using the fact that $h_{\lambda}(i_1,j_2)=i-j-1$, we get that
    \[
      \frac{d_{\lambda}d_{\lambda\boxplus i \boxplus j}}{d_{\lambda \boxplus i}d_{\lambda \boxplus j}}
      =\frac{|\lambda|+2}{|\lambda|+1}\, \frac{(j-i)^2}{(j-i)^2-1},
    \]
    and \eqref{rel:hl1} follows.
  }

  To prove \eqref{rel:hl5}, assume $j \in B^-(\lambda)$.  Then, for $|i-j|>1$, \eqref{rel:hl1} for $\lambda \boxminus j$ implies
  \begin{equation} \label{eq:h15-proof-eq}
    \frac{1}{(i-j)^2}\frac{|\lambda|d_{\lambda\boxplus i}d_{\lambda\boxminus j}}{(|\lambda|+1)d_\lambda}=\frac{|\lambda|d_{\lambda\boxplus i}d_{\lambda\boxminus j}}{(|\lambda|+1)d_\lambda}-d_{\lambda\boxplus i \boxminus j}.
  \end{equation}
  When $|i-j|=1$, we have $d_{\lambda \boxplus i \boxminus j} = 0$, and so \eqref{eq:h15-proof-eq} in fact holds for all $i \ne j$.  Therefore,
  \begin{multline*}
    \sum_{i \in B^+(\lambda)} \frac{1}{(i-j)^2} \frac{|\lambda| d_{\lambda \boxplus i} d_{\lambda \boxminus j}}{(|\lambda|+1) d_\lambda}
    = \frac{|\lambda|d_{\lambda\boxminus j }}{(|\lambda|+1)d_\lambda}\sum_{i\in B^+(\lambda)}d_{\lambda\boxplus i }-\sum_{i\in B^+(\lambda)}d_{\lambda\boxplus i \boxminus j}
    \\
    = \frac{|\lambda|d_{\lambda\boxminus j }}{(|\lambda|+1)d_\lambda}\sum_{i\in B^+(\lambda)}d_{\lambda\boxplus i }-\sum_{i \in B^+(\lambda \boxminus j)}d_{\lambda \boxminus j \boxplus i} + d_\lambda
    \stackrel{\eqref{rel:hl3}}{=} |\lambda| d_{\lambda \boxminus j} - |\lambda| d_{\lambda \boxminus j} + d_\lambda
    = d_\lambda,
  \end{multline*}
  where the second equality follows from the fact that $\lambda \boxplus i \boxminus j = \bzero$ whenever $i \notin B^+(\lambda \boxminus j)$.  Relation \eqref{rel:hl5} follows.

  To prove relation \eqref{rel:hl8}, suppose $i \in B^+(\lambda)$.  Then
  \begin{multline*}
    \sum_{j \in B^-(\lambda)} \frac{|\lambda| d_{\lambda \boxplus i} d_{\lambda \boxminus j}}{(i-j)^2 d_\lambda}
    \stackrel{\eqref{eq:h15-proof-eq}}{=} \sum_{j\in B^-(\lambda)} \frac{|\lambda| d_{\lambda \boxplus i} d_{\lambda \boxminus j}}{d_\lambda} - \sum_{j\in B^-(\lambda)} (|\lambda|+1) d_{\lambda \boxplus i \boxminus j} \\
    \stackrel{\eqref{rel:hl}}{=} |\lambda| d_{\lambda \boxplus i} - (|\lambda|+1) \left( \sum_{j\in B^-(\lambda \boxplus i)} d_{\lambda \boxplus i \boxminus j} - d_\lambda \right)
    \stackrel{\eqref{rel:hl}}{=} |\lambda| d_{\lambda \boxplus i} - (|\lambda|+1)(d_{\lambda \boxplus i} - d_\lambda)
    = (|\lambda|+1) d_\lambda - d_{\lambda \boxplus i},
  \end{multline*}
  where, in the second equality, we have used the fact that $\lambda \boxplus i \boxminus j = \bzero$ whenever $j \not \in B^-(\lambda \boxplus i)$.  Relation \eqref{rel:hl8} follows.
\end{proof}

\begin{lem}
  \begin{enumerate}
    \item For $\lambda \vdash n$, we have
      \begin{equation} \label{eq:trace-of-epsilon_lambda}
        \varepsilon_\rL(e_\lambda) = \sum_{i \in B^-(\lambda)} \frac{d_\lambda}{n d_{\lambda \boxminus i}} e_{\lambda \boxminus i}.
      \end{equation}

    \item For $\lambda \vdash n$ and $i \in \Z$, we have
      \begin{equation} \label{eq:JM-idempotent-product}
        J_{n+1} e_{\lambda \boxplus i} e_\lambda = i e_{\lambda \boxplus i} e_\lambda.
      \end{equation}

    \item For $\lambda \vdash n$ and $i \in \Z$, we have
      \begin{equation} \label{eq:idempotent-symmetrize}
        \frac{1}{(n+1)!} \sum_{w \in S_{n+1}} w e_\lambda w^{-1}
        = \sum_{i \in B^+(\lambda)} \frac{d_\lambda}{d_{\lambda \boxplus i}} e_{\lambda \boxplus i}.
      \end{equation}
  \end{enumerate}
\end{lem}

\begin{proof}
  \begin{asparaenum}
    \item Fix $\lambda \vdash n$.  As $\Q S_{n-1}$-modules, we have a decomposition $V_\lambda \cong \bigoplus_{i \in B^-(\lambda)} V_{\lambda \boxminus i}$.  Then
      \begin{multline*}
        \varepsilon_\rL (e_\lambda)
        \stackrel{\eqref{eq:central-idempotent-explicit}}{=}
        \varepsilon_\rL \left( \frac{d_\lambda}{n!} \sum_{w \in S_n} \tr \left( w^{-1}_\lambda \right) w \right)
        = \frac{d_\lambda}{n!} \sum_{w \in S_{n-1}} \tr \left( w^{-1}_\lambda \right) w
        \\
        = \frac{d_\lambda}{n!} \sum_{i \in B^-(\lambda)} \sum_{w \in S_{n-1}} \tr \left( w^{-1}_{\lambda \boxminus i} \right) w
        \stackrel{\eqref{eq:central-idempotent-explicit}}{=}
        \sum_{i \in B^-(\lambda)} \frac{d_\lambda}{n d_{\lambda \boxminus i}} e_{\lambda \boxminus i}.
      \end{multline*}

    \item Recall that if $\lambda \vdash n$ and $V$ is an $A_n$-module, then $e_\lambda(V)$ is the $\lambda$-isotypic component of $V$.  Now, to show that two elements of $A_{n+1}$ are equal it suffices to show that they act identically on every irreducible representation of $A_{n+1}$.  For $\mu \vdash (n+1)$ and $\lambda \vdash n$, let $V_{\mu,\lambda}$ be the $\lambda$-isotypic component of $V_\mu$ (this is either zero or isomorphic to $V_\lambda$ as an $A_n$-module).  Then
        \[
          e_{\lambda\boxplus i}e_\lambda (V_\mu)=\delta_{\mu,\lambda\boxplus i}V_{\lambda\boxplus i,\lambda}
        \]
        The result then follows from the fact that $V_{\lambda\boxplus i,\lambda}$ is the $i$-eigenspace of $V_{\lambda \boxplus i}$ under multiplication by $J_{n+1}$.

    \item Fix $\lambda \vdash n$.  Then $\frac{1}{(n+1)!} \sum_{w \in S_{n+1}} w e_\lambda w^{-1}$ belongs to the center of $\Q S_{n+1}$, and so is a linear combination of the central idempotents $e_\nu$, $\nu \vdash n+1$.  Thus we have
        \[
          \frac{1}{(n+1)!} \sum_{w \in S_{n+1}} w e_\lambda w^{-1}
          = \sum_{\nu \vdash n+1} c_\nu e_\nu
        \]
        for some $c_\nu \in \Q$.  Therefore, for $\mu \vdash n+1$, we have
        \[
          c_\mu e_\mu
          = e_\mu \sum_{\nu \vdash n+1} c_\nu e_\nu
          = \frac{e_\mu}{(n+1)!} \sum_{w \in S_{n+1}} w e_\lambda w^{-1}
          = \frac{1}{(n+1)!} \sum_{w \in S_{n+1}} w e_\mu e_\lambda w^{-1}.
        \]
        Then, considering the action of $\Q S_{n+1}$ on $V_\mu$, we have
        \[
          c_{\mu} d_{\mu}
          = \tr \left( c_\mu e_\mu \right)
          = \frac{1}{(n+1)!} \sum_{w \in S_{n+1}} \tr \left( e_\mu e_\lambda \right)
          =
          \begin{cases}
            d_\lambda & \text{if } \mu = \lambda \boxplus i \text{ for some } i \in \Z, \\
            0 & \text{otherwise}.
          \end{cases}
        \]
        The result follows. \qedhere
  \end{asparaenum}
\end{proof}

%%%%%%%%%%%%%%%%%%%%%%%%%%%%%%%%%%%%%%%%
%
\section{The 2-category $\sA$} \label{sec:sA-def}
%
%%%%%%%%%%%%%%%%%%%%%%%%%%%%%%%%%%%%%%%%

In this section we define an additive linear 2-category $\sA$ and investigate some of its key properties.  We will show in Section~\ref{subsec:truncated-KL} that $\sA$ is equivalent to the degree zero part of a truncation of a categorified quantum group.  Later, in Section~\ref{sec:sHepsilon-A-equivalence}, we will also show that $\sA$ is equivalent to a summand of a truncation of a Heisenberg 2-category.  We will also see, in Proposition~\ref{prop:sU-sM-equivalence}, that $\sA$ is equivalent to the category $\sM$.

%----------------------------------------
\subsection{Definition}
%----------------------------------------

The set of objects of $\sA$ is the free monoid on the set of partitions:
\[
  \Ob \sA = \N[\cP].
\]
We denote the zero object by $\bzero$.  The 1-morphisms of $\sA$ are generated by (i.e.\ direct sums of compositions of)
\begin{gather*}
  \sF_i 1_\lambda = 1_{\lambda \boxplus i} \sF_i = 1_{\lambda \boxplus i} \sF_i 1_\lambda,\quad \text{and} \\
  \sE_i 1_\lambda = 1_{\lambda \boxminus i} \sE_i = 1_{\lambda \boxminus i} \sE_i 1_\lambda,\quad
  i \in \Z,\ \lambda \in \cP.
\end{gather*}
We adopt the convention that if $\lambda$ does not have an addable (resp.\ removable) $i$-box, then $1_{\lambda \boxplus i}= 0$ (resp.\ $1_{\lambda \boxminus i}=0$) and hence $\lambda \boxplus i \cong \bzero$ (resp.\ $\lambda \boxminus i \cong \bzero$) in $\sA$.  In particular,
\begin{equation}
  \sE_i^2 1_\lambda = 0
  \quad \text{and} \quad
  \sF_i^2 1_\lambda = 0
  \quad \text{for all } \lambda \in \cP,
\end{equation}
and similarly
\begin{equation}
  \sE_i\sE_{i\pm1}\sE_i 1_\lambda = 0
  \quad \text{and} \quad
  \sF_i\sF_{i\pm1}\sF_i 1_\lambda = 0
  \quad \text{for all } \lambda \in \cP.
\end{equation}

The space of 2-morphisms between two 1-morphisms is the $\Q$-algebra generated by suitable planar diagrams modulo local relations.  The diagrams consist of oriented compact one-manifolds immersed into the plane strip $\R \times [0,1]$ modulo local relations, with strands labeled by integers and regions of the strip labeled/colored by elements of $\cP \sqcup \{\bzero\}$.  In particular, the identity 2-morphism of $\sF_i 1_\lambda$ will be denoted by an upward strand labeled $i$, where the region to the right of the arrow is labeled $\lambda$, while the identity 2-morphism of $\sE_i 1_\lambda$ is denoted by a downward strand labeled $i$, where the region to the right of the strand is labeled $\lambda$:
\[
  \begin{tikzpicture}[>=stealth,baseline={([yshift=-.5ex]current bounding box.center)}]
    \draw[->] (0,0) node [anchor=north] {\strandlabel{i}} -- (0,0.5);
    \draw (0.25,0.25) node {\regionlabel{\lambda}};
  \end{tikzpicture}
  \qquad \qquad \qquad
  \begin{tikzpicture}[>=stealth,baseline={([yshift=-.5ex]current bounding box.center)}]
    \draw[<-] (0,0) node [anchor=north] {\strandlabel{i}} -- (0,0.5);
    \draw (0.25,0.25) node {\regionlabel{\lambda}};
  \end{tikzpicture}
\]

Strands of distinct color may intersect transversely, but no triple intersections are allowed.  The space of 2-morphisms is the space of such planar diagrams up to isotopy and modulo local relations.  The domain and codomain are given by the orientations of the strands at the bottom and top of the diagram respectively.

The local relations are as follows, where $i$, $j$, and $k$ range over all integers satisfying
\[
  |i-j|,|i-k|,|j-k| > 1.
\]
For relations when the regions are not labeled, we impose the relation for all labelings of the regions.

\noindent\begin{minipage}{0.5\linewidth}
  \begin{equation} \label{rel:up-up-up-braid}
    \begin{tikzpicture}[>=stealth,baseline={([yshift=.5ex]current bounding box.center)}]
      \draw (0,0) node [anchor=north] {\strandlabel{i}} -- (1,1)[->];
      \draw (1,0) node [anchor=north] {\strandlabel{k}} -- (0,1)[->];
      \draw[->] (0.5,0) node [anchor=north] {\strandlabel{j}} .. controls (0,0.5) .. (0.5,1);
    \end{tikzpicture}
    \ =\
    \begin{tikzpicture}[>=stealth,baseline={([yshift=.5ex]current bounding box.center)}]
      \draw (0,0) node [anchor=north] {\strandlabel{i}} -- (1,1)[->];
      \draw (1,0) node [anchor=north] {\strandlabel{k}} -- (0,1)[->];
      \draw[->] (0.5,0) node [anchor=north] {\strandlabel{j}} .. controls (1,0.5) .. (0.5,1);
    \end{tikzpicture}
  \end{equation}
\end{minipage}%
\begin{minipage}{0.5\linewidth}
  \begin{equation} \label{rel:up-up-double-cross}
    \begin{tikzpicture}[>=stealth,baseline={([yshift=.5ex]current bounding box.center)}]
      \draw[->] (0,0) node [anchor=north] {\strandlabel{i}} .. controls (0.5,0.5) .. (0,1);
      \draw[->] (0.5,0) node [anchor=north] {\strandlabel{j}} .. controls (0,0.5) .. (0.5,1);
    \end{tikzpicture}
    \ =\
    \begin{tikzpicture}[>=stealth,baseline={([yshift=.5ex]current bounding box.center)}]
      \draw[->] (0,0) node [anchor=north] {\strandlabel{i}} --(0,1);
      \draw[->] (0.5,0) node [anchor=north] {\strandlabel{j}} -- (0.5,1);
    \end{tikzpicture}
  \end{equation}
\end{minipage}\par\vspace{\belowdisplayskip}

\noindent\begin{minipage}{0.5\linewidth}
  \begin{equation} \label{rel:down-up-double-cross}
    \begin{tikzpicture}[>=stealth,baseline={([yshift=.5ex]current bounding box.center)}]
      \draw[<-] (0,0) node [anchor=north] {\strandlabel{i}} .. controls (0.5,0.5) .. (0,1);
      \draw[->] (0.5,0) node [anchor=north] {\strandlabel{j}} .. controls (0,0.5) .. (0.5,1);
    \end{tikzpicture}
    \ =\
    \begin{tikzpicture}[>=stealth,baseline={([yshift=.5ex]current bounding box.center)}]
      \draw[<-] (0,0) node [anchor=north] {\strandlabel{i}} --(0,1);
      \draw[->] (0.5,0) node [anchor=north] {\strandlabel{j}} -- (0.5,1);
    \end{tikzpicture}
  \end{equation}
\end{minipage}%
\begin{minipage}{0.5\linewidth}
  \begin{equation} \label{rel:up-down-double-cross}
    \begin{tikzpicture}[>=stealth,baseline={([yshift=.5ex]current bounding box.center)}]
      \draw[->] (0,0) node [anchor=north] {\strandlabel{i}} .. controls (0.5,0.5) .. (0,1);
      \draw[<-] (0.5,0) node [anchor=north] {\strandlabel{j}} .. controls (0,0.5) .. (0.5,1);
    \end{tikzpicture}
    \ =\
    \begin{tikzpicture}[>=stealth,baseline={([yshift=.5ex]current bounding box.center)}]
      \draw[->] (0,0) node [anchor=north] {\strandlabel{i}} --(0,1);
      \draw[<-] (0.5,0) node [anchor=north] {\strandlabel{j}} -- (0.5,1);
    \end{tikzpicture}
  \end{equation}
\end{minipage}\par\vspace{\belowdisplayskip}

\noindent\begin{minipage}{0.5\linewidth}
  \begin{equation} \label{rel:down-up-ii}
    \begin{tikzpicture}[>=stealth,baseline={([yshift=.5ex]current bounding box.center)}]
      \draw[<-] (0,0) node [anchor=north] {\strandlabel{i}} --(0,1);
      \draw[->] (0.5,0) node [anchor=north] {\strandlabel{i}} -- (0.5,1);
    \end{tikzpicture}
    \ =\
    \begin{tikzpicture}[anchorbase]
      \draw[->] (0,1) node [anchor=south] {\strandlabel{i}} -- (0,0.9) arc (180:360:.25) -- (0.5,1);
      \draw[<-] (0,0) -- (0,0.1) arc (180:0:.25) -- (0.5,0) node [anchor=north] {\strandlabel{i}};
    \end{tikzpicture}
  \end{equation}
\end{minipage}%
\begin{minipage}{0.5\linewidth}
  \begin{equation} \label{rel:up-down-ii}
    \begin{tikzpicture}[>=stealth,baseline={([yshift=.5ex]current bounding box.center)}]
      \draw[->] (0,0) node [anchor=north] {\strandlabel{i}} --(0,1);
      \draw[<-] (0.5,0) node [anchor=north] {\strandlabel{i}} -- (0.5,1);
    \end{tikzpicture}
    \ =\
    \begin{tikzpicture}[anchorbase]
      \draw[<-] (0,1) -- (0,0.9) arc (180:360:.25) -- (0.5,1) node [anchor=south] {\strandlabel{i}};
      \draw[->] (0,0) node [anchor=north] {\strandlabel{i}} -- (0,0.1) arc (180:0:.25) -- (0.5,0);
    \end{tikzpicture}
  \end{equation}
\end{minipage}\par\vspace{\belowdisplayskip}

\noindent\begin{minipage}{0.5\linewidth}
  \begin{equation} \label{rel:clockwise-i-circle}
    \begin{tikzpicture}[anchorbase]
      \draw[-<] (0,0) arc (180:540:0.3) node [anchor=east] {\strandlabel{i}};
      \draw (0.8,0.1) node {\regionlabel{\lambda}};
    \end{tikzpicture}
    \ =
    \id_\lambda \text{ for } i \in B^-(\lambda)
  \end{equation}
\end{minipage}%
\begin{minipage}{0.5\linewidth}
  \begin{equation} \label{rel:ccc}
    \begin{tikzpicture}[anchorbase]
      \draw[->] (0,0) arc (180:540:0.3) node [anchor=east] {\strandlabel{i}};
      \draw (0.8,0.1) node {\regionlabel{\lambda}};
    \end{tikzpicture}
    \ =
    \id_\lambda \text{ for } i \in B^+(\lambda)
  \end{equation}
\end{minipage}\par\vspace{\belowdisplayskip}

\begin{rem} \label{rem:double-cross-diff-by-one}
  Note that the crossings
  \[
    \begin{tikzpicture}[>=stealth,baseline={([yshift=-.5ex]current bounding box.center)}]
      \draw[->] (0,0) node [anchor=north] {\strandlabel{i}} -- (0.5,0.5);
      \draw[->] (0.5,0) node [anchor=north] {\strandlabel{j}} -- (0,0.5);
      \draw (0.6,0.25) node {\regionlabel{\lambda}};
    \end{tikzpicture}
  \]
  are always zero for $|i-j| \le 1$.  In other words, nonzero diagrams can only have strands crossing if their colors differ by at least two.  This is because $\lambda \boxplus i \boxplus j = \bzero$ when $|i-j| \le 1$.  Therefore, relations \eqref{rel:up-up-double-cross}--\eqref{rel:up-down-double-cross} allow us to resolve all nonzero double crossings.
\end{rem}

%-------------------------------------------------
\subsection{Truncated categorified quantum groups} \label{subsec:truncated-KL}
%-------------------------------------------------

In \cite[\S2]{CL15}, Cautis and Lauda associated a graded additive linear 2-category to a Cartan datum and choice of scalars, generalizing the definition of \cite{KL3}.  We let $\sU$ be the 2-category defined in \cite[\S2]{CL15} for the Cartan datum of type $A_\infty$ and the choice of scalars
\begin{equation} \label{eq:CL-scalar-choice}
  t_{ij} = 1,\ s_{ij}^{pq} = 0,\ r_i=1, \quad \text{for all } i,j,p,q,
\end{equation}
except that we enlarge the 2-category by allowing finite formal direct sums of objects.  By \cite[Rem.~(3), p.~210]{CL15}, we in fact lose no generality in making the choices \eqref{eq:CL-scalar-choice}.  The set of objects of $\sU$ is the free monoid generated by the weight lattice of $\fsl_\infty$.  We define $\sU^\trunc$ to be the quotient of $\sU$ by the identity 2-morphisms of the identity 1-morphisms of all objects corresponding to weights that do not appear in the basic representation.  Since the weights of the basic representation are in natural bijection with partitions (see Section~\ref{subsec:basic-rep}), the objects of $\sU^\trunc$ can be identified with elements of the free monoid $\N[\cP]$ on the set of partitions.

\begin{prop}
  All the 2-morphism spaces of $\sU^\trunc$ are nonnegatively graded.  The positive degree 2-morphism spaces are spanned by diagrams with dots.
\end{prop}

\begin{proof}
  The degrees of the 2-morphisms in $\sU$ are given by the degrees of the crossings, cups, caps, and dots in \cite[Def.~1.1]{CL15}.  The crossing has degree
  \[
    \deg \left(
    \begin{tikzpicture}[>=stealth,baseline={([yshift=-.5ex]current bounding box.center)}]
      \draw[->] (0,0) node [anchor=north] {\strandlabel{i}} -- (0.5,0.5);
      \draw[->] (0.5,0) node [anchor=north] {\strandlabel{j}} -- (0,0.5);
      \draw (0.6,0.25) node {\regionlabel{\lambda}};
    \end{tikzpicture}
    \right)
    =
    \begin{cases}
      -2, &\text{if } i=j, \\
      1, &\text{if } |i-j|=1,\\
      0, &\text{otherwise}.
    \end{cases}
  \]
  However, the leftmost region above is labeled $\lambda \boxminus j \boxminus i$, which is always isomorphic to $\bzero$ (i.e.\ is not a weight of the basic representation) when $|i -j| \le 1$.  (Note that the upward oriented strands in $\sU$ correspond to subtracting boxes under the bijection \eqref{eq:partions-label-basic-rep-weights}, as opposed to adding boxes, as in $\sA$.)  Therefore, the only crossings which are nonzero in $\sU^\trunc$ have degree zero.  The situation for downwards oriented crossings is analogous.

  Similarly, the right cup
  
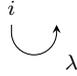
\[
    \begin{tikzpicture}[anchorbase]
      \draw[->] (0,0) node[anchor=south] {\strandlabel{i}} -- (0,-0.1) arc (180:360:.3) -- (0.6,0);
      \draw (0.8,-0.5) node {\regionlabel{\lambda}};
    \end{tikzpicture}
  \]
  has degree $1 + (\lambda,\alpha_i)$.  However, this cup is zero in $\sU^\trunc$ unless $\lambda$ has a removable $i$-box, in which case $(\lambda,\alpha_i)=-1$.  Thus, the only right cups that are nonzero in $\sU^\trunc$ have degree zero.  The situation for the other cups and caps is analogous.

  The result now follows from the fact that a dot on an $i$-colored strand has degree $(\alpha_i,\alpha_i)=2$.
\end{proof}

Let $\sU_0$ be the additive linear 2-category defined as follows.  The objects of $\sU_0$ are the same as the objects of $\sU^\trunc$.  The 1-morphisms in $\sU_0$ are formal direct sums of compositions of the generating 1-morphisms given in \cite[Def.~1.1]{CL15} \emph{without degree shifts}.  The 2-morphism spaces of $\sU_0$ are the degree zero part of the corresponding 2-morphism spaces of $\sU^\trunc$ (equivalently, the quotient of the corresponding 2-morphism spaces of $\sU^\trunc$ by the ideal consisting of 2-morphisms of strictly positive degree).

\begin{rem}
  It will follow from Theorem~\ref{theo:sA-tsU-equivalence} and Corollary~\ref{cor:KrullSchmidt} that $\sU_0$ is idempotent complete; hence there is no need to pass to the idempotent completion, as is done in \cite{CL15} with the larger category $\sU$.
\end{rem}

\begin{theo} \label{theo:sA-tsU-equivalence}
  The 2-categories $\sA$ and $\sU_0$ are equivalent via a 2-functor that acts on 2-morphisms by reversing the orientation of strands.
\end{theo}

\begin{proof}
  The sets of objects and 1-morphisms of $\sA$ and $\sU_0$ are clearly the same.  Furthermore, the spaces of 2-morphisms both consist of string diagrams with strands labeled by integers.  Therefore, it suffices to check that the local relations are the same.
  \details{
    In the language of \cite{CL15}, we have
    \begin{itemize}
      \item $I = \Z$
      \item $(\alpha_i,\alpha_j) = 2\delta_{i,j} - \delta_{i,j-1} - \delta_{i,j+1}$
      \item $d_{i,j} = -(\alpha_i,\alpha_j)$
      \item $d_i = 1$ for all $i \in I$
      \item $q_i = q$ for all $i \in I$
      \item $[n]_i = [n]$ for all $i \in I$
    \end{itemize}
  }
  The relations of \cite[\S2.2]{CL15} correspond in $\sA$ to the fact that we consider diagrams up to isotopy.  Relations \cite[(2.8)]{CL15} and \cite[(2.9)]{CL15} become trivial since they involve strands of the same color crossing or parallel strands of the same color, which yield the zero 2-morphism in the truncation.  Relation \cite[(2.10)]{CL15} corresponds to \eqref{rel:up-up-double-cross}.  Note that the $(\alpha_i,\alpha_j) \ne 0$ case of \cite[(2.10)]{CL15} is trivial in the truncation since strands of color $i$ and $i+1$ cannot cross.  Relation \cite[(2.12)]{CL15} is not relevant since it involves dots.  Relation \cite[(2.13)]{CL15} becomes \eqref{rel:up-up-up-braid}.  Note that \cite[(2.14)]{CL15} becomes trivial in the truncation since $(\alpha_i,\alpha_j) < 0$ implies $|i-j|=1$, in which case we have two strands of the same color crossing.

  The relations \cite[(2.16)]{CL15} correspond to \eqref{rel:down-up-double-cross} and \eqref{rel:up-down-double-cross}.  For $\lambda \in X$ and $i \in \Z$, we have
  \[
    \langle i, \lambda \rangle =  (\alpha_i,\lambda) =
    \begin{cases}
      1 & \text{if } i \in B^+(\lambda),\\
      -1 & \text{if } i \in B^-(\lambda),\\
      0 & \text{otherwise}.
    \end{cases}
  \]
  Therefore, relations \cite[(2.17)]{CL15} become trivial for us since the conditions on $m$ are never satisfied.  On the other hand, the relations at the top of page 211 of \cite{CL15} correspond to \eqref{rel:clockwise-i-circle} and \eqref{rel:ccc}.

  Now consider the extended $\fsl_2$ relations of \cite[\S2.6]{CL15}.  Here one considers relations involving strands of some fixed color $i$.  In the truncation, the region labels in \cite[\S2.6]{CL15} of nonzero diagrams are 0 (corresponding to a region label $\lambda$ with no $i$-addable or $i$-removable boxes), 1 (corresponding to a $\lambda$ with an $i$-removable box but no $i$-addable box), or $-1$ (corresponding to a $\lambda$ with an $i$-addable box, but no $i$-removable box).  The relations \cite[(2.21)]{CL15} become trivial in the truncation since a strand crossing itself is zero.  It follows from \cite[(2.19)]{CL15} with $n=1$ and $j=0$, together with \eqref{rel:clockwise-i-circle} (which tells us that a clockwise circle in outer region 1 is equal to one) that a counterclockwise fake bubble with dot label $-2$ and outer region 1 is equal to 1.  Thus, the first relation of \cite[(2.22)]{CL15} becomes \eqref{rel:down-up-ii}.  The second relation in \cite[(2.22)]{CL15} becomes zero in the truncation, since $n>0$ implies that $n$ has no $i$-addable box.  The relations \cite[(2.23)]{CL15} are trivial in the truncation since they involve a strand crossing itself.  The first relation in \cite[(2.24)]{CL15} is trivial in the truncation, while the second relation in \cite[(2.24)]{CL15} becomes \eqref{rel:up-down-ii}.  Finally, relations \cite[(2.25), (2.26)]{CL15} are trivial in the truncation.
\end{proof}

%-----------------------------
\subsection{1-morphism spaces} \label{subsec:A-1morph}
%-----------------------------

\begin{lem} \label{lem:key-1mor-isoms}
  For $\lambda \in \cP$ and $i,j \in \Z$, we have
  \begin{gather}
    \sE_i \sE_j 1_\lambda \cong \sE_j \sE_i 1_\lambda,\quad \text{if } |i-j| > 1, \label{key-isom:EiEj} \\
    \sF_i \sF_j 1_\lambda \cong \sF_j \sF_i 1_\lambda,\quad \text{if } |i-j| > 1, \label{key-isom:FiFj} \\
    \sE_i \sF_j 1_\lambda \cong \sF_j \sE_i 1_\lambda,\quad \text{if } i \ne j, \label{key-isom:EiFj} \\
    \sE_i \sF_i 1_\lambda \cong 1_\lambda,\quad \text{if } i \in B^+(\lambda), \label{key-isom:EiFi} \\
    \sF_i \sE_i 1_\lambda \cong 1_\lambda,\quad \text{if } i \in B^-(\lambda). \label{key-isom:FiEi}
  \end{gather}
\end{lem}

\begin{proof}
  The isomorphisms \eqref{key-isom:EiEj} and \eqref{key-isom:FiFj} follow immediately from \eqref{rel:up-up-double-cross}, isomorphism \eqref{key-isom:EiFj} follows immediately from \eqref{rel:down-up-double-cross} and \eqref{rel:up-down-double-cross}, isomorphism \eqref{key-isom:EiFi} follows from \eqref{rel:down-up-ii} and \eqref{rel:ccc}, and \eqref{key-isom:FiEi} follows from \eqref{rel:up-down-ii} and \eqref{rel:clockwise-i-circle}.
\end{proof}

\begin{prop} \label{prop:A-1-morph-spaces}
  We have
  \[
    \dim \Mor_{K(\sA)}(\lambda,\mu) = 1 \quad \text{for all } \lambda,\mu \in \cP.
  \]
  In particular, for $\lambda,\mu \in \cP$, $\Mor_{K(\sA)}(\lambda,\mu)$ is spanned by the class of a single 1-morphism in $\sA$ of the form
  \begin{equation} \label{eq:1-morph-normal-form}
    \sF_{i_1} \sF_{i_2} \dotsm \sF_{i_k} \sE_{j_1} \sE_{j_2} \dotsm \sE_{j_\ell} 1_\lambda,
  \end{equation}
  where $\{i_1,\dotsc,i_k\} \cap \{j_1,\dotsc,j_\ell\} = \varnothing$ and $\mu = \lambda \boxminus j_\ell \dotsb \boxminus j_2 \boxminus j_1 \boxplus i_k \boxplus \dotsb \boxplus i_2 \boxplus i_1$.
\end{prop}

\begin{proof}
  The proof of this statement is analogous to the proof of Proposition~\ref{prop:dotU-presentation}.
\end{proof}

%-----------------------------
\subsection{2-morphism spaces}
%-----------------------------

\begin{lem}[Triple point moves] \label{lem:triple-point}
  Suppose $i,j,k \in \Z$.  The relation
  \begin{equation} \label{eq:triple-point}
    \begin{tikzpicture}[>=stealth,baseline={([yshift=.5ex]current bounding box.center)}]
      \draw (0,0) node[anchor=north] {\strandlabel{i}} -- (1,1);
      \draw (1,0) node[anchor=north] {\strandlabel{k}} -- (0,1);
      \draw (0.5,0) node[anchor=north] {\strandlabel{j}} .. controls (0,0.5) .. (0.5,1);
    \end{tikzpicture}
    \ = \
    \begin{tikzpicture}[>=stealth,baseline={([yshift=.5ex]current bounding box.center)}]
      \draw (0,0) node[anchor=north] {\strandlabel{i}} -- (1,1);
      \draw (1,0) node[anchor=north] {\strandlabel{k}} -- (0,1);
      \draw (0.5,0) node[anchor=north] {\strandlabel{j}} .. controls (1,0.5) .. (0.5,1);
    \end{tikzpicture}
  \end{equation}
  holds for all possible orientations of the strands.
\end{lem}

\begin{proof}
  The case where all strands are pointed up is \eqref{rel:up-up-up-braid}, and then the case where all strands are pointed down follows by isotopy invariance.  We compute
  \[
    \begin{tikzpicture}[scale=0.5,>=stealth,baseline={([yshift=.5ex]current bounding box.center)}]
      \draw[<-] (0,0) node[anchor=north] {\strandlabel{i}} -- (2,2);
      \draw[->] (2,0) node[anchor=north] {\strandlabel{k}} -- (0,2);
      \draw[<-] (1,2) .. controls (0,1) .. (1,0) node[anchor=north] {\strandlabel{j}};
    \end{tikzpicture}
    \ \stackrel{\eqref{rel:up-up-double-cross}}{=} \
    \begin{tikzpicture}[scale=0.5,>=stealth,baseline={([yshift=.5ex]current bounding box.center)}]
      \draw[<-] (0,0) node[anchor=north] {\strandlabel{i}} -- (2,2);
      \draw[->] (2,0) node[anchor=north] {\strandlabel{k}} -- (0,2);
      \draw[<-] (1,2) .. controls (0,1) and (0.5,-.1) .. (1.3,.7) .. controls (2.1,1.5) and (2.3,.8) .. (1,0) node[anchor=north] {\strandlabel{j}};
    \end{tikzpicture}
    \stackrel{\eqref{rel:up-up-up-braid}}{=} \
    \begin{tikzpicture}[scale=0.5,>=stealth,baseline={([yshift=.5ex]current bounding box.center)}]
      \draw[<-] (0,0) node[anchor=north] {\strandlabel{i}} -- (2,2);
      \draw[->] (2,0) node[anchor=north] {\strandlabel{k}} -- (0,2);
      \draw[<-] (1,2) .. controls (0.2,1.5) and (-0.1,.5) .. (0.8,1.2) .. controls (1.7,2.1) and (2,1) .. (1,0) node[anchor=north] {\strandlabel{j}};
    \end{tikzpicture}
    \stackrel{\eqref{rel:up-up-double-cross}}{=} \
    \begin{tikzpicture}[scale=0.5,>=stealth,baseline={([yshift=.5ex]current bounding box.center)}]
      \draw[<-] (0,0) node[anchor=north] {\strandlabel{i}} -- (2,2);
      \draw[->] (2,0) node[anchor=north] {\strandlabel{k}} -- (0,2);
      \draw[<-] (1,2) .. controls (2,1) .. (1,0) node[anchor=north] {\strandlabel{j}};
    \end{tikzpicture}\ .
  \]
  We omit proofs of the other cases, which are analogous.
  \details{
    We have
    \[
      \begin{tikzpicture}[scale=0.5,>=stealth,baseline={([yshift=-.5ex]current bounding box.center)}]
        \draw[->] (0,0) node[anchor=north] {\strandlabel{i}} to (2,2);
        \draw[->] (2,0) node[anchor=north] {$k$} to (0,2);
        \draw[->] (1,2) .. controls (0,1) .. (1,0) node[anchor=north] {\strandlabel{j}};
      \end{tikzpicture}
      \ \stackrel{\eqref{rel:down-up-double-cross}}{=}\
      \begin{tikzpicture}[scale=0.5,>=stealth,baseline={([yshift=-.5ex]current bounding box.center)}]
        \draw[->] (3.5,0) node[anchor=north] {\strandlabel{i}} to (5.5,2);
        \draw[->] (5.5,0) node[anchor=north] {$k$} to (3.5,2);
        \draw[->] (4.5,2) .. controls (3.5,1) and (4,-.1) .. (4.8,.7) ..
        controls (5.6,1.5) and (5.8,.8) .. (4.5,0) node[anchor=north] {\strandlabel{j}};
      \end{tikzpicture}
      = \
      \begin{tikzpicture}[scale=0.5,>=stealth,baseline={([yshift=-.5ex]current bounding box.center)}]
        \draw[->] (3.5,0) node[anchor=north] {\strandlabel{i}} to (5.5,2);
        \draw[->] (5.5,0) node[anchor=north] {$k$} to (3.5,2);
        \draw[->] (4.5,2) .. controls (3.7,1.5) and (3.4,.5) .. (4.3,1.2) ..
        controls (5.2,2.1) and (5.5,1) .. (4.5,0) node[anchor=north] {\strandlabel{j}};
      \end{tikzpicture}
      \stackrel{\eqref{rel:up-down-double-cross}}{=} \
      \begin{tikzpicture}[scale=0.5,>=stealth,baseline={([yshift=-.5ex]current bounding box.center)}]
        \draw[->] (1,0) node[anchor=north] {\strandlabel{i}} to (3,2);
        \draw[->] (3,0) node[anchor=north] {$k$} to (1,2);
        \draw[->] (2,2) .. controls (3,1) .. (2,0) node[anchor=north] {\strandlabel{j}};
      \end{tikzpicture}\ ,
    \]
    where the second equality comes from the above case.  Then
    \[
      \begin{tikzpicture}[scale=0.5,>=stealth,baseline={([yshift=-.5ex]current bounding box.center)}]
        \draw[<-] (0,0) node[anchor=north] {\strandlabel{i}} to (2,2);
        \draw[->] (2,0) node[anchor=north] {$k$} to (0,2);
        \draw[->] (1,2) .. controls (0,1) .. (1,0) node[anchor=north] {\strandlabel{j}};
      \end{tikzpicture}
      \ \stackrel{\eqref{rel:down-up-double-cross}}{=}\
      \begin{tikzpicture}[scale=0.5,>=stealth,baseline={([yshift=-.5ex]current bounding box.center)}]
        \draw[<-] (3.5,0) node[anchor=north] {\strandlabel{i}} to (5.5,2);
        \draw[->] (5.5,0) node[anchor=north] {$k$} to (3.5,2);
        \draw[->] (4.5,2) .. controls (3.5,1) and (4,-.1) .. (4.8,.7) ..
        controls (5.6,1.5) and (5.8,.8) .. (4.5,0) node[anchor=north] {\strandlabel{j}};
      \end{tikzpicture}
      = \
      \begin{tikzpicture}[scale=0.5,>=stealth,baseline={([yshift=-.5ex]current bounding box.center)}]
        \draw[<-] (3.5,0) node[anchor=north] {\strandlabel{i}} to (5.5,2);
        \draw[->] (5.5,0) node[anchor=north] {$k$} to (3.5,2);
        \draw[->] (4.5,2) .. controls (3.7,1.5) and (3.4,.5) .. (4.3,1.2) ..
        controls (5.2,2.1) and (5.5,1) .. (4.5,0) node[anchor=north] {\strandlabel{j}};
      \end{tikzpicture}
      \stackrel{\eqref{rel:up-down-double-cross}}{=} \
      \begin{tikzpicture}[scale=0.5,>=stealth,baseline={([yshift=-.5ex]current bounding box.center)}]
        \draw[<-] (1,0) node[anchor=north] {\strandlabel{i}} to (3,2);
        \draw[->] (3,0) node[anchor=north] {$k$} to (1,2);
        \draw[->] (2,2) .. controls (3,1) .. (2,0) node[anchor=north] {\strandlabel{j}};
      \end{tikzpicture}\ ,
    \]
    where the second equality comes from the previous case.  Together with isotopy invariance, this covers all possible orientations of the strands.
  }
\end{proof}

\begin{lem} \label{lem:no-bubbles}
  For all $\lambda \in \cP$, we have $\2Mor_\sA (1_\lambda,1_\lambda) \cong \Q \id_\lambda$.  In other words, all closed diagrams in a region labeled $\lambda$ are isomorphic to some multiple of the empty diagram.
\end{lem}

\begin{proof}
  Using the local relations, closed diagrams can be written as linear combinations of nested circles.  By \eqref{rel:clockwise-i-circle} and \eqref{rel:ccc}, these are equal to multiples of the empty diagram.
\end{proof}

By definition, the 1-morphisms of $\sA$ are sequences of $\sE_i$'s and $\sF_i$'s, followed by $1_\lambda$ for some $\lambda \in \cP$.  If we think of such a sequence as a row of colored arrows, a down $i$-colored arrow for each $\sE_i$ and an up $i$-colored arrow for each $\sF_i$, then the 2-morphisms between two 1-morphisms are strands connecting the arrows in such a way that the color and orientation of each strand agrees with its two endpoints.  (We use Lemma~\ref{lem:no-bubbles} here to ignore closed diagrams.)  By \eqref{rel:up-up-double-cross}, \eqref{rel:down-up-double-cross}, and \eqref{rel:up-down-double-cross}, we may simplify such a diagram so that it contains no double crossings.  Then, by isotopy invariance and Lemma~\ref{lem:triple-point}, we see that the 2-morphism is uniquely determined by the matching of arrows induced by the strands.  Furthermore, since strands of the same color cannot cross, the 2-morphism is determined by a crossingless matching of $i$-colored arrows for each $i \in \Z$.  We call such a collection of crossingless matchings a \emph{colored matching}.  An example of such a colored matching is the following, where $i$, $j$, $k$, and $\ell$ are pairwise distinct.
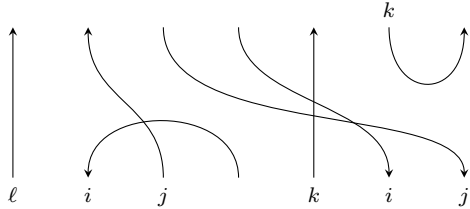
\[
  \begin{tikzpicture}[anchorbase]
    \draw[->] (0,0) node[anchor=north] {\strandlabel{\ell}} -- (0,2);
    \draw[<-] (1,0) node[anchor=north] {\strandlabel{i}} .. controls (1,1) and (3,1) .. (3,0);
    \draw[->] (2,0) node[anchor=north] {\strandlabel{j}} .. controls (2,1) and (1,1) .. (1,2);
    \draw[->] (4,0) node[anchor=north] {\strandlabel{k}} -- (4,2);
    \draw[<-] (5,0) node[anchor=north] {\strandlabel{i}} .. controls (5,1) and (3,1) .. (3,2);
    \draw[<-] (6,0) node[anchor=north] {\strandlabel{j}} .. controls (6,1) and (2,0.5) .. (2,2);
    \draw[->] (5,2) node[anchor=south] {\strandlabel{k}} .. controls (5,1) and (6,1) .. (6,2);
  \end{tikzpicture}
\]
Note that, since strands of colors $i$ and $j$ intersect, we must have $|i-j|>1$ in order for the 2-morphism to be nonzero, by Remark~\ref{rem:double-cross-diff-by-one}.  Similarly, we have $|i-k| > 1$ and $|j-k|>1$.  However, it is possible that $|k-\ell|=1$.

\begin{prop} \label{prop:matching-2-morphisms}
  Suppose $\lambda, \mu \in \cP$, and that $\mathsf{P}, \mathsf{Q} \in \1Mor_\sA(\lambda,\mu)$ are two nonzero 1-morphisms that are sequences of $\sE_i$'s and $\sF_i$'s (followed by $1_\lambda$).  Then every nonzero 2-morphism from $\mathsf{P}$ to $\mathsf{Q}$ is an isomorphism.  In particular, $\1Mor_\sA(\mathsf{P},\mathsf{Q})$ is one-dimensional.
\end{prop}

\begin{proof}
  It is easy to see that there is always at least one colored matching.  Repeated use of relations \eqref{rel:down-up-ii} and \eqref{rel:up-down-ii} allows us to see that any two crossingless matchings of $i$-colored arrows as described above are equal, up to scalar multiple.

  Now let $\alpha$ be a nonzero 2-morphism from $\mathsf{P}$ to $\mathsf{Q}$.  By the above argument, $\alpha$ is a multiple of a colored matching from $\mathsf{P}$ to $\mathsf{Q}$.  Let $\beta$ be a multiple of a colored matching from $\mathsf{Q}$ to $\mathsf{P}$ and consider the composition $\alpha \circ \beta$.  We can resolve double crossings by \eqref{rel:up-up-double-cross}, \eqref{rel:down-up-double-cross}, and \eqref{rel:up-down-double-cross}.  Any circles may be slid into open regions (so that they do not intersect any other strands) using \eqref{rel:up-up-double-cross}, \eqref{rel:down-up-double-cross}, and \eqref{rel:up-down-double-cross} and then removed using \eqref{rel:clockwise-i-circle} and \eqref{rel:ccc}.  Thus, $\alpha \circ \beta$ is a nonzero multiple of a colored matching.  As above, it must be a multiple of the identity matching.  Similarly $\beta \circ \alpha$ is a multiple of the identity matching.
\end{proof}

\begin{cor} \label{cor:KrullSchmidt}
  The 2-category $\sA$ is Krull--Schmidt.  More precisely, for any two objects of $\sA$, the morphism category between these two objects in Krull--Schmidt.
\end{cor}

\begin{proof}
  It follows from Proposition~\ref{prop:A-1-morph-spaces} that every 1-morphism in $\sA$ is a multiple of a 1-morphism of the form \eqref{eq:1-morph-normal-form}.  Then the result follows from Proposition~\ref{prop:matching-2-morphisms}.
\end{proof}

%------------------------------
\subsection{Decategorification}
%------------------------------

We now state one of our main results.

\begin{theo} \label{theo:sA-categorifies-cA}
  The functor $\cA \to K(\sA)$ that is the identity on objects and, on 1-morphisms, is uniquely determined by
  \[
    e_i 1_\lambda \mapsto [\sE_i 1_\lambda],\quad
    f_i 1_\lambda \mapsto [\sF_i 1_\lambda],
  \]
  is an isomorphism.  In other words, $\sA$ categorifies $\cA$.
\end{theo}

\begin{proof}
  The fact that the functor is well-defined follows from Lemma~\ref{lem:key-1mor-isoms}.  It is surjective since the images in the Grothendieck group of all the generating 1-morphisms $\sE_i 1_\lambda$ and $\sF_i 1_\lambda$ are in the image of the functor.  Injectivity will be proven in Corollary~\ref{cor:dotU0-KA-injective}.
\end{proof}

%%%%%%%%%%%%%%%%%%%%%%%%%%%%%%%%%%%%%%%%%%%%%%
%
\section{The 2-category $\sH^\trunc$} \label{sec:sH-def}
%
%%%%%%%%%%%%%%%%%%%%%%%%%%%%%%%%%%%%%%%%%%%%%%

%----------------------
\subsection{Definition} \label{subsec:sH-def}
%----------------------

We introduce here a 2-category based on the diagrammatic monoidal category introduced by Khovanov in \cite{Kho14}.  We begin by defining an additive linear 2-category ${\sH^\trunc}'$.  The set of objects of ${\sH^\trunc}'$ is the free monoid $\N[\N]$ on $\N$, where $\bzero$ is a zero object.  The set of 1-morphisms of ${\sH^\trunc}'$ is generated by
\[
  \sQ_+ 1_k = 1_{k+1} \sQ_+ = 1_{k+1} \sQ_+ 1_k
  \quad \text{and} \quad
  \sQ_- 1_{k+1} = 1_k \sQ_- = 1_k \sQ_- 1_{k+1},
  \quad k \in \N.
\]
In other words, if we let $\sQ_c := \sQ_{c_1} \otimes \dotsb \otimes \sQ_{c_\ell}$ for a finite sequence $c = c_1 \cdots c_\ell$ of $+$ and $-$ signs, then the 1-morphisms of ${\sH^\trunc}'$ from $n$ to $m$ are finite direct sums of $\sQ_c 1_n$ for $c = c_1 \cdots c_\ell$ satisfying
\[
  m - n = \# \{i \mid c_i = +\} - \# \{i \mid c_i = -\}.
\]
If, for some $1 \le k < \ell$, we have
\[
  n + \# \{1 \le i \le k \mid c_i = +\} - \# \{1 \le i \le k \mid c_i = -\} < 0,
\]
then $\sQ_c 1_n$ is the zero morphism.  In other words, we view negative integers as the zero object $\bzero$.

The space of morphisms between two objects is the $\Q$-algebra generated by suitable planar diagrams.  The diagrams consist of oriented compact one-manifolds immersed into the plane strip $\R \times [0,1]$ modulo certain local relations, with regions of the strip labeled by nonnegative integers.  The 2-morphism that is the identity on $\sQ_+ 1_k$ will be denoted by an upward strand, where the region to the right of the arrow is labeled $k$ (and the region to the left has label $k+1$), while the identity on $\sQ_- 1_k$ will be denoted by a downward strand, where the region to the right of the strand is labeled $k$ (and the region to the left has label $k-1$):
\[
  \begin{tikzpicture}[>=stealth,baseline={([yshift=-.5ex]current bounding box.center)}]
    \draw[->] (0,0) -- (0,1);
    \draw (0.3,0.5) node {\regionlabel{k}};
  \end{tikzpicture}
  \qquad \qquad \qquad
  \begin{tikzpicture}[>=stealth,baseline={([yshift=-.5ex]current bounding box.center)}]
    \draw[<-] (0,0) -- (0,1);
    \draw (0.3,0.5) node {\regionlabel{k}};
  \end{tikzpicture}
\]
The endpoints of the strings are located at $\{1,\dotsc,m\} \times \{0\}$ and $\{1,\dotsc,k\} \times \{1\}$, where $m$ and $k$ are the lengths of the sequences $c$ and $c'$ respectively.
The local relations are as follows (for any compatible labeling of the regions).

\noindent\begin{minipage}{0.5\linewidth}
  \begin{equation} \label{eq:Kho-rel-braid}
    \begin{tikzpicture}[anchorbase]
      \draw (0,0) -- (1,1)[->];
      \draw (1,0) -- (0,1)[->];
      \draw[->] (0.5,0) .. controls (0,0.5) .. (0.5,1);
    \end{tikzpicture}
    \ =\
    \begin{tikzpicture}[anchorbase]
      \draw (0,0) -- (1,1)[->];
      \draw (1,0) -- (0,1)[->];
      \draw[->] (0.5,0) .. controls (1,0.5) .. (0.5,1);
    \end{tikzpicture}
  \end{equation}
\end{minipage}%
\begin{minipage}{0.5\linewidth}
  \begin{equation} \label{eq:Kho-rel-transposition-squared}
    \begin{tikzpicture}[anchorbase]
      \draw[->] (0,0) .. controls (0.5,0.5) .. (0,1);
      \draw[->] (0.5,0) .. controls (0,0.5) .. (0.5,1);
    \end{tikzpicture}
    \ =\
    \begin{tikzpicture}[anchorbase]
      \draw[->] (0,0) --(0,1);
      \draw[->] (0.5,0) -- (0.5,1);
    \end{tikzpicture}
  \end{equation}
\end{minipage}\par\vspace{\belowdisplayskip}

\noindent\begin{minipage}{0.5\linewidth}
  \begin{equation} \label{eq:Kho-rel-down-up-double-cross}
    \begin{tikzpicture}[anchorbase]
      \draw[<-] (0,0) .. controls (0.5,0.5) .. (0,1);
      \draw[->] (0.5,0) .. controls (0,0.5) .. (0.5,1);
    \end{tikzpicture}
    \ =\
    \begin{tikzpicture}[anchorbase]
      \draw[<-] (0,0) --(0,1);
      \draw[->] (0.5,0) -- (0.5,1);
    \end{tikzpicture}
    \ -\
    \begin{tikzpicture}[anchorbase]
      \draw[->] (0,1) -- (0,0.9) arc (180:360:.25) -- (0.5,1);
      \draw[<-] (0,0) -- (0,0.1) arc (180:0:.25) -- (0.5,0);
    \end{tikzpicture}
  \end{equation}
\end{minipage}%
\begin{minipage}{0.5\linewidth}
  \begin{equation} \label{eq:Kho-rel-up-down-double-cross}
    \begin{tikzpicture}[anchorbase]
      \draw[->] (0,0) .. controls (0.5,0.5) .. (0,1);
      \draw[<-] (0.5,0) .. controls (0,0.5) .. (0.5,1);
    \end{tikzpicture}
    \ =\
    \begin{tikzpicture}[anchorbase]
      \draw[->] (0,0) --(0,1);
      \draw[<-] (0.5,0) -- (0.5,1);
    \end{tikzpicture}
  \end{equation}
\end{minipage}\par\vspace{\belowdisplayskip}

\noindent\begin{minipage}{0.5\linewidth}
  \begin{equation}\label{eq:Kho-rel-ccc}
    \begin{tikzpicture}[anchorbase]
      \draw [->](0,0) arc (0:360:0.3);
    \end{tikzpicture}
    \ = 1
  \end{equation}
\end{minipage}%
\begin{minipage}{0.5\linewidth}
  \begin{equation}\label{eq:Kho-rel-left-curl}
    \begin{tikzpicture}[scale=0.5,anchorbase]
      \draw (0,0) .. controls (0,.5) and (.7,.5) .. (.9,0);
      \draw (0,0) .. controls (0,-.5) and (.7,-.5) .. (.9,0);
      \draw (1,-1) .. controls (1,-.5) .. (.9,0);
      \draw[->] (.9,0) .. controls (1,.5) .. (1,1);
    \end{tikzpicture}
    \ =0
  \end{equation}
\end{minipage}\par\vspace{\belowdisplayskip}

\noindent By convention, any string diagram containing a region labeled by a negative integer is the zero 2-morphism.  This is compatible with our convention above for 1-morphisms.

\begin{defin}[Idempotent completion of a 2-category] \label{def:idem-completion}
  For a 2-category $\sC$, we define an idempotent completion $\Kar(\sC)$ as follows.
  \begin{itemize}
    \item The objects of $\Kar(\sC)$ are triples $(x,\se,\epsilon)$, where $x \in \Ob \sC$, $\se$ is an idempotent 1-morphism of $x$ in $\sC$, and $\epsilon$ is an idempotent 2-morphism (under vertical composition) of $\se$ in $\sC$.

    \item The 1-morphisms of $\Kar(\sC)$ between objects $(x,\se,\epsilon)$ and $(x',\se',\epsilon')$ are pairs $(\mathsf{g},\beta)$, where $\mathsf{g} \colon x \to x'$ is a 1-morphism in $\sC$ such that $\se' \mathsf{g} \se = \mathsf{g}$ and $\beta \colon \mathsf{g} \to \mathsf{g}$ is an idempotent 2-morphism in $\sC$ such that $\epsilon' \beta \epsilon = \beta$.

    \item The 2-morphisms between parallel 1-morphisms $(\mathsf{g},\beta),(\mathsf{h},\gamma) \colon (x,\se,\epsilon) \to (x',\se',\epsilon')$ are 2-morphisms $\alpha \colon \mathsf{g} \to \mathsf{h}$ in $\sC$ such that $\gamma \circ \alpha \circ \beta = \alpha$.
  \end{itemize}
  Composition of 1-morphisms is pairwise composition of 1-morphisms as in $\cC$ and composition of 2-morphisms is as in $\sC$.
\end{defin}

If $\sC$ is a 2-category, we have a natural inclusion of $\sC$ into $\Kar(\sC)$ sending the object $x$ to $(x,1_x,\id_x)$ and the 1-morphism $\mathsf{x}$ to $(\mathsf{x},\id_\mathsf{x})$.  The idempotent completion of a 2-category $\sC$ is universal in the sense that any 2-functor $\sC \to \mathscr{D}$ to a 2-category $\mathscr{D}$ in which all idempotent 1-morphisms and idempotent 2-morphisms split factors through a 2-functor $\Kar(\sC) \to \mathscr{D}$.

We then define
\[
  \sH^\trunc = \Kar ({\sH^\trunc}').
\]

\begin{rem} \label{rem:idem-completion}
  \begin{enumerate}
    \item In the 2-category ${\sH^\trunc}'$, the only idempotent 1-morphisms are the identity 1-morphisms.  However, we state Definition~\ref{def:idem-completion} in full generality.

    \item Note that Definition~\ref{def:idem-completion} differs from the definition of the idempotent completion for 2-categories often considered in the categorification literature (e.g.\ in \cite[Def.~3.21]{KL3}), where one takes only the usual Karoubi envelope (in the sense of 1-categories) of the morphism categories.   Even in the case where the only idempotent 1-morphisms are the identity 1-morphisms (as for ${\sH^\trunc}'$), the idempotent completion $\Kar(\sC)$ of Definition~\ref{def:idem-completion} often has more objects than $\sC$, since $\sC$ may have idempotent 2-morphisms of the identity 1-morphisms.  We will see in Section~\ref{subsec:sH-decomp} that this is indeed the case for $\sH^\trunc$.  Note that when the only idempotent 2-morphisms of identity 1-morphisms are the identity 2-morphisms, as is the case in \cite{KL3,Kho14}, the two notions of idempotent completions of 2-categories agree.  To the best of our knowledge, the more general definition of idempotent completions of 2-categories given above has not previously appeared in the categorification literature.
  \end{enumerate}
\end{rem}

\begin{rem} \label{rem:fH-discussion}
  If we repeat the construction of this subsection, but begin with the set $\N[\Z]$ of objects instead of $\N[\N]$ and do not set diagrams with negative region labels to be zero, we obtain a 2-category $\sH$ that is a 2-category version of Khovanov's monoidal category.  (See \cite[Def.~4.8]{LS13}, setting $q=1$.)  Thus, $\sH^\trunc$ can be viewed as a truncation of $\sH$, where we force negative objects to be zero. (More precisely, we quotient by 2-morphisms that are diagrams with negative region labels.)  So we have a
  natural truncation 2-functor
  \begin{equation} \label{eq:truncation-fH-to-sH}
    \sH \to \sH^\trunc
  \end{equation}
  sending negative elements of $\Z$ to zero.   We will see in Section~\ref{subsec:sH-decomp} that the truncation leads to the presence of many more idempotent 2-morphisms in $\sH^\trunc$ than in $\sH$.  See Remark~\ref{rem:truncation-produces-idems}.
\end{rem}

%-----------------------------------------------------------
\subsection{A decomposition $\sH^\trunc = \sH_\epsilon \oplus \sH_\delta$} \label{subsec:sH-decomp}
% ----------------------------------------------------------

In order to simplify diagrams, we will use a dashed strand to represent one or more solid strands.   The number of solid strands is uniquely determined by the region labels on either side of the dashed strand.  For example
\[
  \begin{tikzpicture}[anchorbase]
    \draw[->,dashed] (0,0) -- (0,1);
    \node at (-0.5,0.5) {\regionlabel{n+k}};
    \node at (0.3,0.5) {\regionlabel{k}};
  \end{tikzpicture}
  \ = \
  \begin{tikzpicture}[anchorbase]
    \draw[->] (0,0) -- (0,1);
    \draw[->] (0.2,0) -- (0.2,1);
    \node at (0.55,0.5) {$\cdots$};
    \draw[->] (0.9,0) -- (0.9,1);
    \node at (1.2,0.5) {\regionlabel{k}};
    \node at (-0.5,0.5) {\regionlabel{n+k}};
  \end{tikzpicture}
\]
where there are $n$ strands in the right-hand diagram.

As a result of the local relations \eqref{eq:Kho-rel-braid} and \eqref{eq:Kho-rel-transposition-squared}, for $k,n \in \N$, we have a natural algebra homomorphism
\[
  \Q S_n \to 2\End_{\sH^\trunc}(\sQ_+^n 1_k),
\]
obtained by labeling the strands $1,2,\dotsc,n$ from right to left and associating a braid-like diagram to a permutation in the natural way.  Rotating diagrams through an angle $\pi$, we obtain a natural homomorphism
\[
  (\Q S_n)^\op \to 2\End_{\sH^\trunc}(\sQ_-^n 1_k),
\]
where $A^\op$ denotes the opposite algebra of an algebra $A$.  For $z \in \Q S_n$, we will denote the corresponding elements of $2\End_{\sH^\trunc}(\sQ_+^n 1_k)$ and $2\End_{\sH^\trunc}(\sQ_-^n 1_k)$ by
\[
  \begin{tikzpicture}[anchorbase]
    \draw (-0.5,-0.25) -- (-0.5,0.25) -- (0.5,0.25) -- (0.5,-0.25) -- (-0.5,-0.25);
    \draw (0,0) node {$z$};
    \draw[->,dashed] (0,0.25) -- (0,0.7);
    \draw[dashed] (0,-0.25) -- (0,-0.7);
  \end{tikzpicture}
  \quad \text{and} \quad
  \begin{tikzpicture}[anchorbase]
    \draw (-0.5,-0.25) -- (-0.5,0.25) -- (0.5,0.25) -- (0.5,-0.25) -- (-0.5,-0.25);
    \draw (0,0) node {$z$};
    \draw[dashed] (0,0.25) -- (0,0.7);
    \draw[->,dashed] (0,-0.25) -- (0,-0.7);
  \end{tikzpicture}
\]
respectively, where there the dashed strand represents $n$ strands (determined by $z$).  It follows that
\[
  \begin{tikzpicture}[anchorbase]
    \draw (-0.5,-0.25) -- (-0.5,0.25) -- (0.5,0.25) -- (0.5,-0.25) -- (-0.5,-0.25);
    \draw (0,0) node {$z$};
    \draw[dashed] (-0.8,1) -- (-0.8,-0.25) .. controls (-0.8,-1) and (0,-1) .. (0,-0.25);
    \draw[<-,dashed] (0.8,-1) -- (0.8,0.25) .. controls (0.8,1) and (0,1) .. (0,0.25);
    \ =\
  \end{tikzpicture}
  \ =\
  \begin{tikzpicture}[anchorbase]
    \draw (-0.5,-0.25) -- (-0.5,0.25) -- (0.5,0.25) -- (0.5,-0.25) -- (-0.5,-0.25);
    \draw (0,0) node {$z$};
    \draw[dashed] (0,0.25) -- (0,1);
    \draw[->,dashed] (0,-0.25) -- (0,-1);
  \end{tikzpicture}
\]
for all $z \in \Q S_n$.

\begin{lem} \label{lem:dashed-sum-identity}
  For all $n>0$, we have
  \begin{equation} \label{eq:dashed-sum-identity}
    \begin{tikzpicture}[anchorbase]
      \draw[<-,dashed] (0,0.7) -- (0,0.5) arc(180:360:0.3) -- (0.6,0.7);
      \draw[->,dashed] (0,-0.7) -- (0,-0.5) arc(180:0:0.3) -- (0.6,-0.7);
      \node at (0.3,0.-0.5) {\regionlabel{0}};
      \node at (0.3,0.5) {\regionlabel{0}};
      \node at (-0.2,0) {\regionlabel{n}};
    \end{tikzpicture}
    = \sum_{z \in S_n}
    \begin{tikzpicture}[anchorbase]
      \draw (-0.5,-0.25) -- (-0.5,0.25) -- (0.5,0.25) -- (0.5,-0.25) -- (-0.5,-0.25);
      \draw (0,0) node {$z$};
      \draw[->,dashed] (0,0.25) -- (0,0.7);
      \draw[dashed] (0,-0.25) -- (0,-0.7);
      \draw (0.7,-0.25) -- (0.7,0.25) -- (1.7,0.25) -- (1.7,-0.25) -- (0.7,-0.25);
      \draw (1.2,0) node {$z^{-1}$};
      \draw[dashed] (1.2,0.25) -- (1.2,0.7);
      \draw[->,dashed] (1.2,-0.25) -- (1.2,-0.7);
      \node at (0.6,0.5) {\regionlabel{0}};
    \end{tikzpicture}
    \ .
  \end{equation}
\end{lem}

\begin{proof}
  This follows from repeated use of \eqref{eq:Kho-rel-down-up-double-cross}.
\end{proof}

Recall the definition of the central idempotent $e_\lambda$ in \eqref{eq:central-idempotent-explicit}.

\begin{lem} \label{lem:orthoid}
  For $n\in \N$, in $2\End_{\sH^\trunc}(1_n)$ we have the following orthogonal idempotent decomposition of $\id_n$:
  \begin{equation} \label{eq:id1n-decomp}
    \id_n = \sum_{\lambda \vdash n} \epsilon_\lambda + \delta_n,
  \end{equation}
  where, for $n >0$, we define
  \begin{equation}
    \epsilon_\lambda = \frac{1}{n!}\
    \begin{tikzpicture}[anchorbase]
      \draw (-0.5,-0.25) -- (-0.5,0.25) -- (0.5,0.25) -- (0.5,-0.25) -- (-0.5,-0.25);
      \draw (0,0) node {$e_\lambda$};
      \draw[->,dashed] (0,0.25) .. controls (0,1) and (1,1) .. (1,0.25) -- (1,-0.25) .. controls (1,-1) and (0,-1) .. (0,-0.25);
      \node at (0.75,0) {\regionlabel{0}};
    \end{tikzpicture}
    \ ,\quad
    \delta_n = 1 - \epsilon_n,\quad
    \epsilon_n = \sum_{\lambda \vdash n} \epsilon_\lambda = \frac{1}{n!}\
    \begin{tikzpicture}[anchorbase]
      \draw (-0.5,-0.25) -- (-0.5,0.25) -- (0.5,0.25) -- (0.5,-0.25) -- (-0.5,-0.25);
      \draw (0,0) node {$1_{S_n}$};
      \draw[->,dashed] (0,0.25) .. controls (0,1) and (1,1) .. (1,0.25) -- (1,-0.25) .. controls (1,-1) and (0,-1) .. (0,-0.25);
      \node at (0.75,0) {\regionlabel{0}};
    \end{tikzpicture}\ ,
  \end{equation}
  where $1_{S_n}$ denotes the identity element of $S_n$.  By convention, we define $\epsilon_0 = \epsilon_\varnothing = \id_0$ and $\delta_0=0$.
\end{lem}

\begin{proof}
  It is clear that \eqref{eq:id1n-decomp} is satisfied.  The fact that $\delta_n$, $\epsilon_\lambda$, $\lambda \vdash n$, are orthogonal idempotents follows from Lemma~\ref{lem:dashed-sum-identity} and the fact that the $e_\lambda$ are central.
  \details{
    Since $e_\lambda$ is central, we have
    \[
      \begin{tikzpicture}[anchorbase,scale=.7]
        \draw[dashed] (0.5,0) -- (0.5,0.5);
        \draw (-.2,.5) rectangle (1.2,1.2);
        \node at (.5,.85) {$s$};
        \draw[dashed] (0.5,1.2) -- (0.5,1.5);
        \draw (-.2,1.5) rectangle (1.2,2.2);
        \node at (.5,1.85) {$e_\lambda$};
        \draw[dashed] (0.5,2.2) -- (0.5,2.5);
        \draw (-.2,2.5) rectangle (1.2,3.2);
        \node at (.5,2.85) {$s^{-1}$};
        \draw[->,dashed] (0.5,3.2) -- (0.5,4);
      \end{tikzpicture}
      \quad=\quad
      \begin{tikzpicture}[anchorbase,scale=.7]
        \draw[dashed] (0.5,0) -- (0.5,1.5);
         \draw (-.2,1.5) rectangle (1.2,2.2);
        \node at (.5,1.85) {$e_\lambda$};
        \draw[->,dashed] (0.5,2.2) -- (0.5,4);
      \end{tikzpicture}
      \ .
    \]
  }
\end{proof}

\begin{rem} \label{rem:truncation-produces-idems}
  The fact that the $\epsilon_\lambda$ are idempotents relies on the fact that diagrams with negative regions are equal to zero.  The analogous diagrams in the 2-category version $\sH$ of Khovanov's category, where we allow regions labeled by any integer (see Remark~\ref{rem:fH-discussion}), are not idempotents.  In fact, it follows from \cite[Prop.~3]{Kho14} that the only idempotent 2-morphism of any $1_n$, $n \in \Z$, is the trivial diagram.  So we see that passing to the truncation, where we kill negative regions, introduces a large number of idempotents.
\end{rem}

\begin{lem} \label{lem:concSlide}
  For all $n \in \N$, we have
  \[
    \epsilon_{n+1} \sQ_+ = \sQ_+ \epsilon_n,\quad
    \epsilon_n \sQ_- = \sQ_- \epsilon_{n+1},\quad
    \delta_{n+1} \sQ_+ = \sQ_+ \delta_n,\quad
    \delta_n \sQ_- = \sQ_- \delta_{n+1},
  \]
  in $2\End_{\sH^\trunc} (\sQ_+ 1_n)$.
\end{lem}

\begin{proof}
  This follows from successive applications of \eqref{eq:Kho-rel-down-up-double-cross}, followed by \eqref{eq:Kho-rel-transposition-squared}.
\end{proof}

Recall Definition~\ref{def:idem-completion} of the idempotent completion of a 2-category.  We define $\sH_\epsilon$ to be the full sub-2-category of $\sH^\trunc$ whose objects are direct sums of triples $(n,1_n,\epsilon)$ where $n \in \N$, and $\epsilon$ is an idempotent 2-morphism of $1_n$ such that $\epsilon \epsilon_n = \epsilon$.  Similarly, we define $\sH_\delta$ to be the full sub-2-category of $\sH^\trunc$ whose objects are direct sums of triples $(n,1_n,\epsilon)$ where $n \in \N$, and $\epsilon$ is an idempotent 2-morphism of $1_n$ such that $\epsilon \delta_n = \epsilon$.

Recall that a 2-functor is an equivalence of 2-categories if and only if it is essentially surjective on objects, essentially full on 1-morphisms, and fully faithful on 2-morphisms.
\details{
  See, for example Section 7.5 of \href{https://www.math.uchicago.edu/~may/IMA/Lack.pdf}{A 2-categories companion} by S.~Lack.
}

\begin{prop} \label{prop:sH-decomp}
  We have an equivalence of 2-categories
  \[
    \sH^\trunc \cong \sH_\epsilon \oplus \sH_\delta.
  \]
\end{prop}

\begin{proof}
  It follows from Lemma~\ref{lem:orthoid} that, for all objects $(n,1_n,\epsilon)$ of $\sH^\trunc$, we have the decomposition
  \[
    (n,1_n,\epsilon) \cong (n,1_n,\epsilon \epsilon_n) \oplus (n,1_n,\epsilon \delta_n).
  \]
  By Lemma~\ref{lem:concSlide}, together with the fact that $\epsilon_n \delta_n = 0$ for all $n \in \N$, we see that any 1-morphism between an object of the form $(n,1_n,\epsilon \epsilon_n)$ and an object of the form $(n',1_{n'},\epsilon' \delta_{n'})$ is isomorphic to zero.
\end{proof}

\begin{prop} \label{prop:sHepsilon-quotient-clockwise-circle}
  The 2-category $\sH_\epsilon$ is isomorphic to the 2-category whose objects and 1-morphisms are the same as those of $\sH^\trunc$, and whose 2-morphisms spaces are quotients of the 2-morphism spaces of $\sH^\trunc$ by the local relation that a clockwise circle in a region labeled $n$ is equal to $n$:
  \begin{equation} \label{eq:clockwise-circle-equal-n}
    \begin{tikzpicture}[>=stealth,baseline={([yshift=-1.5ex]current bounding box.center)}]
      \draw[<-] (0.3,0) arc (0:360:0.3);
      \node at (-0.5,0.5) {\regionlabel{n}};
    \end{tikzpicture}
    \ = n,\quad n \in \N.
  \end{equation}
\end{prop}

\begin{proof}
  This follows from the fact that, in any quotient of the 2-morphism spaces,
  \[
    \delta_n = 0 \ \forall\ n \in \N
    \iff \epsilon_n = 1 \ \forall\ n \in \N
    \iff \eqref{eq:clockwise-circle-equal-n} \text{ holds}.
  \]
  Indeed, the reverse implication in the final $\iff$ is clear.  For the forward implication, note that $\epsilon_1 = 1$ gives \eqref{eq:clockwise-circle-equal-n} for $n=1$.  Then, supposing \eqref{eq:clockwise-circle-equal-n} holds for $n \le k$, the equation $\epsilon_{k+1}=1$ gives \eqref{eq:clockwise-circle-equal-n} for $n=k+1$.
\end{proof}

\begin{prop} \label{prop:sHepsilon-indecomposables}
  The objects
  \[
    (n,1_n,\epsilon_\lambda),\quad n \in \N,\ \lambda \vdash n,
  \]
  form a complete list of pairwise-nonisomorphic indecomposable objects of $\sH_\epsilon$.  Furthermore, any object of $\sH_\epsilon$ can be written uniquely (up to permutation) as a direct sum of indecomposable objects.
\end{prop}

\begin{proof}
  Fix $n \in \N$.  By \cite[Prop.~4.12]{LRS16}, the 2-endomorphism space of $1_n$ is spanned by
  \begin{equation} \label{eq:QSk-closure}
    \begin{tikzpicture}[anchorbase]
      \draw (-0.5,-0.25) -- (-0.5,0.25) -- (0.5,0.25) -- (0.5,-0.25) -- (-0.5,-0.25);
      \draw (0,0) node {$z$};
      \draw[->,dashed] (0,0.25) .. controls (0,1) and (1,1) .. (1,0.25) -- (1,-0.25) .. controls (1,-1) and (0,-1) .. (0,-0.25);
      \node at (-0.4,0.8) {\regionlabel{n}};
    \end{tikzpicture}
    \ ,\quad z \in \Q S_k,\ k \in \N.
  \end{equation}
  Now, if $k > n$, then the innermost region of the above diagram is negative, and so the diagram is zero.  On the other hand, if $k < n$, then by Proposition~\ref{prop:sHepsilon-quotient-clockwise-circle}, we can insert additional clockwise circles in the center region, up to a scalar multiple.  Therefore, the 2-endomorphism space of $1_n$ is spanned by the diagrams \eqref{eq:QSk-closure} for $z \in \Q S_n$.  Now
  \[
    \begin{tikzpicture}[anchorbase]
      \draw (-0.5,-0.25) -- (-0.5,0.25) -- (0.5,0.25) -- (0.5,-0.25) -- (-0.5,-0.25);
      \draw (0,0) node {$z$};
      \draw[->,dashed] (0,0.25) .. controls (0,1) and (1,1) .. (1,0.25) -- (1,-0.25) .. controls (1,-1) and (0,-1) .. (0,-0.25);
      \node at (0.75,0) {\regionlabel{0}};
    \end{tikzpicture}
    \ = \frac{1}{n!} \sum_{w \in S_n} \
    \begin{tikzpicture}[anchorbase]
      \draw (-1.3,-0.25) -- (-1.3,0.25) -- (-0.3,0.25) -- (-0.3,-0.25) -- (-1.3,-0.25);
      \node at (-0.8,0) {$z$};
      \draw (0,0.15) -- (0,0.65) -- (0.8,0.65) -- (0.8,0.15) -- (0,0.15);
      \node at (0.4,0.4) {$w$};
      \draw (0,-0.15) -- (0,-0.65) -- (0.8,-0.65) -- (0.8,-0.15) -- (0,-0.15);
      \node at (0.4,-0.4) {$w^{-1}$};
      \draw[->,dashed] (-0.8,0.25) -- (-0.8,0.65) .. controls (-0.8,1) and (0.4,1) .. (0.4,0.65);
      \draw[->,dashed] (0.4,0.15) -- (0.4,-0.15);
      \draw[->,dashed] (0.4,-0.65) .. controls (0.4,-1) and (-0.8,-1) .. (-0.8,-0.65) -- (-0.8,-0.25);
      \node at (-0.3,0.5) {\regionlabel{0}};
    \end{tikzpicture}
    \ =\
    \begin{tikzpicture}[anchorbase]
      \draw (-0.5,-0.25) -- (-0.5,0.25) -- (0.5,0.25) -- (0.5,-0.25) -- (-0.5,-0.25);
      \draw (0,0) node {$z'$};
      \draw[->,dashed] (0,0.25) .. controls (0,1) and (1,1) .. (1,0.25) -- (1,-0.25) .. controls (1,-1) and (0,-1) .. (0,-0.25);
      \node at (0.75,0) {\regionlabel{0}};
    \end{tikzpicture}\ ,
  \]
  where $z' = \frac{1}{n!} \sum_{w \in S_n} w^{-1} z w \in Z(\Q S_n)$.  Therefore, the 2-endomorphism space of $1_n$ is spanned by the diagrams \eqref{eq:QSk-closure} for $z$ in the center $Z(\Q S_n)$ of $\Q S_n$.  But then this 2-endomorphism space is spanned by the diagrams \eqref{eq:QSk-closure} as $z$ ranges over the central idempotents $e_\lambda$, $\lambda \vdash n$.  In other words, this space is spanned by the $\epsilon_\lambda$, $\lambda \vdash n$.  It follows from Lemma~\ref{lem:epsilon_lambda-action} below that these elements are also linearly independent.
\end{proof}

%---------------------------
\subsection{Region shifting}
% --------------------------

We define a shift 2-functor
\[
  \partial \colon {\sH^\trunc}' \to {\sH^\trunc}',
\]
given by lowering region labels by one.  More precisely, on objects we define
\[
  \partial(0) = \partial(\bzero) = \bzero,\quad
  \partial(n) = n-1,\quad n > 1.
\]
On 1-morphisms, we define
\[
  \partial(\sQ_c 1_n) = \sQ_c 1_{n-1},
\]
for $n \in \N$ and $c$ a (possibly empty) sequence of $+$ and $-$.  On 2-morphisms, $\partial$ is given by lowering the region labels of diagrams by one.  The 2-functor $\partial$ induces a 2-functor
\[
  \partial \colon \sH^\trunc \to \sH^\trunc
\]
on the idempotent completion $\sH^\trunc = \Kar({\sH^\trunc}')$.

\begin{prop} \label{prop:Hdelta-H-equiv}
  The 2-functor $\partial$ sends $\sH_\epsilon$ to zero.  Furthermore, the restriction of $\partial$ to $\sH_\delta$ induces an equivalence of 2-categories $\sH_\delta \cong \sH^\trunc$.
\end{prop}

\begin{proof}
  Since $\partial$ maps $\epsilon_\lambda$ to zero for all partitions $\lambda$, the first statement follows.  Now let $\partial_\delta \colon \sH_\delta \to \sH^\trunc$ be the restriction of the 2-functor $\partial$.  We define a 2-functor $s' \colon {\sH^\trunc}' \to \sH_\delta$.  On objects, we define
  \[
    s'(\bzero) = \bzero,\quad
    s'(n) = (n+1,1_{n+1},\delta_{n+1}),\quad n \in \N.
  \]
  On 1-morphisms, we define
  \[
    s'(\sQ_c 1_n) = (\sQ_c 1_{n+1},\sQ_c \delta_{n+1}).
  \]
  On 2-morphisms, we define $s'$ by increasing the region labels in diagrams by one.  Note that the definition of $s'$ is compatible with our convention that diagrams with negative region labels are zero since Lemma~\ref{lem:concSlide}, together with the fact that $\delta_0=0$, ensures that $s'$ maps any diagram with a negative region label to zero.  The 2-functor $s'$ induces a 2-functor $s \colon \sH^\trunc \to \sH_\delta$.

  It is clear that $\partial_\delta \circ s$ is isomorphic to the identity 2-functor since $\partial(\delta_{n+1}) = \id_n$ for all $n \in \N$.  Similarly, $s \circ \partial_\delta$ is isomorphic to the identity 2-functor since the object $(0,1_0,\delta_0)$ is already isomorphic to zero in $\sH_\delta$.
\end{proof}

\begin{cor} \label{cor:sH-infinite-decomp}
  We have an equivalence of 2-categories $\sH^\trunc \cong \bigoplus_{m=0}^\infty \sH_\epsilon$.
\end{cor}

\begin{proof}
  It is clear that, for any 2-morphism $\theta$ in $\sH^\trunc$, we have $\partial^m(\theta)=0$ for sufficiently large $m$.  Thus, the result follows from Propositions~\ref{prop:sH-decomp} and~\ref{prop:Hdelta-H-equiv}.
\end{proof}

%%%%%%%%%%%%%%%%%%%%%%%%%%%%%%%%%%%%%%%%%%%%%%%%%
%
\section{Equivalence of $\sH_\epsilon$ and $\sA$} \label{sec:sHepsilon-A-equivalence}
%
%%%%%%%%%%%%%%%%%%%%%%%%%%%%%%%%%%%%%%%%%%%%%%%%%%%%%%%%%

%----------------------------------------------------
\subsection{A 2-functor from $\sA$ to $\sH_\epsilon$} \label{subsec:sHepsilon-sU-equiv}
% ---------------------------------------------------

\begin{lem} \label{lem:plSlide}
  For $\mu \vdash n+1$, we have the following equality of 2-morphisms in $\sH^\trunc$:
  \begin{equation} \label{eq:plSlide}
    \begin{tikzpicture}[anchorbase]
      \node at (-0.5,0) {\regionlabel{\epsilon_\mu}};
      \draw[->] (0,-1) -- (0,1);
      \node at (0.3,0.4) {\regionlabel{n}};
    \end{tikzpicture}
    \ = \frac{1}{n!}\
    \begin{tikzpicture}[anchorbase]
      \draw (-0.5,-0.25) -- (-0.5,0.25) -- (0.5,0.25) -- (0.5,-0.25) -- (-0.5,-0.25);
      \draw (0,0) node {$e_\mu$};
      \draw[->,dashed] (0.1,0.25) .. controls (0.1,1) and (1,1) .. (1,0.25) -- (1,-0.25) .. controls (1,-1) and (0.1,-1) .. (0.1,-0.25);
      \draw[->] (-0.3,0.25) -- (-0.3,1);
      \draw[<-] (-0.3,-0.25) -- (-0.3,-1);
      \node at (0.75,0) {\regionlabel{0}};
      \node at (-0.05,0.6) {\regionlabel{n}};
    \end{tikzpicture}
  \end{equation}
\end{lem}

\begin{proof}
  We have
  \[
    \begin{tikzpicture}[anchorbase]
      \node at (-0.5,0) {\regionlabel{\epsilon_\mu}};
      \draw[->] (0,-1) -- (0,1);
      \node at (0.3,0.4) {\regionlabel{n}};
    \end{tikzpicture}
    \ =\
    \frac{1}{(n+1)!}\
    \begin{tikzpicture}[anchorbase]
      \draw (-0.5,-0.25) -- (-0.5,0.25) -- (0.5,0.25) -- (0.5,-0.25) -- (-0.5,-0.25);
      \draw (0,0) node {$e_\mu$};
      \draw[->,dashed] (0,0.25) .. controls (0,1) and (1,1) .. (1,0.25) -- (1,-0.25) .. controls (1,-1) and (0,-1) .. (0,-0.25);
      \node at (0.75,0) {\regionlabel{0}};
      \draw[->] (1.3,-1) -- (1.3,1);
      \node at (1.6,0.4) {\regionlabel{n}};
    \end{tikzpicture}
    \ =\frac{1}{(n+1)!} \sum_{\ell=0}^n\
    \begin{tikzpicture}[anchorbase]
      \draw (-0.7,-0.25) -- (-0.7,0.25) -- (0.7,0.25) -- (0.7,-0.25) -- (-0.7,-0.25);
      \draw (0,0) node {$e_\mu$};
      \draw[->] (-1,-1) .. controls (-1,-0.5) and (0,-0.75) .. (0,-0.25);
      \draw[<-] (-1,1) .. controls (-1,0.5) and (0,0.75) .. (0,0.25);
      \draw[->,dashed] (0.4,0.25) -- (0.4,0.3) arc (180:0:0.3) -- (1,0.3) -- (1,-0.3) arc (360:180:0.3) -- (0.4,-0.25);
      \draw[->,dashed] (-0.4,0.25) -- (-0.4,0.3) .. controls (-0.4,1) and (1.3,1) .. (1.3,0.3) -- (1.3,-0.3) .. controls (1.3,-1) and (-0.4,-1) .. (-0.4,-0.3) -- (-0.4,-0.25);
      \node at (-0.5,0.8) {\regionlabel{n}};
      \node at (0.2,0.5) {\regionlabel{\ell}};
    \end{tikzpicture}
    = \frac{1}{n!}\
    \begin{tikzpicture}[anchorbase]
      \draw (-0.5,-0.25) -- (-0.5,0.25) -- (0.5,0.25) -- (0.5,-0.25) -- (-0.5,-0.25);
      \draw (0,0) node {$e_\mu$};
      \draw[->,dashed] (0.1,0.25) .. controls (0.1,1) and (1,1) .. (1,0.25) -- (1,-0.25) .. controls (1,-1) and (0.1,-1) .. (0.1,-0.25);
      \draw[->] (-0.3,0.25) -- (-0.3,1);
      \draw[<-] (-0.3,-0.25) -- (-0.3,-1);
      \node at (0.75,0) {\regionlabel{0}};
      \node at (-0.05,0.6) {\regionlabel{n}};
    \end{tikzpicture}
    \ ,
  \]
  where the second equality follows from repeated use of \eqref{eq:Kho-rel-down-up-double-cross} together with the fact that diagrams with negative region labels are equal to zero, and the last equality follows from the fact that the idempotent $e_\mu$ is central (so we can slide crossings through the box labeled $e_\mu$).
\end{proof}

\begin{cor} \label{cor:epsilon_lambda-Q+-sandwich}
  For $\lambda\vdash n$ and $\mu \vdash n+1$, we have
  \[
    \epsilon_\mu \sQ_+ \epsilon_\lambda = 0
    \quad \text{unless } \lambda \subseteq \mu.
  \]
\end{cor}

\begin{proof}
  We have
  \begin{equation} \label{eq:idempotent-dot-sandwhich}
    \begin{tikzpicture}[anchorbase]
      \draw[->] (0,-1) -- (0,1);
      \node at (-0.4,0) {\regionlabel{\epsilon_\mu}};
      \node at (0.4,0) {\regionlabel{\epsilon_\lambda}};
    \end{tikzpicture}
    \stackrel{\eqref{eq:plSlide}}{=} \frac{1}{n!(n-1)!}\
    \begin{tikzpicture}[anchorbase]
      \draw (-1.3,-0.25) -- (-1.3,0.25) -- (-0.3,0.25) -- (-0.3,-0.25) -- (-1.3,-0.25);
      \node at (-0.8,0) {$e_\mu$};
      \draw (1.3,-0.25) -- (1.3,0.25) -- (0.3,0.25) -- (0.3,-0.25) -- (1.3,-0.25);
      \node at (0.8,0) {$e_\lambda$};
      \draw[->] (-1.1,-1) -- (-1.1,-0.25);
      \draw[->] (-1.1,0.25) -- (-1.1,1);
      \draw[->] (-0.8,0.25) .. controls (-0.8,1.1) and (1,1.1) .. (1,0.25);
      \draw[->] (1,-0.25) .. controls (1,-1.1) and (-0.8,-1.1) .. (-0.8,-0.25);
      \draw[->,dashed] (-0.5,0.25) .. controls (-0.5,0.7) and (-0.1,0.7) .. (-0.1,0.25) -- (-0.1,-0.25) .. controls (-0.1,-0.7) and (-0.5,-0.7) .. (-0.5,-0.25);
      \draw[<-,dashed] (0.6,0.25) .. controls (0.6,0.7) and (0.1,0.7) .. (0.1,0.25) -- (0.1,-0.25) .. controls (0.1,-0.7) and (0.6,-0.7) .. (0.6,-0.25);
      \node at (-0.8,0.8) {\regionlabel{n}};
    \end{tikzpicture}
    \ = \frac{1}{n!}\
    \begin{tikzpicture}[anchorbase]
      \draw (-1.3,-0.25) -- (-1.3,0.25) -- (-0.3,0.25) -- (-0.3,-0.25) -- (-1.3,-0.25);
      \node at (-0.8,0) {$e_\mu$};
      \draw (1,-0.25) -- (1,0.25) -- (0,0.25) -- (0,-0.25) -- (1,-0.25);
      \node at (0.5,0) {$e_\lambda$};
      \draw[->] (-1.1,-1) -- (-1.1,-0.25);
      \draw[->] (-1.1,0.25) -- (-1.1,1);
      \draw[->,dashed] (-0.8,0.25) .. controls (-0.8,1) and (0.5,1) .. (0.5,0.25);
      \draw[->,dashed] (0.5,-0.25) .. controls (0.5,-1) and (-0.8,-1) .. (-0.8,-0.25);
      \node at (-0.8,0.8) {\regionlabel{n}};
    \end{tikzpicture}
    \ = \frac{1}{n!}\
    \begin{tikzpicture}[anchorbase]
      \draw (-1.3,-0.25) -- (-1.3,0.25) -- (-0.3,0.25) -- (-0.3,-0.25) -- (-1.3,-0.25);
      \node at (-0.8,0) {$e_\mu e_\lambda$};
      \draw[->] (-1.1,-1) -- (-1.1,-0.25);
      \draw[->] (-1.1,0.25) -- (-1.1,1);
      \draw[->,dashed] (-0.8,0.25) .. controls (-0.8,0.8) and (0,0.8) .. (0,0.25) -- (0,-0.25) .. controls (0,-0.8) and (-0.8,-0.8) .. (-0.8,-0.25);
      \node at (-0.8,0.8) {\regionlabel{n}};
    \end{tikzpicture}
    \ ,
  \end{equation}
  where the second equality follows from repeated use of \eqref{eq:Kho-rel-down-up-double-cross}, the fact that diagrams with negative region labels are equal to zero, and the fact that the idempotent $e_\lambda$ is central (as in the proof of Lemma~\ref{lem:plSlide}).  The result then follows from the fact that $e_\mu e_\lambda = 0$ unless $\lambda \subseteq \mu$.
\end{proof}

\begin{lem}
  For $\lambda \vdash n$, we have

  \noindent\begin{minipage}{0.5\linewidth}
    \begin{equation} \label{eq:epsilon_lambda-clockwise}
      \begin{tikzpicture}[anchorbase]
        \draw[<-] (0.5,0) arc(0:360:0.5);
        \node at (0,0) {\regionlabel{\epsilon_\lambda}};
      \end{tikzpicture}
      \ = \sum_{i \in \Z} \frac{(n+1) d_\lambda}{d_{\lambda \boxplus i}} \epsilon_{\lambda \boxplus i},
    \end{equation}
  \end{minipage}%
  \begin{minipage}{0.5\linewidth}\
    \begin{equation} \label{eq:epsilon_lambda-counter-clockwise}
      \begin{tikzpicture}[anchorbase]
        \draw[->] (0.5,0) arc(0:360:0.5);
        \node at (0,0) {\regionlabel{\epsilon_\lambda}};
      \end{tikzpicture}
      \ = \sum_{i \in \Z} \frac{d_\lambda}{n d_{\lambda \boxminus i}} \epsilon_{\lambda \boxminus i}.
    \end{equation}
  \end{minipage}\par\vspace{\belowdisplayskip}
\end{lem}

\begin{proof}
  Suppose $\lambda \vdash n$.  To prove \eqref{eq:epsilon_lambda-clockwise}, we compute
  \begin{multline*}
    \begin{tikzpicture}[anchorbase]
      \draw[<-] (0.5,0) arc(0:360:0.5);
      \node at (0,0) {\regionlabel{\epsilon_\lambda}};
    \end{tikzpicture}
    \ = \sum_{i \in \Z}\
    \begin{tikzpicture}[anchorbase]
      \draw[<-] (0.5,0) arc(0:360:0.5);
      \node at (0,0) {\regionlabel{\epsilon_\lambda}};
      \node at (-0.8,0.3) {\regionlabel{\epsilon_{\lambda \boxplus i}}};
    \end{tikzpicture}
    \ \stackrel{\eqref{eq:idempotent-dot-sandwhich}}{=}
    \sum_{i \in \Z} \frac{1}{n!}\
    \begin{tikzpicture}[anchorbase]
      \draw (-1.4,-0.25) -- (-1.4,0.25) -- (-0.2,0.25) -- (-0.2,-0.25) -- (-1.4,-0.25);
      \node at (-0.8,0) {$e_{\lambda \boxplus i} e_\lambda$};
      \draw[->,dashed] (-0.8,0.25) .. controls (-0.8,0.8) and (0,0.8) .. (0,0.25) -- (0,-0.25) .. controls (0,-0.8) and (-0.8,-0.8) .. (-0.8,-0.25);
    \end{tikzpicture}
    \ =\
    \frac{1}{n!(n+1)!} \sum_{w \in S_{n+1}}
    \begin{tikzpicture}[anchorbase]
      \draw (-1.4,-0.25) -- (-1.4,0.25) -- (-0.2,0.25) -- (-0.2,-0.25) -- (-1.4,-0.25);
      \node at (-0.8,0) {$e_{\lambda \boxplus i} e_\lambda$};
      \draw (0,0.15) -- (0,0.65) -- (0.8,0.65) -- (0.8,0.15) -- (0,0.15);
      \node at (0.4,0.4) {$w$};
      \draw (0,-0.15) -- (0,-0.65) -- (0.8,-0.65) -- (0.8,-0.15) -- (0,-0.15);
      \node at (0.4,-0.4) {$w^{-1}$};
      \draw[->,dashed] (-0.8,0.25) -- (-0.8,0.65) .. controls (-0.8,1) and (0.4,1) .. (0.4,0.65);
      \draw[->,dashed] (0.4,0.15) -- (0.4,-0.15);
      \draw[->,dashed] (0.4,-0.65) .. controls (0.4,-1) and (-0.8,-1) .. (-0.8,-0.65) -- (-0.8,-0.25);
    \end{tikzpicture}
    \\
    \stackrel{\eqref{eq:idempotent-symmetrize}}{=}
    \sum_{i \in \Z} \frac{d_\lambda}{n! d_{\lambda \boxplus i}}\
    \begin{tikzpicture}[anchorbase]
      \draw (-1.2,-0.25) -- (-1.2,0.25) -- (-0.4,0.25) -- (-0.4,-0.25) -- (-1.2,-0.25);
      \node at (-0.8,0) {$e_{\lambda \boxplus i}$};
      \draw[->,dashed] (-0.8,0.25) .. controls (-0.8,0.8) and (0,0.8) .. (0,0.25) -- (0,-0.25) .. controls (0,-0.8) and (-0.8,-0.8) .. (-0.8,-0.25);
    \end{tikzpicture}
    =
    \sum_{i \in \Z} \frac{(n+1) d_\lambda}{d_{\lambda \boxplus i}} \epsilon_{\lambda \boxplus i}.
  \end{multline*}

  To prove \eqref{eq:epsilon_lambda-counter-clockwise}, we compute
  \[
    \begin{tikzpicture}[anchorbase]
      \draw[->] (0.5,0) arc(0:360:0.5);
      \node at (0,0) {\regionlabel{\epsilon_\lambda}};
    \end{tikzpicture}
    \ \stackrel{\eqref{eq:plSlide}}{=}\
    \frac{1}{(n-1)!}\
    \begin{tikzpicture}[anchorbase]
      \draw (-0.5,-0.25) -- (-0.5,0.25) -- (0.5,0.25) -- (0.5,-0.25) -- (-0.5,-0.25);
      \draw (0,0) node {$e_\lambda$};
      \draw[->,dashed] (0.1,0.25) .. controls (0.1,1) and (1,1) .. (1,0.25) -- (1,-0.25) .. controls (1,-1) and (0.1,-1) .. (0.1,-0.25);
      \draw[->] (-0.3,0.25) .. controls (-0.3,0.8) and (-0.7,0.8) .. (-0.7,0.25) -- (-0.7,-0.25) .. controls (-0.7,-0.8) and (-0.3,-0.8) .. (-0.3,-0.25);
      \node at (0.75,0) {\regionlabel{0}};
    \end{tikzpicture}
    \ \stackrel{\substack{\eqref{eq:Kho-rel-left-curl} \\ \eqref{eq:trace-of-epsilon_lambda}}}{=}
    \frac{d_\lambda}{n! d_{\lambda \boxminus i}}\
    \begin{tikzpicture}[anchorbase]
      \draw (-0.5,-0.25) -- (-0.5,0.25) -- (0.5,0.25) -- (0.5,-0.25) -- (-0.5,-0.25);
      \draw (0,0) node {$e_{\lambda \boxminus i}$};
      \draw[->,dashed] (0.1,0.25) .. controls (0.1,1) and (1,1) .. (1,0.25) -- (1,-0.25) .. controls (1,-1) and (0.1,-1) .. (0.1,-0.25);
      \node at (0.75,0) {\regionlabel{0}};
    \end{tikzpicture}
    \ =
    \frac{d_\lambda}{n d_{\lambda \boxminus i}} \epsilon_{\lambda \boxminus i}. \qedhere
  \]
\end{proof}

We will use an open circle to denote a right curl in $\sH^\trunc$:
\begin{equation} \label{eq:right-curl}
  \begin{tikzpicture}[anchorbase]
    \draw[->] (0,0) -- (0,1);
    \redcircle{(0,0.5)};
  \end{tikzpicture}
  \ :=\
  \begin{tikzpicture}[anchorbase,scale=0.5]
    \draw (2,1) .. controls (2,1.5) and (1.3,1.5) .. (1.1,1);
    \draw (2,1) .. controls (2,.5) and (1.3,.5) .. (1.1,1);
    \draw (1,0) .. controls (1,.5) .. (1.1,1);
    \draw[->] (1.1,1) .. controls (1,1.5) .. (1,2);
  \end{tikzpicture}
\end{equation}

\begin{lem} \label{lem:right-curl-removal}
  For a partition $\lambda$ and $i \in \Z$, we have
  \begin{equation} \label{eq:right-curl-removal}
    \begin{tikzpicture}[anchorbase]
      \draw[->] (0,0) -- (0,1);
      \node at (0.4,0.7) {\regionlabel{\epsilon_\lambda}};
      \node at (-0.5,0.7) {\regionlabel{\epsilon_{\lambda \boxplus i}}};
      \redcircle{(0,0.5)};
    \end{tikzpicture}
    \ = i \left(
    \begin{tikzpicture}[anchorbase]
      \draw[->] (0,0) -- (0,1);
      \node at (0.4,0.5) {\regionlabel{\epsilon_\lambda}};
      \node at (-0.5,0.5) {\regionlabel{\epsilon_{\lambda \boxplus i}}};
    \end{tikzpicture}
    \right)
  \end{equation}
\end{lem}

\begin{proof}
  For $\lambda \vdash n+1$, we have
  \[
    \begin{tikzpicture}[anchorbase]
      \draw[->] (0,-1) -- (0,1);
      \node at (-0.4,0.5) {\regionlabel{\epsilon_{\lambda \boxplus i}}};
      \node at (0.4,0.5) {\regionlabel{\epsilon_\lambda}};
      \redcircle {(0,0)};
    \end{tikzpicture}
    \stackrel{\eqref{eq:idempotent-dot-sandwhich}}{=} \frac{1}{n!}\
    \begin{tikzpicture}[anchorbase]
      \draw (-0.7,-0.5) -- (-0.7,0) -- (0.7,0) -- (0.7,-0.5) -- (-0.7,-0.5);
      \node at (0,-0.25) {$e_{\lambda \boxplus i} e_\lambda$};
      \draw[->] (-0.3,-1) -- (-0.3,-0.5);
      \draw[->] (-0.3,0) -- (-0.3,1);
      \draw[->,dashed] (0.3,0) .. controls (0.3,0.5) and (0.9,0.5) .. (0.9,0) -- (0.9,-0.5) .. controls (0.9,-1) and (0.3,-1) .. (0.3,-0.5);
      \node at (0.3,0.6) {\regionlabel{n}};
      \redcircle{(-0.3,0.5)};
    \end{tikzpicture}
    \ = \frac{1}{n!}\
    \begin{tikzpicture}[anchorbase]
      \draw (-0.7,-0.75) -- (-0.7,-0.25) -- (0.7,-0.25) -- (0.7,-0.75) -- (-0.7,-0.75);
      \node at (0,-0.5) {$e_{\lambda \boxplus i} e_\lambda$};
      \draw (-0.7,0.75) -- (-0.7,0.25) -- (0.7,0.25) -- (0.7,0.75) -- (-0.7,0.75);
      \node at (0,0.5) {$J_{n+1}$};
      \draw[->,dashed] (0,-0.25) -- (0,0.25);
      \draw[->] (-0.3,-1.4) -- (-0.3,-0.75);
      \draw[->] (-0.3,0.75) -- (-0.3,1.4);
      \draw[->,dashed] (0.3,0.75) .. controls (0.3,1.3) and (1,1.3) .. (1,0.75) -- (1,-0.75) .. controls (1,-1.3) and (0.3,-1.3) .. (0.3,-0.75);
      \node at (0,1) {\regionlabel{n}};
    \end{tikzpicture}
    \ \stackrel{\eqref{eq:JM-idempotent-product}}{=} \frac{i}{n!}\
    \begin{tikzpicture}[anchorbase]
      \draw (-0.7,-0.25) -- (-0.7,0.25) -- (0.7,0.25) -- (0.7,-0.25) -- (-0.7,-0.25);
      \node at (0,0) {$e_{\lambda \boxplus i} e_\lambda$};
      \draw[->] (-0.3,-1) -- (-0.3,-0.25);
      \draw[->] (-0.3,0.25) -- (-0.3,1);
      \draw[->,dashed] (0.3,0.25) .. controls (0.3,0.75) and (0.9,0.75) .. (0.9,0.25) -- (0.9,-0.25) .. controls (0.9,-0.75) and (0.3,-0.75) .. (0.3,-0.25);
    \end{tikzpicture}
    \ \stackrel{\eqref{eq:idempotent-dot-sandwhich}}{=} i \left(
    \begin{tikzpicture}[anchorbase]
      \draw[->] (0,0) -- (0,1);
      \node at (0.4,0.5) {\regionlabel{\epsilon_\lambda}};
      \node at (-0.5,0.5) {\regionlabel{\epsilon_{\lambda \boxplus i}}};
    \end{tikzpicture}
    \right),
  \]
  where the second equality follows from repeated use of \eqref{eq:Kho-rel-down-up-double-cross} together with the fact that diagrams with negative region labels are equal to zero.
\end{proof}

\begin{lem} \label{lem:crossingRemoval}
  For a partition $\lambda$ and $i,j \in \Z$, $i \ne j$, we have
  \begin{equation} \label{eq:crossingRemoval}
    \begin{tikzpicture}[anchorbase]
      \draw[->] (0,0) -- (1.4,1.4);
      \draw[->] (1.4,0) -- (0,1.4);
      \node at (1.2,.7) {\regionlabel{\epsilon_\lambda}};
      \node at (.7,.2) {\regionlabel{\epsilon_{\lambda\boxplus i}}};
      \node at (.7,1.2) {\regionlabel{\epsilon_{\lambda \boxplus i}}};
      \node at (0,.7) {\regionlabel{\epsilon_{\lambda\boxplus i \boxplus j}}};
    \end{tikzpicture}
    \quad = \frac{1}{j-i} \left(
    \begin{tikzpicture}[anchorbase]
      \draw[->] (0,0) -- (0,1.4);
      \draw[->] (1,0) -- (1,1.4);
      \node at (1.4,.7) {\regionlabel{\epsilon_\lambda}};
      \node at (.5,.7) {\regionlabel{\epsilon_{\lambda \boxplus i}}};
      \node at (-.6,.7) {\regionlabel{\epsilon_{\lambda\boxplus i \boxplus j}}};
    \end{tikzpicture}
    \right)
  \end{equation}
\end{lem}

\begin{proof}
  We have
  \[
    (j-i) \left(
    \begin{tikzpicture}[anchorbase]
      \draw[->] (0,0) -- (1.4,1.4);
      \draw[->] (1.4,0) -- (0,1.4);
      \node at (1.2,.7) {\regionlabel{\epsilon_\lambda}};
      \node at (.7,.2) {\regionlabel{\epsilon_{\lambda\boxplus i}}};
      \node at (.7,1.2) {\regionlabel{\epsilon_{\lambda \boxplus i}}};
      \node at (0,.7) {\regionlabel{\epsilon_{\lambda\boxplus i \boxplus j}}};
    \end{tikzpicture}
    \right)
    \stackrel{\eqref{eq:right-curl-removal}}{=} \
    \begin{tikzpicture}[anchorbase]
      \draw[->] (0,0) -- (1.4,1.4);
      \draw[->] (1.4,0) -- (0,1.4);
      \node at (1.2,.7) {\regionlabel{\epsilon_\lambda}};
      \node at (.7,.2) {\regionlabel{\epsilon_{\lambda\boxplus i}}};
      \node at (.7,1.2) {\regionlabel{\epsilon_{\lambda \boxplus i}}};
      \node at (0,.7) {\regionlabel{\epsilon_{\lambda\boxplus i \boxplus j}}};
      \redcircle{(0.5,0.5)};
    \end{tikzpicture}
    \ -\
    \begin{tikzpicture}[anchorbase]
      \draw[->] (0,0) -- (1.4,1.4);
      \draw[->] (1.4,0) -- (0,1.4);
      \node at (1.2,.7) {\regionlabel{\epsilon_\lambda}};
      \node at (.7,.2) {\regionlabel{\epsilon_{\lambda\boxplus i}}};
      \node at (.7,1.2) {\regionlabel{\epsilon_{\lambda \boxplus i}}};
      \node at (0,.7) {\regionlabel{\epsilon_{\lambda\boxplus i \boxplus j}}};
      \redcircle{(0.9,0.9)};
    \end{tikzpicture}
    \ =\
    \begin{tikzpicture}[anchorbase]
      \draw[->] (0,0) -- (0,1.4);
      \draw[->] (1,0) -- (1,1.4);
      \node at (1.4,.7) {\regionlabel{\epsilon_\lambda}};
      \node at (.5,.7) {\regionlabel{\epsilon_{\lambda \boxplus i}}};
      \node at (-.6,.7) {\regionlabel{\epsilon_{\lambda\boxplus i \boxplus j}}};
    \end{tikzpicture}
  \]
  where the last equality follows from the relation in \cite[\S2.1]{Kho14} concerning sliding right curls through crossings.
\end{proof}

We now define an additive linear 2-functor
\[
  \bS \colon \sA \to \sH_\epsilon
\]
as follows.  On objects, we define
\[
  \bS(\bzero) = \bzero,\quad
  \bS(\lambda) = (n,1_n,\epsilon_\lambda)\quad \text{for } \lambda \vdash n.
\]
On 1-morphisms, $\bS$ is determined by
\[
  \bS(\sE_i 1_\lambda) = (\sQ_- 1_n, \epsilon_{\lambda \boxminus i} \sQ_- \epsilon_\lambda),\quad
  \bS(\sF_i 1_\lambda) = (\sQ_+ 1_n, \epsilon_{\lambda \boxplus i} \sQ_+ \epsilon_\lambda),\quad
  \text{for } \lambda \vdash n,
\]
where, by convention, we set $\epsilon_\bzero = 0$.  On 2-morphisms, $\bS$ is determined as follows.  Suppose $\theta$ is a diagram representing a 2-morphism in $\sA$.  Then $\bS(\theta)$ is the diagram obtained from $\theta$ by placing a $\epsilon_\lambda$ in each region labeled $\lambda$ and then acting as follows on crossings, cups, and caps:

\noindent\begin{minipage}{0.5\linewidth}
  \begin{equation}
    \bS \left(
    \begin{tikzpicture}[anchorbase]
      \draw[->] (0,0) node [anchor=north] {\strandlabel{i}} -- (0.5,0.5);
      \draw[->] (0.5,0) node [anchor=north] {\strandlabel{j}} -- (0,0.5);
      \node at (0.6,0.25) {\regionlabel{\lambda}};
    \end{tikzpicture}
    \right)
    \ =\
    \xi_{i,j} \
    \begin{tikzpicture}[anchorbase]
      \draw [->](0,0) -- (1,1);
      \draw [->](1,0) -- (0,1);
      \node at (1,0.5) {\regionlabel{\epsilon_\lambda}};
      \node at (0.5,1) {\regionlabel{\epsilon_{\lambda \boxplus i}}};
      \node at (0.5,0) {\regionlabel{\epsilon_{\lambda \boxplus j}}};
      \node at (-0.2,0.5) {\regionlabel{\epsilon_{\lambda \boxplus i \boxplus j}}};
    \end{tikzpicture}
    \ ,
  \end{equation}
\end{minipage}%
\begin{minipage}{0.5\linewidth}\
    \begin{equation}
    \bS \left(
    \begin{tikzpicture}[anchorbase]
      \draw[<-] (0,0) node [anchor=north] {\strandlabel{i}} -- (0.5,0.5);
      \draw[<-] (0.5,0) node [anchor=north] {\strandlabel{j}} -- (0,0.5);
      \node at (0.6,0.25) {\regionlabel{\lambda}};
    \end{tikzpicture}
    \right)
    \ =\
    \xi_{i,j} \
    \begin{tikzpicture}[anchorbase]
      \draw [<-](0,0) -- (1,1);
      \draw [<-](1,0) -- (0,1);
      \node at (1,0.5) {\regionlabel{\epsilon_\lambda}};
      \node at (0.5,1) {\regionlabel{\epsilon_{\lambda \boxminus i}}};
      \node at (0.5,0) {\regionlabel{\epsilon_{\lambda \boxminus j}}};
      \node at (-0.2,0.5) {\regionlabel{\epsilon_{\lambda \boxminus i \boxminus j}}};
    \end{tikzpicture}
    \ ,
  \end{equation}
\end{minipage}\par\vspace{\belowdisplayskip}

\begin{equation}
  \bS \left(
  \begin{tikzpicture}[anchorbase]
    \draw[->] (0,0) node[anchor=east] {\strandlabel{j}} -- (0.5,0.5);
    \draw[<-] (0.5,0) -- (0,0.5) node[anchor=east] {\strandlabel{i}};
    \draw (0.7,0.25) node {\regionlabel{\lambda}};
  \end{tikzpicture}
  \right)
  = \xi_{i,j} \frac{(|\lambda| + 1) d_\lambda d_{\lambda \boxminus i \boxplus j}}{|\lambda| d_{\lambda \boxminus i} d_{\lambda \boxplus j}}\
  \begin{tikzpicture}[anchorbase]
    \draw [->](0,0) -- (1,1);
    \draw [<-](1,0) -- (0,1);
    \node at (1,0.5) {\regionlabel{\epsilon_\lambda}};
    \node at (0.5,1) {\regionlabel{\epsilon_{\lambda \boxplus j}}};
    \node at (0.5,0) {\regionlabel{\epsilon_{\lambda \boxminus i}}};
    \node at (-0.2,0.5) {\regionlabel{\epsilon_{\lambda \boxminus i \boxplus j}}};
  \end{tikzpicture}
  \ ,
\end{equation}

\begin{equation}
  \bS \left(
  \begin{tikzpicture}[>=stealth,baseline={([yshift=-.5ex]current bounding box.center)}]
    \draw[<-] (0,0) node[anchor=east] {\strandlabel{j}} -- (0.5,0.5);
    \draw[->] (0.5,0) -- (0,0.5) node[anchor=east] {\strandlabel{i}};
    \node at (0.6,0.25) {\regionlabel{\lambda}};
  \end{tikzpicture}
  \right)
  = \xi_{i,j}\
  \begin{tikzpicture}[anchorbase]
    \draw [<-](0,0) -- (1,1);
    \draw [->](1,0) -- (0,1);
    \node at (1,0.5) {\regionlabel{\epsilon_\lambda}};
    \node at (0.5,1) {\regionlabel{\epsilon_{\lambda \boxminus j}}};
    \node at (0.5,0) {\regionlabel{\epsilon_{\lambda \boxplus i}}};
    \node at (-0.2,0.5) {\regionlabel{\epsilon_{\lambda \boxplus i \boxminus j}}};
  \end{tikzpicture}
  \ ,
\end{equation}

\noindent\begin{minipage}{0.5\linewidth}
  \begin{equation}
    \bS \left(
    \begin{tikzpicture}[>=stealth,baseline={([yshift=-.5ex]current bounding box.center)}]
      \draw[->] (0,0) node[anchor=north] {\strandlabel{i}} -- (0,0.1) arc (180:0:.3) -- (0.6,0);
      \node at (.9,0.3) {\regionlabel{\lambda}};
    \end{tikzpicture}
    \right)
    = \frac{d_\lambda}{|\lambda| d_{\lambda \boxminus i}}\
    \begin{tikzpicture}[anchorbase]
      \draw[->] (0,0) -- (0,0.1) arc (180:0:.5) -- (1,0);
      \node at (1.2,0.5) {\regionlabel{\epsilon_\lambda}};
      \node at (0.5,0.2) {\regionlabel{\epsilon_{\lambda \boxminus i}}};
    \end{tikzpicture}\ ,
  \end{equation}
\end{minipage}%
\begin{minipage}{0.5\linewidth}
  \begin{equation}
    \bS \left(
    \begin{tikzpicture}[>=stealth,baseline={([yshift=-.5ex]current bounding box.center)}]
      \draw[->] (0,0) node[anchor=south] {\strandlabel{i}} -- (0,-0.1) arc (180:360:.3) -- (0.6,0);
      \draw (.9,-0.3) node {\regionlabel{\lambda}};
    \end{tikzpicture}
    \right) = \frac{(|\lambda|+1)d_\lambda}{d_{\lambda \boxplus i}} \
    \begin{tikzpicture}[anchorbase]
      \draw[->] (0,0) -- (0,-0.1) arc (180:360:.5) -- (1,0);
      \node at (0.5,-0.2) {\regionlabel{\epsilon_{\lambda \boxplus i}}};
      \node at (1.2,-0.5) {\regionlabel{\epsilon_{\lambda}}};
    \end{tikzpicture}\ ,
  \end{equation}
\end{minipage}\par\vspace{\belowdisplayskip}

\noindent\begin{minipage}{0.5\linewidth}
  \begin{equation}
    \bS \left(
    \begin{tikzpicture}[>=stealth,baseline={([yshift=-.5ex]current bounding box.center)}]
      \draw[<-] (0,0) -- (0,0.1) arc (180:0:.3) -- (0.6,0) node[anchor=north] {\strandlabel{i}};
      \node at (0.9,0.3) {\regionlabel{\lambda}};
    \end{tikzpicture}
    \right) =\
    \begin{tikzpicture}[anchorbase]
      \draw[<-] (0,0) -- (0,0.1) arc (180:0:.5) -- (1,0);
      \node at (1.2,0.5) {\regionlabel{\epsilon_\lambda}};
      \node at (0.5,0.2) {\regionlabel{\epsilon_{\lambda \boxplus i}}};
    \end{tikzpicture}\ ,
  \end{equation}
\end{minipage}%
\begin{minipage}{0.5\linewidth}
  \begin{equation}
    \bS \left(
    \begin{tikzpicture}[>=stealth,baseline={([yshift=-.5ex]current bounding box.center)}]
      \draw[<-] (0,0) -- (0,-0.1) arc (180:360:.3) -- (0.6,0) node[anchor=south] {\strandlabel{i}};
      \draw (.9,-0.3) node {\regionlabel{\lambda}};
    \end{tikzpicture}
    \right) =\
    \begin{tikzpicture}[anchorbase]
      \draw[<-] (0,0) -- (0,-0.1) arc (180:360:.5) -- (1,0);
      \node at (0.5,-0.2) {\regionlabel{\epsilon_{\lambda \boxminus i}}};
      \node at (1.2,-0.5) {\regionlabel{\epsilon_{\lambda}}};
    \end{tikzpicture}\ ,
  \end{equation}
\end{minipage}\par\vspace{\belowdisplayskip}
\noindent where
\begin{equation}
  \xi_{i,j} = \frac{i-j}{i-j-1}.
\end{equation}
Note that
\begin{equation} \label{eq:xi-i-j-identity}
  \frac{1}{\xi_{i,j} \xi_{j,i}} = 1 - \frac{1}{(i-j)^2}.
\end{equation}

\begin{prop}
  The map $\bS$ described above is a well-defined 2-functor.
\end{prop}

\begin{proof}
  We first verify that $\bS$ respects isotopy invariance.  It is straightforward to verify that
  \[
    \bS \left(
    \begin{tikzpicture}[anchorbase]
      \draw (0,0) -- (0,0.5) .. controls (0,0.8) and (0.3,0.8) .. (0.3,0.5) .. controls (0.3,0.2) and (0.6,0.2) .. (0.6,0.5) -- (0.6,1);
    \end{tikzpicture}
    \right)
    = \bS \left(\
    \begin{tikzpicture}[anchorbase]
      \draw (0,0) -- (0,1);
    \end{tikzpicture}\
    \right)
    \qquad \text{and} \qquad
    \bS \left(
    \begin{tikzpicture}[anchorbase]
      \draw (0,0) -- (0,0.5) .. controls (0,0.8) and (-0.3,0.8) .. (-0.3,0.5) .. controls (-0.3,0.2) and (-0.6,0.2) .. (-0.6,0.5) -- (-0.6,1);
    \end{tikzpicture}
    \right)
    = \bS \left(\
    \begin{tikzpicture}[anchorbase]
      \draw (0,0) -- (0,1);
    \end{tikzpicture}\
    \right)
  \]
  with all possible orientations of the strands and all possible labelings of the regions and strands.  We also have
  \[
    \bS \left(
    \begin{tikzpicture}[scale=0.5,>=stealth,baseline={([yshift=-.5ex]current bounding box.center)}]
      \draw[->] (-1.5,1.5) node[anchor=south] {\strandlabel{i}} .. controls (-1.5,0.5) and (-1,-1) .. (0,0) .. controls (1,1) and (1.5,-0.5) .. (1.5,-1.5);
      \draw[->] (-2,1.5) node[anchor=south] {\strandlabel{j}} .. controls (-2,-2) and (1.5,-1.5) .. (0,0) .. controls (-1.5,1.5) and (2,2) .. (2,-1.5);
      \draw (2.3,0.5) node {\regionlabel{\lambda}};
    \end{tikzpicture}
    \right)
    =
    \xi_{i,j} \frac{d_\lambda}{n d_{\lambda \boxminus j}} \frac{d_{\lambda \boxminus j}}{(n-1){d_{\lambda \boxminus j \boxminus i}}} \frac{nd_{\lambda \boxminus i}}{d_\lambda} \frac{(n-1)d_{\lambda \boxminus i \boxminus j}}{d_{\lambda \boxminus i}}
    \begin{tikzpicture}[scale=0.5,>=stealth,baseline={([yshift=-.5ex]current bounding box.center)}]
      \draw[->] (-1.5,1.5) .. controls (-1.5,0.5) and (-1,-1) .. (0,0) .. controls (1,1) and (1.5,-0.5) .. (1.5,-1.5);
      \draw[->] (-2,1.5) .. controls (-2,-2) and (1.5,-1.5) .. (0,0) .. controls (-1.5,1.5) and (2,2) .. (2,-1.5);
      \node at (2.3,0.5) {\regionlabel{\epsilon_\lambda}};
      \node at (0.45,0.6) {\regionlabel{\epsilon_{\lambda \boxminus j}}};
      \node at (-0.45,-0.6) {\regionlabel{\epsilon_{\lambda \boxminus i}}};
      \node at (-1.2,-1.4) {\regionlabel{\epsilon_{\lambda \boxminus i \boxminus j}}};
    \end{tikzpicture}
    \ =
    \bS \left(
    \begin{tikzpicture}[anchorbase]
      \draw[<-] (0,0) node [anchor=north] {\strandlabel{i}} -- (0.5,0.5);
      \draw[<-] (0.5,0) node [anchor=north] {\strandlabel{j}} -- (0,0.5);
      \node at (0.7,0.25) {\regionlabel{\lambda}};
    \end{tikzpicture}
    \right).
  \]
  Similarly, one easily verifies that
  \begin{gather*}
    \bS \left(
    \begin{tikzpicture}[scale=0.5,>=stealth,baseline={([yshift=-.5ex]current bounding box.center)}]
      \draw[->] (1.5,1.5) node[anchor=south] {\strandlabel{j}} .. controls (1.5,0.5) and (1,-1) .. (0,0) .. controls (-1,1) and (-1.5,-0.5) .. (-1.5,-1.5);
      \draw[->] (2,1.5) node[anchor=south] {\strandlabel{i}} .. controls (2,-2) and (-1.5,-1.5) .. (0,0) .. controls (1.5,1.5) and (-2,2) .. (-2,-1.5);
      \draw (2.3,-0.6) node {\regionlabel{\lambda}};
    \end{tikzpicture}
    \right)
    =
    \bS \left(
    \begin{tikzpicture}[anchorbase]
      \draw[<-] (0,0) node [anchor=north] {\strandlabel{i}} -- (0.5,0.5);
      \draw[<-] (0.5,0) node [anchor=north] {\strandlabel{j}} -- (0,0.5);
      \node at (0.7,0.25) {\regionlabel{\lambda}};
    \end{tikzpicture}
    \right),
    \\
    \bS \left(\
    \begin{tikzpicture}[>=stealth,baseline={([yshift=-.5ex]current bounding box.center)}]
      \draw[->] (1,0) node[anchor=north] {\strandlabel{j}} .. controls (1,0.5) and (0,0.5) .. (0,1);
      \draw[->] (-0.5,1) .. controls (-0.5,0) and (0,0) .. (0.5,0.5) .. controls (1,1) and (1.5,1) .. (1.5,0) node[anchor=north] {\strandlabel{i}};
      \draw (1.7,0.5) node {\regionlabel{\lambda}};
    \end{tikzpicture}
    \right)
    =
    \bS \left(
    \begin{tikzpicture}[>=stealth,baseline={([yshift=-.5ex]current bounding box.center)}]
      \draw[<-] (-1,0) node[anchor=north] {\strandlabel{i}} .. controls (-1,0.5) and (0,0.5) .. (0,1);
      \draw[<-] (0.5,1) .. controls (0.5,0) and (0,0) .. (-0.5,0.5) .. controls (-1,1) and (-1.5,1) .. (-1.5,0) node[anchor=north] {\strandlabel{j}};
      \draw (0.7,0.5) node {\regionlabel{\lambda}};
    \end{tikzpicture}
    \right)
    =
    \bS \left(
    \begin{tikzpicture}[>=stealth,baseline={([yshift=-.5ex]current bounding box.center)}]
      \draw[->] (0,0) node[anchor=east] {\strandlabel{j}} -- (0.5,0.5);
      \draw[<-] (0.5,0) -- (0,0.5) node[anchor=east] {\strandlabel{i}};
      \draw (0.7,0.25) node {\regionlabel{\lambda}};
    \end{tikzpicture}
    \right),
    \\
    \bS \left(\
    \begin{tikzpicture}[>=stealth,baseline={([yshift=-.5ex]current bounding box.center)}]
      \draw[<-] (1,0) node[anchor=north] {\strandlabel{j}} .. controls (1,0.5) and (0,0.5) .. (0,1);
      \draw[<-] (-0.5,1) .. controls (-0.5,0) and (0,0) .. (0.5,0.5) .. controls (1,1) and (1.5,1) .. (1.5,0) node[anchor=north] {\strandlabel{i}};
      \draw (1.7,0.5) node {\regionlabel{\lambda}};
    \end{tikzpicture}
    \right)
    = \bS \left(
    \begin{tikzpicture}[>=stealth,baseline={([yshift=-.5ex]current bounding box.center)}]
      \draw[->] (-1,0) node[anchor=north] {\strandlabel{i}} .. controls (-1,0.5) and (0,0.5) .. (0,1);
      \draw[->] (0.5,1) .. controls (0.5,0) and (0,0) .. (-0.5,0.5) .. controls (-1,1) and (-1.5,1) .. (-1.5,0) node[anchor=north] {\strandlabel{j}};
      \draw (0.7,0.5) node {\regionlabel{\lambda}};
    \end{tikzpicture}
    \right)
    = \bS \left(
    \begin{tikzpicture}[>=stealth,baseline={([yshift=-.5ex]current bounding box.center)}]
      \draw[<-] (0,0) node[anchor=east] {\strandlabel{j}} -- (0.5,0.5);
      \draw[->] (0.5,0) -- (0,0.5) node[anchor=east] {\strandlabel{i}};
      \draw (0.8,0.25) node {\regionlabel{\lambda}};
    \end{tikzpicture}
    \right).
  \end{gather*}
  It follows that $\bS$ respects isotopy invariance.  It remains to check that $\bS$ respects the local relations \eqref{rel:up-up-up-braid}--\eqref{rel:ccc}.

  \smallskip

  \paragraph{\emph{Relation~\eqref{rel:up-up-up-braid}}:} We have
  \begin{multline*}
    \frac{1}{\xi_{i,j} \xi_{j,k} \xi_{i,k}} \bS \left(
    \begin{tikzpicture}[>=stealth,baseline={([yshift=.5ex]current bounding box.center)}]
      \draw (0,0) node [anchor=north] {\strandlabel{i}} -- (2,2)[->];
      \draw (2,0) node [anchor=north] {\strandlabel{k}} -- (0,2)[->];
      \draw[->] (1,0) node [anchor=north] {\strandlabel{j}} .. controls (0,1) .. (1,2);
      \node at (1.7,1) {\regionlabel{\lambda}};
    \end{tikzpicture}
    \right)
    =
    \begin{tikzpicture}[>=stealth,baseline={([yshift=.5ex]current bounding box.center)}]
      \draw (0,0) -- (3,3)[->];
      \draw (3,0) -- (0,3)[->];
      \draw[->] (1.5,0) .. controls (0,1.5) .. (1.5,3);
      \node at (2.55,1.5) {\regionlabel{\epsilon_\lambda}};
      \node at (1.73,0.6) {\regionlabel{\epsilon_{\lambda \boxplus k}}};
      \node at (1.73,2.4) {\regionlabel{\epsilon_{\lambda \boxplus i}}};
      \node at (-0.2,2.2) {\regionlabel{\epsilon_{\lambda \boxplus i \boxplus j \boxplus k}}};
      \node at (0.75,0) {\regionlabel{\epsilon_{\lambda \boxplus k \boxplus j}}};
      \node at (0.75,3) {\regionlabel{\epsilon_{\lambda \boxplus i \boxplus j}}};
      \node at (0.9,1.5) {\regionlabel{\epsilon_{\lambda \boxplus k \boxplus i}}};
    \end{tikzpicture}
    \ =\
    \begin{tikzpicture}[>=stealth,baseline={([yshift=.5ex]current bounding box.center)}]
      \draw (0,0) -- (3,3)[->];
      \draw (3,0) -- (0,3)[->];
      \draw[->] (1.5,0) .. controls (0,1.5) .. (1.5,3);
      \node at (2.55,1.5) {\regionlabel{\epsilon_\lambda}};
      \node at (1.73,0.6) {\regionlabel{\epsilon_{\lambda \boxplus k}}};
      \node at (1.73,2.4) {\regionlabel{\epsilon_{\lambda \boxplus i}}};
      \node at (-0.2,2.2) {\regionlabel{\epsilon_{\lambda \boxplus i \boxplus j \boxplus k}}};
      \node at (0.75,0) {\regionlabel{\epsilon_{\lambda \boxplus k \boxplus j}}};
      \node at (0.75,3) {\regionlabel{\epsilon_{\lambda \boxplus i \boxplus j}}};
    \end{tikzpicture}
    \\
    \stackrel{\eqref{eq:Kho-rel-braid}}{=}\
    \begin{tikzpicture}[>=stealth,baseline={([yshift=.5ex]current bounding box.center)}]
      \draw (0,0) -- (3,3)[->];
      \draw (3,0) -- (0,3)[->];
      \draw[->] (1.5,0) .. controls (3,1.5) .. (1.5,3);
      \node at (2.8,2.25) {\regionlabel{\epsilon_\lambda}};
      \node at (2.25,0.2) {\regionlabel{\epsilon_{\lambda \boxplus k}}};
      \node at (2.25,2.8) {\regionlabel{\epsilon_{\lambda \boxplus i}}};
      \node at (0.5,1.5) {\regionlabel{\epsilon_{\lambda \boxplus i \boxplus j \boxplus k}}};
      \node at (1.5,0.75) {\regionlabel{\epsilon_{\lambda \boxplus k \boxplus j}}};
      \node at (1.5,2.25) {\regionlabel{\epsilon_{\lambda \boxplus i \boxplus j}}};
    \end{tikzpicture}
    \ =\
    \begin{tikzpicture}[>=stealth,baseline={([yshift=.5ex]current bounding box.center)}]
      \draw (0,0) -- (3,3)[->];
      \draw (3,0) -- (0,3)[->];
      \draw[->] (1.5,0) .. controls (3,1.5) .. (1.5,3);
      \node at (2.8,2.25) {\regionlabel{\epsilon_\lambda}};
      \node at (2.25,0.2) {\regionlabel{\epsilon_{\lambda \boxplus k}}};
      \node at (2.25,2.8) {\regionlabel{\epsilon_{\lambda \boxplus i}}};
      \node at (0.5,1.5) {\regionlabel{\epsilon_{\lambda \boxplus i \boxplus j \boxplus k}}};
      \node at (1.5,0.75) {\regionlabel{\epsilon_{\lambda \boxplus k \boxplus j}}};
      \node at (1.5,2.25) {\regionlabel{\epsilon_{\lambda \boxplus i \boxplus j}}};
      \node at (2.25,1.5) {\regionlabel{\epsilon_{\lambda \boxplus j}}};
    \end{tikzpicture}
    = \frac{1}{\xi_{i,j} \xi_{j,k} \xi_{i,k}} \bS \left(
    \begin{tikzpicture}[>=stealth,baseline={([yshift=.5ex]current bounding box.center)}]
      \draw (0,0) node [anchor=north] {\strandlabel{i}} -- (3,3)[->];
      \draw (3,0) node [anchor=north] {\strandlabel{k}} -- (0,3)[->];
      \draw[->] (1.5,0) node [anchor=north] {\strandlabel{j}} .. controls (3,1.5) .. (1.5,3);
      \node at (2.8,2.25) {\regionlabel{\lambda}};
    \end{tikzpicture}
    \right),
  \end{multline*}
  where the second and fourth equalities follow from \eqref{eq:id1n-decomp} and Corollary~\ref{cor:epsilon_lambda-Q+-sandwich}.

  \smallskip

  \paragraph{\emph{Relation~\emph{\eqref{rel:up-up-double-cross}}:}}  For $\lambda \vdash n$, we have
  \begin{multline*}
    \frac{1}{\xi_{i,j} \xi_{j,i}} \bS \left(
    \begin{tikzpicture}[>=stealth,baseline={([yshift=.5ex]current bounding box.center)}]
      \draw[->] (0,0) node [anchor=north] {\strandlabel{i}} .. controls (1,1) .. (0,2);
      \draw[->] (1,0) node [anchor=north] {\strandlabel{j}} .. controls (0,1) .. (1,2);
      \node at (1.1,1.5) {\regionlabel{\lambda}};
    \end{tikzpicture}
    \right)
    =\
    \begin{tikzpicture}[anchorbase]
      \draw[->] (0,0) .. controls (1.3,1) .. (0,2);
      \draw[->] (1,0) .. controls (-0.3,1) .. (1,2);
      \node at (1.1,1.5) {\regionlabel{\epsilon_\lambda}};
      \node at (0.5,1) {\regionlabel{\epsilon_{\lambda \boxplus i}}};
      \node at (0.5,2) {\regionlabel{\epsilon_{\lambda \boxplus j}}};
      \node at (0.5,0) {\regionlabel{\epsilon_{\lambda \boxplus j}}};
      \node at (-0.3,1.6) {\regionlabel{\epsilon_{\lambda \boxplus i \boxplus j}}};
    \end{tikzpicture}
    \ =\
    \begin{tikzpicture}[anchorbase]
      \draw[->] (0,0) .. controls (1.3,1) .. (0,2);
      \draw[->] (1,0) .. controls (-0.3,1) .. (1,2);
      \node at (1.1,1.5) {\regionlabel{\epsilon_\lambda}};
      \node at (0.5,2) {\regionlabel{\epsilon_{\lambda \boxplus j}}};
      \node at (0.5,0) {\regionlabel{\epsilon_{\lambda \boxplus j}}};
      \node at (-0.3,1.6) {\regionlabel{\epsilon_{\lambda \boxplus i \boxplus j}}};
    \end{tikzpicture}
    \ -\
    \begin{tikzpicture}[anchorbase]
      \draw[->] (0,0) .. controls (1.3,1) .. (0,2);
      \draw[->] (1,0) .. controls (-0.3,1) .. (1,2);
      \node at (1.1,1.5) {\regionlabel{\epsilon_\lambda}};
      \node at (0.5,1) {\regionlabel{\epsilon_{\lambda \boxplus j}}};
      \node at (0.5,2) {\regionlabel{\epsilon_{\lambda \boxplus j}}};
      \node at (0.5,0) {\regionlabel{\epsilon_{\lambda \boxplus j}}};
      \node at (-0.3,1.6) {\regionlabel{\epsilon_{\lambda \boxplus i \boxplus j}}};
    \end{tikzpicture}
    \\
    \stackrel{\substack{\eqref{eq:Kho-rel-transposition-squared} \\ \eqref{eq:crossingRemoval}}}{=}
    \left( 1 - \frac{1}{(i-j)^2} \right)
    \begin{tikzpicture}[anchorbase]
      \draw[->] (0,0) -- (0,2);
      \draw[->] (1,0) -- (1,2);
      \node at (1.3,1) {\regionlabel{\epsilon_\lambda}};
      \node at (0.5,1) {\regionlabel{\epsilon_{\lambda \boxplus j}}};
      \node at (-0.5,1) {\regionlabel{\epsilon_{\lambda \boxplus i \boxplus j}}};
    \end{tikzpicture}
    \ \stackrel{\eqref{eq:xi-i-j-identity}}{=}
    \frac{1}{\xi_{i,j} \xi_{j,i}}
    \bS \left(
    \begin{tikzpicture}[>=stealth,baseline={([yshift=.5ex]current bounding box.center)}]
      \draw[->] (0,0) node [anchor=north] {\strandlabel{i}} -- (0,2);
      \draw[->] (0.6,0) node [anchor=north] {\strandlabel{j}} -- (0.6,2);
      \node at (1,1) {\regionlabel{\lambda}};
    \end{tikzpicture}
    \right),
  \end{multline*}
  where the second equality follows from \eqref{eq:id1n-decomp} and Corollary~\ref{cor:epsilon_lambda-Q+-sandwich}.

  \smallskip

  \paragraph{\emph{Relation~\eqref{rel:down-up-double-cross}}:} Let $\lambda \vdash n$ and $i,j \in \Z$, $i \ne j$.  By \eqref{rel:hl1} and \eqref{eq:xi-i-j-identity}, we have
  \[
    \frac{1}{\xi_{i,j} \xi_{j,i}}\,
    \frac{n d_{\lambda \boxminus i} d_{\lambda \boxplus j}}{(n+1) d_\lambda d_{\lambda \boxminus i \boxplus j}} = 1.
  \]
  Thus
  \begin{multline*}
    \bS \left(
    \begin{tikzpicture}[>=stealth,baseline={([yshift=.5ex]current bounding box.center)}]
      \draw[<-] (0,0) node [anchor=north] {\strandlabel{i}} .. controls (1,1) .. (0,2);
      \draw[->] (1,0) node [anchor=north] {\strandlabel{j}} .. controls (0,1) .. (1,2);
      \node at (1.1,1.5) {\regionlabel{\lambda}};
    \end{tikzpicture}
    \right)
    =\
    \begin{tikzpicture}[anchorbase]
      \draw[<-] (0,0) .. controls (1.3,1) .. (0,2);
      \draw[->] (1,0) .. controls (-0.3,1) .. (1,2);
      \node at (1.1,1.5) {\regionlabel{\epsilon_\lambda}};
      \node at (0.5,1) {\regionlabel{\epsilon_{\lambda \boxminus i}}};
      \node at (0.5,2) {\regionlabel{\epsilon_{\lambda \boxplus j}}};
      \node at (0.5,0) {\regionlabel{\epsilon_{\lambda \boxplus j}}};
      \node at (-0.3,1.6) {\regionlabel{\epsilon_{\lambda \boxminus i \boxplus j}}};
    \end{tikzpicture}
    \ =\
    \begin{tikzpicture}[anchorbase]
      \draw[<-] (0,0) .. controls (1.3,1) .. (0,2);
      \draw[->] (1,0) .. controls (-0.3,1) .. (1,2);
      \node at (1.1,1.5) {\regionlabel{\epsilon_\lambda}};
      \node at (0.5,2) {\regionlabel{\epsilon_{\lambda \boxplus j}}};
      \node at (0.5,0) {\regionlabel{\epsilon_{\lambda \boxplus j}}};
      \node at (-0.3,1.6) {\regionlabel{\epsilon_{\lambda \boxminus i \boxplus j}}};
    \end{tikzpicture}
    \\
    \stackrel{\eqref{eq:Kho-rel-down-up-double-cross}}{=}
    \begin{tikzpicture}[anchorbase]
      \draw[<-] (0,0) -- (0,2);
      \draw[->] (1,0) -- (1,2);
      \node at (1.3,1) {\regionlabel{\epsilon_\lambda}};
      \node at (0.5,1) {\regionlabel{\epsilon_{\lambda \boxplus j}}};
      \node at (-0.5,1) {\regionlabel{\epsilon_{\lambda \boxminus i \boxplus j}}};
    \end{tikzpicture}
    \ -\
    \begin{tikzpicture}[anchorbase]
      \draw[<-] (0,0) -- (0,0.2) arc(180:0:0.5) -- (1,0);
      \draw[->] (0,2) -- (0,1.8) arc(180:360:0.5) -- (1,2);
      \node at (0.5,1.75) {\regionlabel{\epsilon_{\lambda \boxplus j}}};
      \node at (0.5,0.25) {\regionlabel{\epsilon_{\lambda \boxplus j}}};
      \node at (0.5,1) {\regionlabel{\epsilon_{\lambda \boxminus i \boxplus j}\ \epsilon_\lambda}};
    \end{tikzpicture}
    \ =\
    \begin{tikzpicture}[anchorbase]
      \draw[<-] (0,0) -- (0,2);
      \draw[->] (1,0) -- (1,2);
      \node at (1.3,1) {\regionlabel{\epsilon_\lambda}};
      \node at (0.5,1) {\regionlabel{\epsilon_{\lambda \boxplus j}}};
      \node at (-0.5,1) {\regionlabel{\epsilon_{\lambda \boxminus i \boxplus j}}};
    \end{tikzpicture}
    \ =
    \bS \left(
    \begin{tikzpicture}[>=stealth,baseline={([yshift=.5ex]current bounding box.center)}]
      \draw[<-] (0,0) node [anchor=north] {\strandlabel{i}} -- (0,2);
      \draw[->] (0.6,0) node [anchor=north] {\strandlabel{j}} -- (0.6,2);
      \node at (1,1) {\regionlabel{\lambda}};
    \end{tikzpicture}
    \right),
  \end{multline*}
  where the second equality follows from \eqref{eq:id1n-decomp} and Corollary~\ref{cor:epsilon_lambda-Q+-sandwich}, since the only partition $\mu$ of size $n-1$ satisfying $\mu \subseteq \lambda \boxminus i \boxplus j$ and $\mu \subseteq \lambda$ is $\mu = \lambda \boxminus i$.

  \smallskip

  \paragraph{\emph{Relation~\eqref{rel:up-down-double-cross}}:}  Let $\lambda \vdash n$ and $i,j \in \Z$, $i \ne j$.  Again, using \eqref{rel:hl1} and \eqref{eq:xi-i-j-identity}, we have
  \[
    \bS \left(
    \begin{tikzpicture}[>=stealth,baseline={([yshift=.5ex]current bounding box.center)}]
      \draw[->] (0,0) node [anchor=north] {\strandlabel{i}} .. controls (1,1) .. (0,2);
      \draw[<-] (1,0) node [anchor=north] {\strandlabel{j}} .. controls (0,1) .. (1,2);
      \node at (1.1,1.5) {\regionlabel{\lambda}};
    \end{tikzpicture}
    \right)
    =\
    \begin{tikzpicture}[anchorbase]
      \draw[->] (0,0) .. controls (1.3,1) .. (0,2);
      \draw[<-] (1,0) .. controls (-0.3,1) .. (1,2);
      \node at (1.1,1.5) {\regionlabel{\epsilon_\lambda}};
      \node at (0.5,1) {\regionlabel{\epsilon_{\lambda \boxplus i}}};
      \node at (0.5,2) {\regionlabel{\epsilon_{\lambda \boxminus j}}};
      \node at (0.5,0) {\regionlabel{\epsilon_{\lambda \boxminus j}}};
      \node at (-0.3,1.6) {\regionlabel{\epsilon_{\lambda \boxplus i \boxminus j}}};
    \end{tikzpicture}
    \ =\
    \begin{tikzpicture}[anchorbase]
      \draw[->] (0,0) .. controls (1.3,1) .. (0,2);
      \draw[<-] (1,0) .. controls (-0.3,1) .. (1,2);
      \node at (1.1,1.5) {\regionlabel{\epsilon_\lambda}};
      \node at (0.5,2) {\regionlabel{\epsilon_{\lambda \boxminus j}}};
      \node at (0.5,0) {\regionlabel{\epsilon_{\lambda \boxminus j}}};
      \node at (-0.3,1.6) {\regionlabel{\epsilon_{\lambda \boxplus i \boxminus j}}};
    \end{tikzpicture}
    \stackrel{\eqref{eq:Kho-rel-up-down-double-cross}}{=}
    \begin{tikzpicture}[anchorbase]
      \draw[->] (0,0) -- (0,2);
      \draw[<-] (1,0) -- (1,2);
      \node at (1.3,1) {\regionlabel{\epsilon_\lambda}};
      \node at (0.5,1) {\regionlabel{\epsilon_{\lambda \boxminus j}}};
      \node at (-0.5,1) {\regionlabel{\epsilon_{\lambda \boxplus i \boxminus j}}};
    \end{tikzpicture}
    \ =
    \bS \left(
    \begin{tikzpicture}[>=stealth,baseline={([yshift=.5ex]current bounding box.center)}]
      \draw[->] (0,0) node [anchor=north] {\strandlabel{i}} -- (0,2);
      \draw[<-] (0.6,0) node [anchor=north] {\strandlabel{j}} -- (0.6,2);
      \node at (1,1) {\regionlabel{\lambda}};
    \end{tikzpicture}
    \right),
  \]
  where the second equality follows from \eqref{eq:id1n-decomp} and Corollary~\ref{cor:epsilon_lambda-Q+-sandwich}, since the only partition $\mu$ of size $n+1$ satisfying $\lambda \boxplus i \boxminus j \subseteq \mu$ and $\lambda \subseteq \mu$ is $\mu = \lambda \boxplus i$.

  \smallskip

  \paragraph{\emph{Relation~\eqref{rel:down-up-ii}}:} For $\lambda \vdash n$ and $i \in \Z$, we have
  \begin{multline*}
    \bS \left(
    \begin{tikzpicture}[>=stealth,baseline={([yshift=.5ex]current bounding box.center)}]
      \draw[<-] (0,0) node [anchor=north] {\strandlabel{i}} --(0,1.5);
      \draw[->] (0.5,0) node [anchor=north] {\strandlabel{i}} -- (0.5,1.5);
      \node at (0.8,0.75) {\regionlabel{\lambda}};
    \end{tikzpicture}
    \right)
    =\
    \begin{tikzpicture}[anchorbase]
      \draw[<-] (0,0) --(0,2);
      \draw[->] (0.8,0) -- (0.8,2);
      \node at (1.1,1) {\regionlabel{\epsilon_\lambda}};
      \node at (-0.3,1) {\regionlabel{\epsilon_\lambda}};
      \node at (0.4,1) {\regionlabel{\epsilon_{\lambda \boxplus i}}};
    \end{tikzpicture}
    \stackrel{\eqref{eq:Kho-rel-down-up-double-cross}}{=}\
    \begin{tikzpicture}[anchorbase]
      \draw[->] (0,2) -- (0,1.8) arc(180:360:0.5) -- (1,2);
      \draw[<-] (0,0) -- (0,0.2) arc(180:0:0.5) -- (1,0);
      \node at (0.5,1.7) {\regionlabel{\epsilon_{\lambda \boxplus i}}};
      \node at (0.5,0.3) {\regionlabel{\epsilon_{\lambda \boxplus i}}};
      \node at (0.5,1) {\regionlabel{\epsilon_\lambda}};
    \end{tikzpicture}
    +\
    \begin{tikzpicture}[anchorbase]
      \draw[<-] (0,0) .. controls (1.3,1) .. (0,2);
      \draw[->] (1,0) .. controls (-0.3,1) .. (1,2);
      \node at (1.1,1.6) {\regionlabel{\epsilon_\lambda}};
      \node at (0.5,2) {\regionlabel{\epsilon_{\lambda \boxplus i}}};
      \node at (0.5,0) {\regionlabel{\epsilon_{\lambda \boxplus i}}};
      \node at (-0.1,1.6) {\regionlabel{\epsilon_\lambda}};
    \end{tikzpicture}
    \ =\
    \begin{tikzpicture}[anchorbase]
      \draw[->] (0,2) -- (0,1.8) arc(180:360:0.5) -- (1,2);
      \draw[<-] (0,0) -- (0,0.2) arc(180:0:0.5) -- (1,0);
      \node at (0.5,1.7) {\regionlabel{\epsilon_{\lambda \boxplus i}}};
      \node at (0.5,0.3) {\regionlabel{\epsilon_{\lambda \boxplus i}}};
      \node at (0.5,1) {\regionlabel{\epsilon_\lambda}};
    \end{tikzpicture}
    + \sum_{j \in \Z}
    \begin{tikzpicture}[anchorbase]
      \draw[<-] (0,0) .. controls (1.3,1) .. (0,2);
      \draw[->] (1,0) .. controls (-0.3,1) .. (1,2);
      \node at (1.1,1.6) {\regionlabel{\epsilon_\lambda}};
      \node at (0.5,2) {\regionlabel{\epsilon_{\lambda \boxplus i}}};
      \node at (0.5,0) {\regionlabel{\epsilon_{\lambda \boxplus i}}};
      \node at (-0.1,1.6) {\regionlabel{\epsilon_\lambda}};
      \node at (0.5,1) {\regionlabel{\epsilon_{\lambda \boxminus j}}};
    \end{tikzpicture}
    \\
    \stackrel{\eqref{eq:crossingRemoval}}{=}
    \begin{tikzpicture}[anchorbase]
      \draw[->] (0,2) -- (0,1.8) arc(180:360:0.5) -- (1,2);
      \draw[<-] (0,0) -- (0,0.2) arc(180:0:0.5) -- (1,0);
      \node at (0.5,1.7) {\regionlabel{\epsilon_{\lambda \boxplus i}}};
      \node at (0.5,0.3) {\regionlabel{\epsilon_{\lambda \boxplus i}}};
      \node at (0.5,1) {\regionlabel{\epsilon_\lambda}};
    \end{tikzpicture}
    \ \left(
    1 + \sum_{j \in \Z} \frac{1}{(i-j)^2}\
    \begin{tikzpicture}[anchorbase]
      \draw[<-] (0.5,0) arc(0:360:0.5);
      \node at (0,0) {\regionlabel{\epsilon_{\lambda \boxminus j}}};
    \end{tikzpicture}
    \right)
    \stackrel{\eqref{eq:epsilon_lambda-clockwise}}{=}
    \left( 1 + \sum_{j \in \Z} \frac{n d_{\lambda \boxminus j}}{d_\lambda (i-j)^2} \right)
    \begin{tikzpicture}[anchorbase]
      \draw[->] (0,2) -- (0,1.8) arc(180:360:0.5) -- (1,2);
      \draw[<-] (0,0) -- (0,0.2) arc(180:0:0.5) -- (1,0);
      \node at (0.5,1.7) {\regionlabel{\epsilon_{\lambda \boxplus i}}};
      \node at (0.5,0.3) {\regionlabel{\epsilon_{\lambda \boxplus i}}};
      \node at (0.5,1) {\regionlabel{\epsilon_\lambda}};
    \end{tikzpicture}
    \\
    \stackrel{\eqref{rel:hl8}}{=} \frac{(n+1)d_\lambda}{d_{\lambda \boxplus i}}\
    \begin{tikzpicture}[anchorbase]
      \draw[->] (0,2) -- (0,1.8) arc(180:360:0.5) -- (1,2);
      \draw[<-] (0,0) -- (0,0.2) arc(180:0:0.5) -- (1,0);
      \node at (0.5,1.7) {\regionlabel{\epsilon_{\lambda \boxplus i}}};
      \node at (0.5,0.3) {\regionlabel{\epsilon_{\lambda \boxplus i}}};
      \node at (0.5,1) {\regionlabel{\epsilon_\lambda}};
    \end{tikzpicture}
    \ =\
    \bS \left(
    \begin{tikzpicture}[anchorbase]
      \draw[->] (0,1) node [anchor=south] {\strandlabel{i}} -- (0,0.9) arc (180:360:.25) -- (0.5,1);
      \draw[<-] (0,0) -- (0,0.1) arc (180:0:.25) -- (0.5,0) node [anchor=north] {\strandlabel{i}};
      \node at (0.7,0.5) {\regionlabel{\lambda}};
    \end{tikzpicture}
    \right)\ .
  \end{multline*}

  \smallskip

  \paragraph{\emph{Relation~\eqref{rel:up-down-ii}:}} For $\lambda \vdash n$ and $i \in \Z$, we have
  \begin{multline*}
    \bS \left(
    \begin{tikzpicture}[>=stealth,baseline={([yshift=.5ex]current bounding box.center)}]
      \draw[->] (0,0) node [anchor=north] {\strandlabel{i}} --(0,1.5);
      \draw[<-] (0.5,0) node [anchor=north] {\strandlabel{i}} -- (0.5,1.5);
      \node at (0.8,0.75) {\regionlabel{\lambda}};
    \end{tikzpicture}
    \right)
    =\
    \begin{tikzpicture}[anchorbase]
      \draw[->] (0,0) --(0,2);
      \draw[<-] (0.8,0) -- (0.8,2);
      \node at (1.1,1) {\regionlabel{\epsilon_\lambda}};
      \node at (-0.3,1) {\regionlabel{\epsilon_\lambda}};
      \node at (0.4,1) {\regionlabel{\epsilon_{\lambda \boxminus i}}};
    \end{tikzpicture}
    \stackrel{\eqref{eq:Kho-rel-up-down-double-cross}}{=}\
    \begin{tikzpicture}[anchorbase]
      \draw[->] (0,0) .. controls (1.3,1) .. (0,2);
      \draw[<-] (1,0) .. controls (-0.3,1) .. (1,2);
      \node at (1.1,1.6) {\regionlabel{\epsilon_\lambda}};
      \node at (0.5,2) {\regionlabel{\epsilon_{\lambda \boxminus i}}};
      \node at (0.5,0) {\regionlabel{\epsilon_{\lambda \boxminus i}}};
      \node at (-0.1,1.6) {\regionlabel{\epsilon_\lambda}};
    \end{tikzpicture}
    \ =\
    \sum_{j \in \Z}
    \begin{tikzpicture}[anchorbase]
      \draw[->] (0,0) .. controls (1.3,1) .. (0,2);
      \draw[<-] (1,0) .. controls (-0.3,1) .. (1,2);
      \node at (1.1,1.6) {\regionlabel{\epsilon_\lambda}};
      \node at (0.5,2) {\regionlabel{\epsilon_{\lambda \boxminus i}}};
      \node at (0.5,0) {\regionlabel{\epsilon_{\lambda \boxminus i}}};
      \node at (-0.1,1.6) {\regionlabel{\epsilon_\lambda}};
      \node at (0.5,1) {\regionlabel{\epsilon_{\lambda \boxplus j}}};
    \end{tikzpicture}
    \stackrel{\eqref{eq:crossingRemoval}}{=}
    \sum_{j \in \Z} \frac{1}{(j-i)^2}
    \begin{tikzpicture}[anchorbase]
      \draw[<-] (0,2) -- (0,1.8) arc(180:360:0.5) -- (1,2);
      \draw[->] (0,0) -- (0,0.2) arc(180:0:0.5) -- (1,0);
      \node at (0.5,1.7) {\regionlabel{\epsilon_{\lambda \boxminus i}}};
      \node at (0.5,0.3) {\regionlabel{\epsilon_{\lambda \boxminus i}}};
      \node at (0.5,1) {\regionlabel{\epsilon_\lambda}};
    \end{tikzpicture}
    \begin{tikzpicture}[anchorbase]
      \draw[->] (0.5,0) arc(0:360:0.5);
      \node at (0,0) {\regionlabel{\epsilon_{\lambda \boxplus j}}};
    \end{tikzpicture}
    \\
    \stackrel{\eqref{eq:epsilon_lambda-counter-clockwise}}{=}\
    \sum_{j \in \Z} \frac{d_{\lambda \boxplus j}}{(n+1)d_\lambda(j-i)^2}
    \begin{tikzpicture}[anchorbase]
      \draw[<-] (0,2) -- (0,1.8) arc(180:360:0.5) -- (1,2);
      \draw[->] (0,0) -- (0,0.2) arc(180:0:0.5) -- (1,0);
      \node at (0.5,1.7) {\regionlabel{\epsilon_{\lambda \boxminus i}}};
      \node at (0.5,0.3) {\regionlabel{\epsilon_{\lambda \boxminus i}}};
      \node at (0.5,1) {\regionlabel{\epsilon_\lambda}};
    \end{tikzpicture}
    \ \stackrel{\eqref{rel:hl5}}{=}\
    \frac{d_\lambda}{n d_{\lambda \boxminus i}}\
    \begin{tikzpicture}[anchorbase]
      \draw[<-] (0,2) -- (0,1.8) arc(180:360:0.5) -- (1,2);
      \draw[->] (0,0) -- (0,0.2) arc(180:0:0.5) -- (1,0);
      \node at (0.5,1.7) {\regionlabel{\epsilon_{\lambda \boxminus i}}};
      \node at (0.5,0.3) {\regionlabel{\epsilon_{\lambda \boxminus i}}};
      \node at (0.5,1) {\regionlabel{\epsilon_\lambda}};
    \end{tikzpicture}
    = \bS \left(
    \begin{tikzpicture}[anchorbase]
      \draw[<-] (0,1) node [anchor=south] {\strandlabel{i}} -- (0,0.9) arc (180:360:.25) -- (0.5,1);
      \draw[->] (0,0) -- (0,0.1) arc (180:0:.25) -- (0.5,0) node [anchor=north] {\strandlabel{i}};
      \node at (0.7,0.5) {\regionlabel{\lambda}};
    \end{tikzpicture}
    \right)\ .
  \end{multline*}

  \smallskip

  \paragraph{\emph{Relation~\eqref{rel:clockwise-i-circle}}:}  Suppose $\lambda \vdash n$.  If $i \not \in B^-(\lambda)$, then it clear that $\bS$ maps the left-hand side of \eqref{rel:clockwise-i-circle} to zero since $\epsilon_{\lambda \boxminus i} = \epsilon_\bzero = 0$.  If $i \in B^-(\lambda)$, we have
  \[
    \bS \left(
    \begin{tikzpicture}[anchorbase]
      \draw[-<] (0,0) arc (180:540:0.3) node [anchor=east] {\strandlabel{i}};
      \draw (0.8,0.1) node {\regionlabel{\lambda}};
    \end{tikzpicture}
    \right)
    = \frac{d_\lambda}{n d_{\lambda \boxminus i}}\
    \begin{tikzpicture}[anchorbase]
      \draw[<-] (-0.5,0) arc(180:540:0.5);
      \node at (0,0) {\regionlabel{\epsilon_{\lambda \boxminus i}}};
      \node at (0.7,0.4) {\regionlabel{\epsilon_\lambda}};
    \end{tikzpicture}
    \stackrel{\eqref{eq:epsilon_lambda-clockwise}}{=}
    \epsilon_\lambda.
  \]

  \smallskip

  \paragraph{\emph{Relation~\eqref{rel:ccc}}:}  Suppose $\lambda \vdash n$.  If $i \not \in B^+(\lambda)$, then it clear that $\bS$ maps the left-hand side of \eqref{rel:ccc} to zero since $\epsilon_{\lambda \boxplus i} = \epsilon_\bzero = 0$.  If $i \in B^+(\lambda)$, we have
  \[
    \bS \left(
    \begin{tikzpicture}[anchorbase]
      \draw[->] (0,0) arc (180:540:0.3) node [anchor=east] {\strandlabel{i}};
      \draw (0.8,0.1) node {\regionlabel{\lambda}};
    \end{tikzpicture}
    \right)
    = \frac{(n+1) d_\lambda}{d_{\lambda \boxplus i}}\
    \begin{tikzpicture}[anchorbase]
      \draw[->] (-0.5,0) arc(180:540:0.5);
      \node at (0,0) {\regionlabel{\epsilon_{\lambda \boxplus i}}};
      \node at (0.7,0.4) {\regionlabel{\epsilon_\lambda}};
    \end{tikzpicture}
    \stackrel{\eqref{eq:epsilon_lambda-counter-clockwise}}{=}
    \epsilon_\lambda. \qedhere
  \]
\end{proof}

The following theorem is one of the main results of the current paper.

\begin{theo} \label{theo:sU-sHepsilon-equivalence}
  The 2-functor $\bS \colon \sA \to \sH_\epsilon$ is an equivalence of 2-categories.
\end{theo}

\begin{proof}
  The 2-functor $\bS$ is essentially surjective on objects by Proposition~\ref{prop:sHepsilon-indecomposables}.  By Corollary~\ref{cor:epsilon_lambda-Q+-sandwich}, it is also essentially full on 1-morphisms by and full on 2-morphisms.  We will show in Corollary~\ref{cor:Psi-faithful-on-2-morphisms} that it is also faithful on 2-morphisms.
\end{proof}

\begin{cor} \label{cor:K(Htrunc)}
  We have an equivalence of categories $K(\sH^\trunc) \cong \bigoplus_{m = 0}^\infty \cA$.
\end{cor}

\begin{proof}
  This follows by combining Theorem~\ref{theo:sU-sHepsilon-equivalence} with Corollary \ref{cor:sH-infinite-decomp} and Theorem \ref{theo:sA-categorifies-cA}.
\end{proof}

\begin{rem} \label{rem:Kho-conjecture}
  Corollary~\ref{cor:K(Htrunc)} can be viewed as an analogue of Khovanov's Heisenberg categorification conjecture \cite[Conj.~1]{Kho14}.  In the framework of 2-categories, Khovanov's conjecture is that Grothendieck group of $\sH$ is isomorphic to $\bigoplus_{m \in \Z} H$.
\end{rem}

% ----------------------------------------------
\subsection{A $2$-functor from ${\sH^\trunc}$ to $\sA$} \label{subsec:principal-2-functor}
% ----------------------------------------------

It follows from Theorem~\ref{theo:sU-sHepsilon-equivalence} that we have an equivalence $\bT \colon \sH_\epsilon \to \sA$ of 2-categories, obtained by inverting $\bS$.  The domain of this 2-functor can be extended by zero to $\sH^\trunc \cong \sH_\epsilon \oplus \sH_\delta$.  Since $\sA$ is idempotent complete, it follows from the universal property of the idempotent completion that this 2-functor is uniquely determined by its restriction $\bT \colon {\sH^\trunc}' \to \sA$ (which we continue to denote by $\bT$).  For future reference, we describe this 2-functor explicitly.

The additive linear $2$-functor $\bT \colon {\sH^\trunc}' \to \sA$ is determined on objects by
\[
  \bzero \mapsto \bzero,\quad n \mapsto \bigoplus_{\lambda \vdash n} \lambda,\quad n \in \N,
\]
and on 1-morphisms by
\[
  \sQ_+ 1_n \mapsto \sum_{\lambda \vdash n} \sum_{i \in \Z} \sF_i 1_\lambda = \sum_{\lambda \vdash n} \sum_{i \in B^+(\lambda)} \sF_i 1_\lambda,\quad
  \sQ_- 1_n \mapsto \sum_{\lambda \vdash n} \sum_{i \in \Z} \sE_i 1_\lambda = \sum_{\lambda \vdash n} \sum_{i \in B^-(\lambda)} \sE_i 1_\lambda.
\]
On 2-morphisms, $\bT$ is given as follows:
\begin{equation} \label{eq:upcross-principal}
  \bT \left(
  \begin{tikzpicture}[>=stealth,baseline={([yshift=-.5ex]current bounding box.center)}]
    \draw [->](0,0) -- (0.5,0.5);
    \draw [->](0.5,0) -- (0,0.5);
    \draw (0.8,0.25) node {\regionlabel{n}};
  \end{tikzpicture}
  \right) = \sum_{\substack{\lambda \vdash n \\ i,j \in \Z}}
  \left( \frac{i-j-1}{i-j}\
  \begin{tikzpicture}[>=stealth,baseline={([yshift=.5ex]current bounding box.center)}]
    \draw [->](0,0) node[anchor=north] {\strandlabel{i}} -- (0.5,0.5);
    \draw [->](0.5,0) node[anchor=north] {\strandlabel{j}} -- (0,0.5);
    \draw (0.8,0.25) node {\regionlabel{\lambda}};
  \end{tikzpicture}
  + \frac{1}{i-j}\
  \begin{tikzpicture}[>=stealth,baseline={([yshift=.5ex]current bounding box.center)}]
    \draw [->](0,0) node[anchor=north] {\strandlabel{i}} -- (0,0.5);
    \draw [->](0.5,0) node[anchor=north] {\strandlabel{j}} -- (0.5,0.5);
    \draw (0.8,0.25) node {\regionlabel{\lambda}};
  \end{tikzpicture}
  \right)
\end{equation}

\details{
  Since $\bT$ maps $\delta_n$ to zero, we have
  \begin{multline*}
    \bT \left(
    \begin{tikzpicture}[>=stealth,baseline={([yshift=-.5ex]current bounding box.center)}]
      \draw [->](0,0) -- (0.5,0.5);
      \draw [->](0.5,0) -- (0,0.5);
      \draw (0.8,0.25) node {\regionlabel{n}};
    \end{tikzpicture}
    \right)
    =
    \sum_{\substack{\lambda \vdash n \\ i,j \in \Z}} \bT \left(
    \begin{tikzpicture}[>=stealth,baseline={([yshift=-.5ex]current bounding box.center)}]
      \draw [->](0,0) -- (1,1);
      \draw [->](1,0) -- (0,1);
      \node at (1,0.5) {\regionlabel{\epsilon_\lambda}};
      \node at (0.5,1) {\regionlabel{\epsilon_{\lambda \boxplus i}}};
      \node at (0.5,0) {\regionlabel{\epsilon_{\lambda \boxplus j}}};
      \node at (-0.2,0.5) {\regionlabel{\epsilon_{\lambda \boxplus i \boxplus j}}};
    \end{tikzpicture}
    \ +\
    \begin{tikzpicture}[>=stealth,baseline={([yshift=-.5ex]current bounding box.center)}]
      \draw [->](0,0) -- (1,1);
      \draw [->](1,0) -- (0,1);
      \node at (1,0.5) {\regionlabel{\epsilon_\lambda}};
      \node at (0.5,1) {\regionlabel{\epsilon_{\lambda \boxplus i}}};
      \node at (0.5,0) {\regionlabel{\epsilon_{\lambda \boxplus i}}};
      \node at (-0.2,0.5) {\regionlabel{\epsilon_{\lambda \boxplus i \boxplus j}}};
    \end{tikzpicture}
    \right)
    \\
    \stackrel{\eqref{eq:crossingRemoval}}{=}
    \sum_{\substack{\lambda \vdash n \\ i,j \in \Z}} \bT \left(
    \begin{tikzpicture}[>=stealth,baseline={([yshift=-.5ex]current bounding box.center)}]
      \draw [->](0,0) -- (1,1);
      \draw [->](1,0) -- (0,1);
      \node at (1,0.5) {\regionlabel{\epsilon_\lambda}};
      \node at (0.5,1) {\regionlabel{\epsilon_{\lambda \boxplus i}}};
      \node at (0.5,0) {\regionlabel{\epsilon_{\lambda \boxplus j}}};
      \node at (-0.2,0.5) {\regionlabel{\epsilon_{\lambda \boxplus i \boxplus j}}};
    \end{tikzpicture}
    \ + \frac{1}{j-i} \left(
    \begin{tikzpicture}[anchorbase]
      \draw[->] (0,0) -- (0,1.4);
      \draw[->] (1,0) -- (1,1.4);
      \node at (1.4,.7) {\regionlabel{\epsilon_\lambda}};
      \node at (.5,.7) {\regionlabel{\epsilon_{\lambda \boxplus i}}};
      \node at (-.6,.7) {\regionlabel{\epsilon_{\lambda\boxplus i \boxplus j}}};
    \end{tikzpicture}
    \right) \right)
    \\
    = \sum_{\substack{\lambda \vdash n \\ i,j \in \Z}}
    \left( \frac{i-j-1}{i-j}\
    \begin{tikzpicture}[>=stealth,baseline={([yshift=.5ex]current bounding box.center)}]
      \draw [->](0,0) node[anchor=north] {\strandlabel{i}} -- (0.5,0.5);
      \draw [->](0.5,0) node[anchor=north] {\strandlabel{j}} -- (0,0.5);
      \draw (0.8,0.25) node {\regionlabel{\lambda}};
    \end{tikzpicture}
    + \frac{1}{i-j}\
    \begin{tikzpicture}[>=stealth,baseline={([yshift=.5ex]current bounding box.center)}]
      \draw [->](0,0) node[anchor=north] {\strandlabel{i}} -- (0,0.5);
      \draw [->](0.5,0) node[anchor=north] {\strandlabel{j}} -- (0.5,0.5);
      \draw (0.8,0.25) node {\regionlabel{\lambda}};
    \end{tikzpicture}
    \right)
  \end{multline*}
  The proofs of the other equations below are similar.
}

\begin{equation} \label{eq:downcross-principal}
  \bT \left(
  \begin{tikzpicture}[>=stealth,baseline={([yshift=-.5ex]current bounding box.center)}]
    \draw [<-](0,0) -- (0.5,0.5);
    \draw [<-](0.5,0) -- (0,0.5);
    \draw (0.8,0.25) node {\regionlabel{n}};
  \end{tikzpicture}
  \right) = \sum_{\substack{\lambda \vdash n \\ i,j \in \Z}}
  \left( \frac{i-j-1}{i-j}\
  \begin{tikzpicture}[>=stealth,baseline={([yshift=.5ex]current bounding box.center)}]
    \draw [<-](0,0) node[anchor=north] {\strandlabel{i}} -- (0.5,0.5);
    \draw [<-](0.5,0) node[anchor=north] {\strandlabel{j}} -- (0,0.5);
    \draw (0.8,0.25) node {\regionlabel{\lambda}};
  \end{tikzpicture}
  + \frac{1}{i-j}\
  \begin{tikzpicture}[>=stealth,baseline={([yshift=.5ex]current bounding box.center)}]
    \draw [<-](0,0) node[anchor=north] {\strandlabel{i}} -- (0,0.5);
    \draw [<-](0.5,0) node[anchor=north] {\strandlabel{j}} -- (0.5,0.5);
    \draw (0.8,0.25) node {\regionlabel{\lambda}};
  \end{tikzpicture}
  \right)
\end{equation}

\begin{equation} \label{eq:rightcross-principal}
  \bT \left(
  \begin{tikzpicture}[>=stealth,baseline={([yshift=-.5ex]current bounding box.center)}]
    \draw[->] (0,0) -- (0.5,0.5);
    \draw[<-] (0.5,0) -- (0,0.5);
    \draw (0.7,0.25) node {\regionlabel{n}};
  \end{tikzpicture}
  \right) = \sum_{\substack{\lambda \vdash n \\ i,j \in \Z}}
  \left(
  \frac{n d_{\lambda \boxminus i} d_{\lambda \boxplus j}}{(n+1) d_\lambda d_{\lambda \boxplus j \boxminus i}}
  \frac{i-j-1}{i-j}\
  \begin{tikzpicture}[>=stealth,baseline={([yshift=-.5ex]current bounding box.center)}]
    \draw[->] (0,0) node[anchor=east] {\strandlabel{j}} -- (0.5,0.5);
    \draw[<-] (0.5,0) -- (0,0.5) node[anchor=east] {\strandlabel{i}};
    \draw (0.7,0.25) node {\regionlabel{\lambda}};
  \end{tikzpicture}
  + \frac{n d_{\lambda \boxminus j} d_{\lambda \boxplus i}}{(n+1) d_\lambda^2} \frac{1}{i-j}\
  \begin{tikzpicture}[>=stealth,baseline={([yshift=-.5ex]current bounding box.center)}]
    \draw[->] (0,1) node [anchor=south] {\strandlabel{i}} -- (0,0.9) arc (180:360:.25) -- (0.5,1);
    \draw[->] (0,0) node [anchor=north] {\strandlabel{j}} -- (0,0.1) arc (180:0:.25) -- (0.5,0);
    \node at (0.7,0.5) {\regionlabel{\lambda}};
  \end{tikzpicture}
  \right)
\end{equation}

\begin{equation} \label{eq:leftcross-principal}
  \bT \left(
  \begin{tikzpicture}[>=stealth,baseline={([yshift=-.5ex]current bounding box.center)}]
    \draw[<-] (0,0) -- (0.5,0.5);
    \draw[->] (0.5,0) -- (0,0.5);
    \draw (0.7,0.25) node {\regionlabel{n}};
  \end{tikzpicture}
  \right) = \sum_{\substack{\lambda \vdash n \\ i,j \in \Z}}
  \left( \frac{i-j-1}{i-j}\
  \begin{tikzpicture}[>=stealth,baseline={([yshift=-.5ex]current bounding box.center)}]
    \draw[<-] (0,0) node[anchor=east] {\strandlabel{j}} -- (0.5,0.5);
    \draw[->] (0.5,0) -- (0,0.5) node[anchor=east] {\strandlabel{i}};
    \draw (0.7,0.25) node {\regionlabel{\lambda}};
  \end{tikzpicture}
  + \frac{1}{i-j}\
  \begin{tikzpicture}[>=stealth,baseline={([yshift=-.5ex]current bounding box.center)}]
    \draw[<-] (0,1) -- (0,0.9) arc (180:360:.25) -- (0.5,1) node [anchor=south] {\strandlabel{j}};
    \draw[<-] (0,0) -- (0,0.1) arc (180:0:.25) -- (0.5,0) node [anchor=north] {\strandlabel{i}};
    \node at (0.7,0.5) {\regionlabel{\lambda}};
  \end{tikzpicture}
  \right)
\end{equation}

\noindent\begin{minipage}{0.5\linewidth}
  \begin{equation}
    \bT \left(
    \begin{tikzpicture}[>=stealth,baseline={([yshift=-.5ex]current bounding box.center)}]
      \draw[->] (0,0) -- (0,0.1) arc (180:0:.3) -- (0.6,0);
      \draw (.9,0.3) node {\regionlabel{n}};
    \end{tikzpicture}
    \right) = \sum_{\substack{\lambda \vdash n \\ i \in \Z}} \frac{n d_{\lambda \boxminus i}}{d_\lambda} \
    \begin{tikzpicture}[>=stealth,baseline={([yshift=.5ex]current bounding box.center)}]
      \draw[->] (0,0) node[anchor=north] {\strandlabel{i}} -- (0,0.1) arc (180:0:.3) -- (0.6,0);
      \draw (.9,0.3) node {\regionlabel{\lambda}};
    \end{tikzpicture}
  \end{equation}
\end{minipage}%
\begin{minipage}{0.5\linewidth}
  \begin{equation}
    \bT \left(
    \begin{tikzpicture}[>=stealth,baseline={([yshift=-.5ex]current bounding box.center)}]
      \draw[->] (0,0) -- (0,-0.1) arc (180:360:.3) -- (0.6,0);
      \draw (.9,-0.3) node {\regionlabel{n}};
    \end{tikzpicture}
    \right) = \sum_{\substack{\lambda \vdash n \\ i \in \Z}} \frac{d_{\lambda \boxplus i}}{(n+1)d_\lambda}\
    \begin{tikzpicture}[>=stealth,baseline={([yshift=-1.3ex]current bounding box.center)}]
      \draw[->] (0,0) node[anchor=south] {\strandlabel{i}} -- (0,-0.1) arc (180:360:.3) -- (0.6,0);
      \draw (.9,-0.3) node {\regionlabel{\lambda}};
    \end{tikzpicture}
  \end{equation}
\end{minipage}\par\vspace{\belowdisplayskip}

\noindent\begin{minipage}{0.5\linewidth}
  \begin{equation}
    \bT \left(
    \begin{tikzpicture}[>=stealth,baseline={([yshift=-.5ex]current bounding box.center)}]
      \draw[<-] (0,0) -- (0,0.1) arc (180:0:.3) -- (0.6,0);
      \draw (.9,0.3) node {\regionlabel{n}};
    \end{tikzpicture}
    \right) = \sum_{\substack{\lambda \vdash n \\ i \in \Z}} \
    \begin{tikzpicture}[>=stealth,baseline={([yshift=.5ex]current bounding box.center)}]
      \draw[<-] (0,0) -- (0,0.1) arc (180:0:.3) -- (0.6,0) node[anchor=north] {\strandlabel{i}};
      \draw (.9,0.3) node {\regionlabel{\lambda}};
    \end{tikzpicture}
  \end{equation}
\end{minipage}%
\begin{minipage}{0.5\linewidth}
  \begin{equation}
    \bT \left(
    \begin{tikzpicture}[>=stealth,baseline={([yshift=-.5ex]current bounding box.center)}]
      \draw[<-] (0,0) -- (0,-0.1) arc (180:360:.3) -- (0.6,0);
      \draw (0.9,-0.3) node {\regionlabel{n}};
    \end{tikzpicture}
    \right) = \sum_{\substack{\lambda \vdash n \\ i \in \Z}} \
    \begin{tikzpicture}[>=stealth,baseline={([yshift=-1.3ex]current bounding box.center)}]
      \draw[<-] (0,0) -- (0,-0.1) arc (180:360:.3) -- (0.6,0) node[anchor=south] {\strandlabel{i}};
      \draw (0.9,-0.3) node {\regionlabel{\lambda}};
    \end{tikzpicture}
  \end{equation}
\end{minipage}\par\vspace{\belowdisplayskip}
Note that, in \eqref{eq:upcross-principal}--\eqref{eq:leftcross-principal}, anytime a denominator is zero, the diagram it multiplies is also zero, and so we ignore such terms.  Note that the non-crossing terms in \eqref{eq:upcross-principal}--\eqref{eq:leftcross-principal} are a result of Lemma~\ref{lem:crossingRemoval}.  We compute the image of a diagram under $\bT$ by applying the above maps to each crossing, cup, and cap, where we interpret the composition of local diagrams where the strand or region labels do not match to be zero.  In this way, the image under $\bT$ of any diagram is a finite linear combination of diagrams in $\sA$.

%%%%%%%%%%%%%%%%%%%%%%%%%%%%%%%%%%%%%%%%%%%%%%%%%%%%%%%%%
%
\section{Actions on modules for symmetric groups} \label{sec:action}
%
%%%%%%%%%%%%%%%%%%%%%%%%%%%%%%%%%%%%%%%%%%%%%%%%%%%%%%%%%

%--------------------------------------------------------------------------------------------
\subsection{Induced actions and the principal realization} \label{subsec:induced-actions}
%--------------------------------------------------------------------------------------------

It follows from the results of Sections~\ref{sec:sA-def} and~\ref{sec:sHepsilon-A-equivalence} that categorified quantum group actions induce categorical Heisenberg actions and vice versa.  We describe this procedure here.  Recall the categorified quantum group $\sU$ (see Section~\ref{subsec:truncated-KL}) and the 2-category version $\sH$ of Khovanov's monoidal category (see Remark~\ref{rem:fH-discussion}).

Suppose we have an action of $\sU$ in an additive linear 2-category $\sC$, i.e. we have an additive linear 2-functor $\bA \colon \sU \to \sC$.  Recall that the set of objects of $\sU$ is the weight lattice of $\fsl_\infty$.  Suppose that $\bA$ categorifies the basic representation.  It follows that $\bA$ maps to zero all objects of $\sU$ that are weights not corresponding to weights of the basic representation.  Then $\bA$ factors through the truncation $\sU^\trunc$.  By Theorem~\ref{theo:sA-tsU-equivalence}, $\sA$ is equivalent to a sub-2-category of $\sU^\trunc$.  Therefore, we have the following commutative diagram of 2-functors:
\begin{equation}
  \vcenter{
    \xymatrix{
      & & & \sU \ar[dr]^\bA \ar[d] \\
      \sH \ar[r]^{\eqref{eq:truncation-fH-to-sH}} & \sH^\trunc \ar[r]^\bT & \sA \ar[r] & \sU^\trunc \ar[r] & \sC
    }
  }
\end{equation}
Thus, we obtain a 2-functor $\sH \to \sC$.  That is, we have an action of the 2-category version of Khovanov's monoidal category in $\sC$.  This is a categorical restriction of the basic representation to the principal Heisenberg subalgebra.

Conversely, suppose $\bB \colon \sH \to \sC$ is an additive linear 2-functor for some additive linear 2-category $\sC$ mapping all objects $n \in \Z$, $n < 0$, to zero.  Then $\bB$ factors through $\sH^\trunc$.  Since $\sA$ is equivalent to a quotient of $\sU^\trunc$, we have the following commutative diagram of 2-functors:
\begin{equation}
  \vcenter{
    \xymatrix{
      & & & \sH \ar[dr]^\bB \ar[d] \\
      \sU \ar[r] & \sU^\trunc \ar[r] & \sA \ar[r]^\bS & \sH^\trunc \ar[r] & \sC
    }
  }
\end{equation}
where we identify the 2-functor $\bS \colon \sA \to \sH_\epsilon$ with its composition with the inclusion $\sH_\epsilon \to \sH^\tr$.  Therefore we obtain a representation $\sU \to\sC$, which can be viewed as a categorical induced action.  An example of this will be developed in  detail in Section~\ref{subsec:H-action}.

Passing to Grothendieck groups, the 2-functor $\bB$ induces a representation of the Heisenberg algebra $H$ on $K(\sC)$.  If this representation is irreducible then, by the Stone--von Neumann Theorem, it is isomorphic to the Fock space representation.  Then the induced action of $\sU$ on $\sC$ categorifies the basic representation.  It follows from the results of Section~\ref{subsec:A-1morph} that the categories $\sA(\lambda,\mu)$ are semisimple (with at most one simple object) for all $\lambda,\mu \in \cP$.  Therefore, equalities of 1-morphisms in the Grothendieck group $\cA$ imply isomorphisms in $\sA$ of the corresponding 1-morphisms. It thus follows from Theorems~\ref{theo:sA-tsU-equivalence}, \ref{theo:sA-categorifies-cA}, and~\ref{theo:sU-sHepsilon-equivalence} that we have isomorphisms of 1-morphisms in $\sA$ corresponding to the expressions of generators of $\fsl_\infty$ in terms of generators of $H$ appearing in the principal realization (see Section~\ref{subsec:basic-rep}).  Therefore, the results of the current paper yield a categorification of the principal realization.

We note that these techniques are very different from those used in the categorification of the homogeneous realization of the basic representation in affine types ADE described in \cite{CL11}.  The main tools of \cite{CL11} are categorical vertex operators, which are certain complexes in Heisenberg 2-representations.  By contrast, our construction does not involve complexes and thus does not require passing to the homotopy category.

%--------------------------------------------------
\subsection{Action of $\sH^\trunc$} \label{subsec:H-action}
%--------------------------------------------------

We now describe an action of the 2-category $\sH^\trunc$ that arises from the action of Khovanov's Heisenberg category on modules for symmetric groups described in \cite[\S3.3]{Kho14}.

Recall the 2-category $\sM$ of Section~\ref{subsec:Sn-module-categories}.  We define an additive linear 2-functor $\bF_\sH \colon {\sH^\trunc}' \to \sM$ as follows.  On objects,
\begin{equation}
  \bF_\sH(\bzero) = \bzero,\quad
  \bF_\sH(n) = \cM_n,\quad n \in \N.
\end{equation}
On 1-morphisms, for $n \in \N$, we define,
\begin{gather}
  \bF_\sH(\sQ_+ 1_n) = (n+1)_n \otimes_{A_n} - \colon \cM_n \to \cM_{n+1}, \\
  \bF_\sH(\sQ_- 1_{n+1}) = \tensor*[_n]{(n+1)}{} \otimes_{A_{n+1}} - \colon \cM_{n+1} \to \cM_n.
\end{gather}

We now define $\bF_\sH$ on 2-morphisms.  The 2-functor $\bF_\sH$ will map 2-morphisms of ${\sH^\trunc}'$ to natural transformations of functors given by tensoring with bimodules.  These natural transformations are given by homomorphisms of the corresponding bimodules.  We define

\noindent\begin{minipage}{0.5\linewidth}
  \begin{equation}
    \bF_\sH \left(\,
    \begin{tikzpicture}[anchorbase]
      \draw [->](0,0) -- (0.5,0.5);
      \draw [->](0.5,0) -- (0,0.5);
      \draw (0.8,0.25) node {\regionlabel{n-1}};
    \end{tikzpicture}
    \right) = R_n,
  \end{equation}
\end{minipage}%
\begin{minipage}{0.5\linewidth}
  \begin{equation}
    \bF_\sH \left(
    \begin{tikzpicture}[anchorbase]
      \draw [<-](0,0) -- (0.5,0.5);
      \draw [<-](0.5,0) -- (0,0.5);
      \draw (-0.3,0.25) node {\regionlabel{n-1}};
    \end{tikzpicture}
    \, \right) = L_n,
  \end{equation}
\end{minipage}\par\vspace{\belowdisplayskip}

\noindent\begin{minipage}{0.5\linewidth}
  \begin{equation}
    \bF_\sH \left(\,
    \begin{tikzpicture}[anchorbase]
      \draw [->](0,0) -- (0.5,0.5);
      \draw [<-](0.5,0) -- (0,0.5);
      \draw (0.6,0.25) node {\regionlabel{n}};
    \end{tikzpicture}
    \right) = \rho,
  \end{equation}
\end{minipage}%
\begin{minipage}{0.5\linewidth}
  \begin{equation}
    \bF_\sH \left(\,
    \begin{tikzpicture}[anchorbase]
      \draw [<-](0,0) -- (0.5,0.5);
      \draw [->](0.5,0) -- (0,0.5);
      \draw (0.6,0.25) node {\regionlabel{n}};
    \end{tikzpicture}
    \right) = \tau,
  \end{equation}
\end{minipage}\par\vspace{\belowdisplayskip}

\noindent\begin{minipage}{0.5\linewidth}
  \begin{equation}
    \bF_\sH \left(\,
    \begin{tikzpicture}[anchorbase]
      \draw[<-] (0,-.6) -- (0,-.5) arc (0:180:.3) -- (-.6,-.6);
      \draw (0.4,-0.3) node {\regionlabel{n+1}};
    \end{tikzpicture}
    \right) = \varepsilon_\rR,
    \end{equation}
\end{minipage}%
\begin{minipage}{0.5\linewidth}
  \begin{equation}
    \bF_\sH \left(\,
    \begin{tikzpicture}[anchorbase]
      \draw[->] (0,.1) -- (0,0) arc (180:360:.3) -- (0.6,.1);
      \draw (.85,-.1) node {\regionlabel{n}};
    \end{tikzpicture}
    \right) = \eta_\rR,
  \end{equation}
\end{minipage}\par\vspace{\belowdisplayskip}

\noindent\begin{minipage}{0.5\linewidth}
  \begin{equation}
    \bF_\sH \left(\,
    \begin{tikzpicture}[anchorbase]
      \draw[->] (0,-.6) -- (0,-.5) arc (0:180:.3) -- (-.6,-.6);
      \draw (0.2,-0.3) node {\regionlabel{n}};
    \end{tikzpicture}
    \right) = \varepsilon_\rL,
    \end{equation}
\end{minipage}%
\begin{minipage}{0.5\linewidth}
  \begin{equation}
    \bF_\sH \left(\,
    \begin{tikzpicture}[anchorbase]
      \draw[<-] (0,.1) -- (0,0) arc (180:360:.3) -- (0.6,.1);
      \draw (1,-.1) node {\regionlabel{n+1}};
    \end{tikzpicture}
    \right) = \eta_\rL,
  \end{equation}
\end{minipage}\par\vspace{\belowdisplayskip}
\noindent where
\begin{gather}
  R_n \colon (n+1)_{n-1} \to (n+1)_{n-1},\quad a \mapsto as_n, \label{eq:Rn-def} \\
  L_n \colon \tensor*[_{n-1}]{(n+1)}{} \to \tensor*[_{n-1}]{(n+1)}{},\quad a \mapsto s_na, \label{eq:Ln-def}
\end{gather}
$\rho$ and $\tau$ are defined in \eqref{eq:rho-def} and \eqref{eq:tau-def}, respectively, and $\varepsilon_\rR$, $\eta_\rR$, $\varepsilon_\rL$, and $\eta_\rL$ are the adjunction maps defined in Proposition~\ref{prop:adjunction-maps}.  It follows from the results of \cite[\S3.3]{Kho14} that $\bF_\sH$ respects the local relations and topological invariance in the definition of $\sH^\trunc$.

Since the only idempotent 1-morphisms in $\sM$ are the identity 1-morphisms and all idempotent 2-morphisms in $\sM$ split, the 2-functor $\bF_\sH$ induces a 2-functor (which we denote by the same symbol)
\[
  \bF_\sH \colon \sH^\trunc \to \sM.
\]

\begin{lem} \label{lem:epsilon_lambda-action}
  The 2-functor $\bF_\sH$ maps the 2-morphism $\epsilon_\lambda$ to the bimodule map $(n) \to (n)$ given by multiplication by the central idempotent $e_\lambda$.
\end{lem}

\begin{proof}
  We compute that $\bF_\sH(\epsilon_\lambda)$ is the bimodule map $(n) \to (n)$ given by
  \[
    a
    \mapsto a \frac{1}{n!} \sum_{w \in S_n} w e_\lambda w^{-1}
    = a e_\lambda. \qedhere
  \]
\end{proof}

\begin{cor}
  The 2-functor $\bF_\sH$ maps the object $(n,1_n,\epsilon_\lambda)$ of $\sH^\trunc$ to $\cM_\lambda$ and maps $\sH_\delta$ to zero.
\end{cor}

%---------------------------
\subsection{Action of $\sA$} \label{subsec:A-action}
%---------------------------

The composition
\[
  \bF_\sA := \bF_\sH \circ \bS
\]
is an additive linear 2-functor that defines an action of $\sA$ on modules for symmetric groups.  For future reference, we describe this 2-functor explicitly here.

On objects, we have
\begin{equation}
  \bF_\sA(\bzero) = \bzero,\quad
  \bF_\sA(\lambda) = \cM_\lambda,\quad \lambda \in \cP.
\end{equation}

On 1-morphisms, for $\lambda \in \cP$ with an addable $i$-box, it follows from \eqref{eq:i-induction-idempotent-decomp} that
\begin{equation} \label{eq:FU-action-F_i}
  \bF_\sA(\sF_i 1_\lambda) = (n+1)_n^i \otimes_{A_n} - \colon \cM_\lambda \to \cM_{\lambda \boxplus i}, \quad \text{where } n = |\lambda|.
\end{equation}
Similarly, for $\lambda \in \cP$ with a removable $i$-box, it follows from \eqref{eq:i-restriction-idempotent-decomp} that
\begin{equation} \label{eq:FU-action-E_i}
  \bF_\sA(\sE_i 1_\lambda) = \tensor*[_{n-1}^i]{(n)}{} \otimes_{A_n} - \colon \cM_\lambda \to \cM_{\lambda \boxminus i},\quad \text{where } n = |\lambda|.
\end{equation}

We now describe $\bF_\sA$ on 2-morphisms.  The 2-functor $\bF_\sA$ maps 2-morphisms of $\sA$ to natural transformations of functors given by tensoring with bimodules.  These natural transformations are given by homomorphisms of the corresponding bimodules.  For  $i,j \in \Z$ with $i \ne j$, and $\lambda \vdash n$, we have

\noindent\begin{minipage}{0.5\linewidth}
  \begin{equation} \label{eq:upcross-action}
    \bF_\sA \left(
    \begin{tikzpicture}[>=stealth,baseline={([yshift=0]current bounding box.center)}]
      \draw [->](0,0) node[anchor=north] {\strandlabel{i}} -- (0.5,0.5);
      \draw [->](0.5,0) node[anchor=north] {\strandlabel{j}} -- (0,0.5);
      \draw (0.7,0.25) node {\regionlabel{\lambda}};
    \end{tikzpicture}
    \right)
    = R_{n+1}|_{i,j}^{j,i},
  \end{equation}
\end{minipage}%
\begin{minipage}{0.5\linewidth}
  \begin{equation} \label{eq:downcross-action}
    \bF_\sA \left(
    \begin{tikzpicture}[>=stealth,baseline={([yshift=0]current bounding box.center)}]
      \draw [<-](0,0) node[anchor=north] {\strandlabel{i}} -- (0.5,0.5);
      \draw [<-](0.5,0) node[anchor=north] {\strandlabel{j}} -- (0,0.5);
      \draw (0.7,0.25) node {\regionlabel{\lambda}};
    \end{tikzpicture}
    \right)
    = L_{n-1}|_{i,j}^{j,i},
  \end{equation}
\end{minipage}\par\vspace{\belowdisplayskip}

\begin{equation} \label{eq:F-rightcross}
  \bF_\sA \left(
  \begin{tikzpicture}[>=stealth,baseline={([yshift=-.5ex]current bounding box.center)}]
    \draw[->] (0,0) node[anchor=east] {\strandlabel{j}} -- (0.5,0.5);
    \draw[<-] (0.5,0) -- (0,0.5) node[anchor=east] {\strandlabel{i}};
    \draw (0.7,0.25) node {\regionlabel{\lambda}};
  \end{tikzpicture}
  \right) = \frac{i-j}{i-j-1} \frac{(n+1)d_\lambda d_{\lambda \boxminus i \boxplus j}}{|\lambda| d_{\lambda \boxminus i} d_{\lambda \boxplus j}} \rho|_{j,i}^{i,j},
\end{equation}
\begin{equation} \label{eq:F-leftcross}
  \bF_\sA \left(
  \begin{tikzpicture}[>=stealth,baseline={([yshift=-.5ex]current bounding box.center)}]
    \draw[<-] (0,0) node[anchor=east] {\strandlabel{j}} -- (0.5,0.5);
    \draw[->] (0.5,0) -- (0,0.5) node[anchor=east] {\strandlabel{i}};
    \draw (0.7,0.25) node {\regionlabel{\lambda}};
  \end{tikzpicture}
  \right) = \frac{i-j}{i-j-1} \tau|_{j,i}^{i,j},
\end{equation}
\noindent\begin{minipage}{0.5\linewidth}
  \begin{equation} \label{eq:FA-right-cap}
    \bF_\sA \left(
    \begin{tikzpicture}[>=stealth,baseline={([yshift=-.5ex]current bounding box.center)}]
      \draw[->] (0,0) node[anchor=north] {\strandlabel{i}} -- (0,0.1) arc (180:0:.3) -- (0.6,0);
      \draw (.9,0.3) node {\regionlabel{\lambda}};
    \end{tikzpicture}
    \right) = \frac{d_\lambda}{n d_{\lambda \boxminus i}} \varepsilon_\rR^{i,i},
    \end{equation}
\end{minipage}%
\begin{minipage}{0.5\linewidth}
  \begin{equation} \label{eq:FU-right-cup}
    \bF_\sA \left(
    \begin{tikzpicture}[>=stealth,baseline={([yshift=-.5ex]current bounding box.center)}]
      \draw[->] (0,0) node[anchor=south] {\strandlabel{i}} -- (0,-0.1) arc (180:360:.3) -- (0.6,0);
      \draw (.9,-0.3) node {\regionlabel{\lambda}};
    \end{tikzpicture}
    \right) = \frac{(n+1)d_\lambda}{d_{\lambda \boxplus i}} \eta_\rR^{i,i},
  \end{equation}
\end{minipage}\par\vspace{\belowdisplayskip}

\noindent\begin{minipage}{0.5\linewidth}
  \begin{equation}
    \bF_\sA \left(
    \begin{tikzpicture}[>=stealth,baseline={([yshift=-.5ex]current bounding box.center)}]
      \draw[<-] (0,0) -- (0,0.1) arc (180:0:.3) -- (0.6,0) node[anchor=north] {\strandlabel{i}};
      \draw (.9,0.3) node {\regionlabel{\lambda}};
    \end{tikzpicture}
    \right) = \varepsilon_\rL^{i,i},
  \end{equation}
\end{minipage}%
\begin{minipage}{0.5\linewidth}
  \begin{equation} \label{eq:FA-left-cup}
    \bF_\sA \left(
    \begin{tikzpicture}[>=stealth,baseline={([yshift=-.5ex]current bounding box.center)}]
      \draw[<-] (0,0) -- (0,-0.1) arc (180:360:.3) -- (0.6,0) node[anchor=south] {\strandlabel{i}};
      \draw (0.9,-0.3) node {\regionlabel{\lambda}};
    \end{tikzpicture}
    \right) = \eta_\rL^{i,i}.
  \end{equation}
\end{minipage}\par\vspace{\belowdisplayskip}
\noindent where, for $i,j,k,\ell \in \Z$, we define the components
\begin{gather}
  R_{n+1}|_{i,j}^{k,\ell} \colon (n+2)_n^{i,j} \hookrightarrow (n+2)_n \xrightarrow{R_{n+1}} (n+2)_n \twoheadrightarrow (n+2)_n^{k,\ell}, \\
  L_{n-1}|_{i,j}^{k,\ell} \colon \tensor*[_{n-2}^{i,j}]{(n)}{} \hookrightarrow {}_{n-2}(n) \xrightarrow{L_{n-1}} {}_{n-2}(n) \twoheadrightarrow \tensor*[_{n-2}^{k,\ell}]{(n)}{}, \\
  \rho|_{j,i}^{k,\ell} \colon (n)_{n-1}^j {}^i(n) \hookrightarrow (n)_{n-1}(n) \xrightarrow{\rho} {}_n(n+1)_n \twoheadrightarrow {}_n^k (n+1)_n^\ell, \\
  \tau|_{j,i}^{k,\ell} \colon {}_n^j(n+1)_n^i \hookrightarrow {}_n(n+1)_n \xrightarrow{\tau} (n)_{n-1}(n) \twoheadrightarrow (n)_{n-1}^k{}^\ell(n), \\
  \varepsilon_\rR^{i,j} \colon (n+1)_n^i\tensor*[^j]{(n+1)}{} \hookrightarrow (n+1)_n(n+1) \xrightarrow{\varepsilon_\rR} (n+1), \label{eq:varepsilonRij-def} \\
  \eta_\rR^{i,j} \colon (n) \xrightarrow{\eta_\rR} \tensor*[_n]{(n+1)}{_n} \twoheadrightarrow \tensor*[_n^i]{(n+1)}{_n^j}, \label{eq:etaRij-def} \\
  \varepsilon_\rL^{i,j} \colon \tensor*[_n^i]{(n+1)}{_n^j} \hookrightarrow \tensor*[_n]{(n+1)}{_n} \xrightarrow{\varepsilon_\rL} (n), \label{eq:varepsilonLij-def} \\
  \eta_\rL^{i,j} \colon (n+1) \xrightarrow{\eta_\rL} (n+1)_n(n+1) \twoheadrightarrow (n+1)_n^i\tensor*[^j]{(n+1)}{}. \label{eq:etaLij-def}
\end{gather}
Note that any time the denominator in a coefficient in \eqref{eq:F-rightcross}--\eqref{eq:FU-right-cup} is zero, the corresponding diagram is zero and so we can ignore such expressions.  For instance, in \eqref{eq:F-rightcross}, when $i = j+1$, the crossing of strands colored $i$ and $j$ is zero (see Remark~\ref{rem:double-cross-diff-by-one}).

It is possible to write the equations \eqref{eq:upcross-action}--\eqref{eq:F-leftcross} in a manner that avoids one of the eigenspace projections.  For example, we have
\begin{multline}
  \bF_\sA \left(
  \begin{tikzpicture}[>=stealth,baseline={([yshift=0]current bounding box.center)}]
    \draw [->](0,0) node[anchor=north] {\strandlabel{i}} -- (0.5,0.5);
    \draw [->](0.5,0) node[anchor=north] {\strandlabel{j}} -- (0,0.5);
    \draw (0.7,0.25) node {\regionlabel{\lambda}};
  \end{tikzpicture}
  \right)=
  \bF_\sH \left(
  \begin{tikzpicture}[anchorbase]
    \draw[->] (0,0) -- (1.4,1.4);
    \draw[->] (1.4,0) -- (0,1.4);
    \node at (1.2,.7) {\regionlabel{\epsilon_\lambda}};
    \node at (.7,.2) {\regionlabel{\epsilon_{\lambda\boxplus j}}};
    \node at (.7,1.2) {\regionlabel{\epsilon_{\lambda \boxplus i}}};
    \node at (0,.7) {\regionlabel{\epsilon_{\lambda\boxplus i \boxplus j}}};
  \end{tikzpicture}
  \right) = \bF_\sH \left(
  \begin{tikzpicture}[anchorbase]
    \draw[->] (0,0) -- (1.4,1.4);
    \draw[->] (1.4,0) -- (0,1.4);
    \node at (1.2,.7) {\regionlabel{\epsilon_\lambda}};
    \node at (.7,.2) {\regionlabel{\epsilon_{\lambda\boxplus j}}};
    \node at (0,.7) {\regionlabel{\epsilon_{\lambda\boxplus i \boxplus j}}};
  \end{tikzpicture}
  -
  \begin{tikzpicture}[anchorbase]
    \draw[->] (0,0) -- (1.4,1.4);
    \draw[->] (1.4,0) -- (0,1.4);
    \node at (1.2,.7) {\regionlabel{\epsilon_\lambda}};
    \node at (.7,.2) {\regionlabel{\epsilon_{\lambda\boxplus j}}};
    \node at (.7,1.2) {\regionlabel{\epsilon_{\lambda \boxplus j}}};
    \node at (0,.7) {\regionlabel{\epsilon_{\lambda\boxplus i \boxplus j}}};
  \end{tikzpicture}
  \right)
  \\
  = \frac{i-j}{i-j-1} R_{n+1} {}^r(e_{\lambda \boxplus i \boxplus j} e_{\lambda \boxplus j} e_\lambda) - \frac{1}{i-j-1},
\end{multline}
where ${}^r z$ denotes right multiplication by an element $z$, so that ${}^r(e_{\lambda \boxplus i \boxplus j} e_{\lambda \boxplus j} e_\lambda)$ is projection onto $(n+2)_n^{i,j} \subseteq (n+2)_n$.

\begin{prop} \label{prop:sU-sM-equivalence}
  The 2-functor $\bF_\sA = \bF_\sH \circ \bS \colon \sA \to \sM$ is an equivalence of $2$-categories.
\end{prop}

\begin{proof}
  By definition, $\bF_\sA$ is essentially surjective on objects.  Consider $\lambda,\mu\in\cP$.  By Proposition~\ref{prop:A-1-morph-spaces} any $1$-morphism in $\1Mor_{\sA}(\lambda,\mu)$ is isomorphic to a multiple of one of the form $P=\sF_{i_1} \sF_{i_2} \dotsm \sF_{i_k} \sE_{j_1} \sE_{j_2} \dotsm \sE_{j_\ell} 1_\lambda$.

  Recall that we have canonical equivalences $\sM_\lambda\cong \cV$ and  $\sM_\mu\cong \cV$ (see Section \ref{subsec:Sn-module-categories}).  Under these equivalences, $\1Mor_{\sM}(\cM_\lambda,\cM_\mu)=\{1_\cV^{\oplus n}:n\geq0\}$, and $\bF_\sA(P) = 1_\cV$.  It follows that $\bF_\sA$ is essentially full on $1$-morphisms.

  Given $\mathsf{P},\mathsf{Q} \in \1Mor_{\sA}(\lambda,\mu)$ two nonzero $1$-morphisms as above, by Proposition~\ref{prop:matching-2-morphisms} we have that $\2Mor_{\sA}(\mathsf{P},\mathsf{Q})$ is one-dimensional.  Since $\2Mor(1_\cV, 1_\cV)$ is also one-dimensional and $\bF_\sA$ preserves 2-isomorphisms, it follows that $\bF_\sA$ induces an isomorphism
  \[
    \2Mor_\sA(\mathsf{P},\mathsf{Q}) \cong \2Mor_{\sM}(\bF_\sA(\mathsf{P}),\bF_\sA(\mathsf{Q})).
  \]
  By linearity we conclude that $\bF_\sA$ is fully faithful on 2-morphisms.  Thus $\bF_\sA$ is an equivalence.
\end{proof}

We can now complete the proofs of Theorems~\ref{theo:sU-sHepsilon-equivalence} and~\ref{theo:sA-categorifies-cA}.

\begin{cor} \label{cor:Psi-faithful-on-2-morphisms}
  The 2-functor $\bS$ is faithful on 2-morphisms.
\end{cor}

\begin{cor} \label{cor:dotU0-KA-injective}
  The functor $\cA \to K(\sA)$ of Theorem~\ref{theo:sA-categorifies-cA} is faithful.
\end{cor}

\begin{proof}
  We have commutative diagram
  \begin{equation} \label{eq:action-comm-diag}
    \vcenter{
      \xymatrix{
        \cA \ar[r]^{\bfr} \ar[d] & \cV \\
        K(\sA) \ar[r]^{[\bF_\sA]} & K(\sM) \ar[u]
      }
    }
  \end{equation}
  Thus, the corollary follows from the fact that the functor $\bfr$ of \eqref{eq:rhoA-def} is faithful.
\end{proof}

%-------------------------------------------------
\subsection{Action of categorified quantum groups} \label{sec:action-cat-quantum-group}
%-------------------------------------------------

From the results of Sections~\ref{subsec:truncated-KL} and~\ref{subsec:A-action}, we immediately obtain an explicit action of the 2-category $\sU$ of \cite{CL15} (the categorified quantum group of type $A_\infty$) on modules for symmetric groups. This categorifies the fundamental representation $L(\Lambda_0)$ of $\fsl_\infty$.

For ease of reference, we describe this action here, which is an additive linear 2-functor
\[
  \bF_\sU \colon \sU \to \sM.
\]
Recall that the set of objects of $\sU$ is the free monoid on the weight lattice of $\fsl_\infty$ and recall the definition of $\omega_\lambda$ in \eqref{eq:omega_lambda-def}.  On objects, we have
\[
  \bF_\sU(x)=
  \begin{cases}
    \cM_\lambda & \text{if } x = \omega_\lambda,\ \lambda \in \cP, \\
    \bzero & \text{if } x \text{ is not of the form } \omega_\lambda \text{ for any } \lambda \in \cP.
  \end{cases}
\]
On 1-morphisms, $\bF_\sU$ acts just as $\bF_\sA$ does in \eqref{eq:FU-action-F_i} and \eqref{eq:FU-action-E_i}.  On 2-morphisms, $\bF_\sU$ maps any diagram with dots to zero.  On diagrams without dots, $\bF_\sU$ acts just as $\bF_\sA$ does in \eqref{eq:upcross-action}--\eqref{eq:FA-left-cup}, but with orientations of strands reversed (see Theorem~\ref{theo:sA-tsU-equivalence}).

\begin{rem}
  In \cite{BK09}, Brundan and Kleshchev constructed an explicit isomorphism between blocks of cyclotomic Hecke algebras and sign-modified cyclotomic Khovanov--Lauda algebras in type $A$.  They then used this isomorphism to describe actions on categories of modules for cyclotomic Hecke algebras in \cite{BK09b}.  This is related to the action described above, using \eqref{eq:upcross-action}--\eqref{eq:FA-left-cup}, since level one cyclotomic Hecke algebras are isomorphic to group algebras of symmetric groups.
\end{rem}

%%%%%%%%%%%%%%%%%%%%%%%%%%%%%%%%%%%%%%%%%%%%%
%
\section{Applications and further directions} \label{sec:applications}
%
%%%%%%%%%%%%%%%%%%%%%%%%%%%%%%%%%%%%%%%%%%%%%

%------------------------------------
\subsection{Diagrammatic computation} \label{subsec:diagrammatic-computation}
%------------------------------------

As an application of the constructions of the current paper, we give some examples of how one can prove combinatorial identities related to the dimensions of modules for symmetric groups using the diagrammatics of the categories introduced above.

\begin{prop} \label{prop:application-identities}
  If $\lambda$ is a partition, then
  \begin{equation}\label{eq:application-identity1}
    \sum_{j\in B^+(\lambda)} \frac{d_{\lambda \boxplus j}}{j-i}=0 \quad \forall\ i\in B^{-}(\lambda),\quad \text{and}
  \end{equation}
  \begin{equation}\label{eq:application-identity2}
    \sum_{j\in B^-(\lambda)}\frac{d_{\lambda \boxminus j}}{i-j}=i\frac{d_{\lambda}}{|\lambda|}\quad \forall\ i\in B^+(\lambda).
  \end{equation}
\end{prop}

\begin{proof}
  For $i \in B^-(\lambda)$, we have
  \begin{multline*}
    0 \stackrel{\eqref{eq:Kho-rel-left-curl}}{=} \
    \begin{tikzpicture}[anchorbase]
      \draw (0,0) .. controls (0,.5) and (.7,.5) .. (.9,0);
      \draw (0,0) .. controls (0,-.5) and (.7,-.5) .. (.9,0);
      \draw (1,-1) .. controls (1,-.5) .. (.9,0);
      \draw[->] (.9,0) .. controls (1,.5) .. (1,1);
      \node at (1.4,0) {\regionlabel{\epsilon_{\lambda \boxminus i}}};
      \node at (0.7,-0.6) {\regionlabel{\epsilon_\lambda}};
    \end{tikzpicture}
    = \sum_{j \in B^+(\lambda)}\
    \begin{tikzpicture}[anchorbase]
     \draw (0,0) .. controls (0,.5) and (.7,.5) .. (.9,0);
     \draw (0,0) .. controls (0,-.5) and (.7,-.5) .. (.9,0);
     \draw (1,-1) .. controls (1,-.5) .. (.9,0);
     \draw[->] (.9,0) .. controls (1,.5) .. (1,1);
     \node at (1.4,0) {\regionlabel{\epsilon_{\lambda \boxminus i}}};
     \node at (0.7,-0.6) {\regionlabel{\epsilon_\lambda}};
     \node at (0.5,0) {\regionlabel{\epsilon_{\lambda \boxplus j}}};
    \end{tikzpicture}
    \stackrel{\eqref{eq:crossingRemoval}}{=} \sum_{j \in B^+(\lambda)} \frac{1}{j-i}\
    \begin{tikzpicture}[anchorbase]
      \draw[->] (0.4,0) arc(0:360:0.4);
      \draw[->] (0.7,-1) -- (0.7,1);
      \node at (0,0) {\regionlabel{\epsilon_{\lambda \boxplus j}}};
      \node at (1.2,0) {\regionlabel{\epsilon_{\lambda \boxminus i}}};
      \node at (0.4,-0.7) {\regionlabel{\epsilon_\lambda}};
    \end{tikzpicture}
    \\
    \stackrel{\eqref{eq:epsilon_lambda-counter-clockwise}}{=}
    \frac{1}{(n+1)d_\lambda} \sum_{j \in B^+(\lambda)} \frac{d_{\lambda \boxplus j}}{j-i}\
    \begin{tikzpicture}[anchorbase]
      \draw[->] (0,-1) -- (0,1);
      \node at (0.4,0) {\regionlabel{\epsilon_{\lambda \boxminus i}}};
      \node at (-0.3,0) {\regionlabel{\epsilon_\lambda}};
    \end{tikzpicture}\ .
  \end{multline*}
  Since the final diagram above is nonzero (it is sent to $i$-induction from $\cM_{\lambda \boxminus i}$ to $\cM_\lambda$ under the 2-functor $\bF_\sH$), relation \eqref{eq:application-identity1} follows.

  Now suppose $i \in B^+(\lambda)$.  Then
  \begin{multline*}
    i \left(
    \begin{tikzpicture}[anchorbase]
      \draw[->] (0,-1) -- (0,1);
      \node at (0.3,0) {\regionlabel{\epsilon_\lambda}};
      \node at (-0.4,0) {\regionlabel{\epsilon_{\lambda \boxplus i}}};
    \end{tikzpicture}
    \right)
    \stackrel{\eqref{eq:right-curl-removal}}{=}\
    \begin{tikzpicture}[anchorbase]
      \draw (0,0) .. controls (0,.5) and (-.7,.5) .. (-.9,0);
      \draw (0,0) .. controls (0,-.5) and (-.7,-.5) .. (-.9,0);
      \draw (-1,-1) .. controls (-1,-.5) .. (-.9,0);
      \draw[->] (-.9,0) .. controls (-1,.5) .. (-1,1);
      \node at (-1.4,0) {\regionlabel{\epsilon_{\lambda \boxplus i}}};
      \node at (-0.7,-0.6) {\regionlabel{\epsilon_\lambda}};
    \end{tikzpicture}
    = \sum_{j \in B^-(\lambda)}
    \begin{tikzpicture}[anchorbase]
      \draw (0,0) .. controls (0,.5) and (-.7,.5) .. (-.9,0);
      \draw (0,0) .. controls (0,-.5) and (-.7,-.5) .. (-.9,0);
      \draw (-1,-1) .. controls (-1,-.5) .. (-.9,0);
      \draw[->] (-.9,0) .. controls (-1,.5) .. (-1,1);
      \node at (-1.4,0) {\regionlabel{\epsilon_{\lambda \boxplus i}}};
      \node at (-0.7,-0.6) {\regionlabel{\epsilon_\lambda}};
      \node at (-0.45,0) {\regionlabel{\epsilon_{\lambda \boxminus j}}};
    \end{tikzpicture}
    \stackrel{\eqref{eq:crossingRemoval}}{=}
    \sum_{j \in B^-(\lambda)} \frac{1}{i-j}\
    \begin{tikzpicture}[anchorbase]
      \draw[<-] (0.4,0) arc(0:360:0.4);
      \draw[->] (-0.7,-1) -- (-0.7,1);
      \node at (0,0) {\regionlabel{\epsilon_{\lambda \boxminus j}}};
      \node at (-0.3,-0.7) {\regionlabel{\epsilon_\lambda}};
      \node at (-1.1,0) {\regionlabel{\epsilon_{\lambda \boxplus i}}};
    \end{tikzpicture}
    \\
    \stackrel{\eqref{eq:epsilon_lambda-clockwise}}{=}
    \sum_{j \in B^-(\lambda)} \frac{|\lambda| d_{\lambda \boxminus j}}{d_\lambda (i-j)}\
    \begin{tikzpicture}[anchorbase]
      \draw[->] (0,-1) -- (0,1);
      \node at (0.3,0) {\regionlabel{\epsilon_\lambda}};
      \node at (-0.4,0) {\regionlabel{\epsilon_{\lambda \boxplus i}}};
    \end{tikzpicture}\ .
  \end{multline*}
  Since the final diagram above is nonzero (as above), relation \eqref{eq:application-identity2} follows. \qedhere
\end{proof}

It is possible to prove the identities \eqref{eq:application-identity1} and \eqref{eq:application-identity2} algebraically, using a careful analysis of the representation theory of the symmetric group.  However, such a proof is considerably longer than the above diagrammatic one.  To the best of the authors' knowledge, these identities have not appeared previously in the literature.  It would be interesting to find purely combinatorial proofs.

%------------------------------
\subsection{Further directions}
%------------------------------

The results of the current paper suggest a number of future research directions.  We briefly describe of few of these here.

\subsubsection{Symmetric groups in positive characteristic}

In light of Proposition~\ref{prop:sU-sM-equivalence}, the 2-category $\sA$ can be viewed as a graphical calculus describing the functors of $i$-induction and $i$-restriction, together with the natural transformations between them.  Throughout this paper, we have worked over the field $\Q$.  It would be natural to instead consider the representation theory of the symmetric group in characteristic $p>0$.  We believe that most of the results presented here have positive characteristic analogues that would yield a relationship between categorified quantum $\widehat{\fsl}_p$ (instead of $\fsl_\infty$) and Heisenberg categories that categorifies the principal embedding.  We refer the reader to the survey \cite{Kle14} for an overview of the modular representation theory of the symmetric group in the context of categorification.

\subsubsection{Cyclotomic Hecke algebras}

Group algebras of symmetric groups are isomorphic to level one degenerate cyclotomic Hecke algebras.  It is natural to expect that the results of the current paper can be extended to higher level degenerate cyclotomic Hecke algebras.   On the Heisenberg side, this corresponds to the higher level Heisenberg categories defined in \cite{MS17}, which involve planar diagrams decorated by dots corresponding to the polynomial generators of the degenerate cyclotomic Hecke algebras.  On the categorified quantum group side, this should correspond to modifying the definition of $\sU_0$ (see Section~\ref{subsec:truncated-KL}) so as not to kill all dots.  This would be related to the results of \cite{BK09,BK09b}.

\subsubsection{More general Heisenberg categories}

The Heisenberg category considered here is a special case of a much more general construction, described in \cite{RS15}, that associates a Heisenberg category (or 2-category) to any graded Frobenius superalgebra.  (The Heisenberg 2-category considered in the current paper corresponds to the case where this Frobenius algebra is simply the base field.)  It would be interesting to generalize the results of the current paper to these more general Heisenberg categories.  Representation theoretically, this amounts to replacing the group algebra of the symmetric group by wreath product algebras associated to the Frobenius algebra in question.  Of special interest would be the case where the Frobenius algebra is the zigzag algebra associated to a finite-type Dynkin diagram, in which case the corresponding Heisenberg categories are the ones considered in \cite{CL12}.

A $q$-deformation of Khovanov's category was also defined in \cite{LS13}.  This deformation corresponds to replacing group algebras of symmetric groups by Hecke algebras of type $A$.  One could form a truncated $q$-deformed Heisenberg 2-category and attempt to relate such a truncation to $q$-deformations of categorified quantum groups.

\subsubsection{Trace decategorification}

In contrast to passing to the Grothendieck group, there is another natural method of decategorification: taking the trace or zeroth Hochschild homology.  The trace of Khovanov's Heisenberg category has been related to W-algebras in \cite{CLLS}. On the other hand, traces of categorified quantum groups have been related to current algebras in \cite{BHLW14,SVV14}.  It would be interesting to investigate the relationship between these two trace decategorifications implied by the results of the current paper and their generalizations mentioned above.

\subsubsection{Geometry}

Heisenberg categories are closely related to the geometry of the Hilbert scheme (see \cite{CL12}).  Similarly, the geometry of quiver varieties \cite{Nak98} can be used to build categorifications of quantum group representations (see for example \cite{VV11,Zhe14,CKL13,Web12}).  It is thus natural to expect that the results of the current paper are related to geometric constructions relating these spaces, such as \cite{Sav06,Nag09,LS10,Lem16}.

%%%%%%%%%%%%%%%%%%%%%%%%%%%%%%%%%%%%%%%%%%%%%%%%%%%%%%%%%%%%%%%%%%%
% References
%%%%%%%%%%%%%%%%%%%%%%%%%%%%%%%%%%%%%%%%%%%%%%%%%%%%%%%%%%%%%%%%%%%

\bibliographystyle{alpha}
\bibliography{Queffelec-Savage-Yacobi}

\end{document}